\documentclass{amsart}
\usepackage{graphicx} 
\usepackage{amsfonts}
\usepackage{amsthm}
\usepackage{amsmath}
\usepackage{mathtools}
\usepackage{amssymb}
\usepackage{enumitem}
\usepackage{booktabs}
\setitemize{noitemsep,topsep=1pt,parsep=0pt,partopsep=1pt}
\setenumerate{noitemsep,topsep=1pt,parsep=0pt,partopsep=1pt}
\usepackage{ dsfont }
\usepackage[table]{xcolor} 

\usepackage{tikz}
\usepackage{pgfplots}
\pgfplotsset{width=10cm,compat=1.9}

\usepackage{biblatex}
\theoremstyle{plain}
\newtheorem{theorem}{Theorem}[section]
\newtheorem{lemma}[theorem]{Lemma}
\newtheorem{proposition}[theorem]{Proposition}

\theoremstyle{definition}
\newtheorem{definition}[theorem]{Definition}
\newtheorem{example}[theorem]{Example}
\newtheorem{remark}[theorem]{Remark}
\newtheorem{corollary}{Corollary}[theorem]

\title{Refined Elementary Capacities from Symplectic Field Theory}
\author{Jonathan Michala}
\date{\today}

\begin{document}

\maketitle

\begin{abstract}
    We extend the family of capacities given by McDuff and Siegel in \cite{McDuff-Siegel_unperturbed_curves} by including a constraint $\ell$ on the number of positive asymptotically cylindrical ends of curves showing up in the definition.
    We prove a generalized computation formula for four-dimensional convex toric domains that offers new, sometimes sharp, embedding obstructions in stabilized and unstabilized cases.
    The formula restricts to the McDuff-Siegel capacities for $\ell = \infty$ and to the Gutt-Hutchings capacities for $\ell = 1$.
    To verify the formula, we must prove the existence of certain curves in the convex toric domain $X_\Omega$, and this requires a new method of proof compared to \cite{McDuff-Siegel_unperturbed_curves}.
    We neck-stretch along $\partial X_\Omega$ with curves known to exist in a well-chosen ellipsoid containing $X_\Omega$, and we obtain the desired curves in the bottom level of the resulting psuedoholomorphic building.
\end{abstract}

\section{Introduction}

\subsection{Background}

To get a better understanding of symplectic structure, one often studies embedding problems - for two given symplectic manifolds $X$ and $Y$, does there exist a smooth embedding $X \hookrightarrow Y$ that preserves the symplectic form?
We denote symplectic embeddings by $X \xhookrightarrow{s} Y$.
Such problems have a rich history, starting in 1985 with Gromov's nonsqueezing theorem \cite{gromov1985pseudo}.

We follow Schlenk's survey \cite{Schlenk_symp_embedding_probs_old_new} to write a summary of symplectic embedding results in Tables \ref{tab:related_work} and \ref{tab:related_work_stabilized}, where $X_{\Omega_\smile}$ and $X_{\Omega_\frown}$ denote concave and convex toric domains; where $E,B$, and $P$ denote ellipsoids, balls, and polydisks; and where $\tau$ denotes the golden ratio.
Table \ref{tab:related_work_stabilized} focuses on stabilized embedding problems, where we take the product of a domain with copies of $\mathbb{C}$.
This is a tractable way to find answers to embedding questions in dimension greater than 4, which is difficult in the general setting, and was pioneered by Hind and Kerman in \cite{hind-Kerman2014new}.
When we say that an answer is \textbf{stable}, we mean that the known sharp result in 4 dimensions stabilizes to higher dimensions.

\begin{table}[]
    \centering
    \begin{tabular}{c|cc}
        \toprule
        Domain and Target & Answer \\
        \midrule
        $E(a,1) \xhookrightarrow{s} B^4(A)$ & Fibonacci Staircase & \cite{McDuff_Schlenk_fibonacci}\\
        \\
        $E(a,1) \xhookrightarrow{s} P(A,A)$ & Pell Staircase & \cite{Frenkel_Muller_Pell_stairs}\\
        \\
        $X_{\Omega_\smile} \xhookrightarrow{s} X_{\Omega_\frown}$ & ECH capacities are sharp & \cite{cristofaro2019symplectic}\\
        \\
        $P(a,1) \xhookrightarrow{s} B^4(A)$ & sharp results for $a \in [1,\tfrac{5+\sqrt{7}}{3}]$ & \cite{Hutchings_2016_beyond_ech,christianson2018symplectic,hind2015symplectic}\\
        & Lower bound for $a \in [\tfrac{5+\sqrt{7}}{3},8]$ & \cite{Hutchings_2016_beyond_ech}\\
        & Upper bounds for $a \in [2,10]$ & \cite{Schlenk2005}\\
        \\
        $P(a,1) \xhookrightarrow{s} E(bc,c)$ & sharp results for $a \in [1,2], b\in \mathbb{Z}_{\geq 1}$ & \cite{Hutchings_2016_beyond_ech}\\
        & extension to half integers & \cite{digiosia2022symplectic}\\
        
    \end{tabular}
    \caption{Related work on symplectic embedding problems}
    \label{tab:related_work}
\end{table}

\begin{table}[]
    \centering
    \begin{tabular}{c|cc}
        \toprule
        Domain and Target & Answer \\
        \midrule
        $E(a,1) \times \mathbb{C}^n \xhookrightarrow{s} B^4(A) \times \mathbb{C}^n$ & Fibonacci Staircase on $[1,\tau^4]$ & \cite{cristofaro2018symplectic}\\
        & stable on a sequence $\searrow \tau^4$ & \cite{cristofaro2018ghost}\\
        & stable for all $3k - 1$ & \cite{Mcduff2017ARO}\\
        & now completely answered & \cite{mcduff2024sesquicuspidal}\\
        \\
        $E(a,1) \times \mathbb{C}^n \xhookrightarrow{s} \mu E(b,1) \times \mathbb{C}^n$ & sharp results for $a \geq b + 1 \geq 3$\\
        & $a,b$ integers of opposite parity & \cite{McDuff-Siegel_unperturbed_curves, cristofaro2022higher}\\
        \\
        $E(a,1) \times \mathbb{C}^n \xhookrightarrow{s} \mu P(b,1) \times \mathbb{C}^n$ & sharp results for $a \geq 2b - 1$\\
        & $a$ odd integer & \cite{McDuff-Siegel_unperturbed_curves,cristofaro2022higher}\\
    \end{tabular}
    \caption{Related work on stabilized embedding problems}
    \label{tab:related_work_stabilized}
\end{table}

\subsection{Main Results}

In Section \ref{sec:definition_properties}, we define the family of capacities $\Tilde{\mathfrak g}_k^\ell(X)$ for Liouville domains with nondegenerate contact boundary (which include convex toric domains), for $k,\ell$ positive integers and possibly $\ell = \infty$.
Roughly, for a given almost complex structure $J$, we take the infimal action of a Reeb orbit set $\Gamma$ with appropriate index and number of positive ends such that the moduli space of $J$-holomorphic curves asymptotic to $\Gamma$ is nonempty.
Then we take the supremum of this value over admissible almost complex structures.
We'll see that the definition is only a slight change from the definition of McDuff-Siegel capacities $\Tilde{\mathfrak g}_k^\infty$ and results in a seemingly minor change in the computation on 4-dimensional convex toric domains.
However, this slight change results in new, sometimes sharp obstructions, stronger than those given previously by $\Tilde{\mathfrak g}_k^1$ (agreeing with Gutt-Hutchings capacities) and $\Tilde{\mathfrak g}_k^\infty$, and proving the formula requires the construction of curves unwitnessed by McDuff-Siegel capacities, which is the focus of much of this paper.

To state the computation formula on four-dimensional convex toric domains, we will need the following ``dual norm'' and admissibility criterion.
For $\Omega \subset \mathbb{R}^n$, compact and convex, define $|| \cdot ||_\Omega^* \colon \mathbb{R}^n \rightarrow \mathbb{R}$ by $||v||_\Omega^* = \max \{\langle v,w \rangle \vert w \in \Omega\}$, where $\langle \cdot, \cdot \rangle$ is the standard inner product.
Note that any maximizer $w$ lies on the boundary of $\Omega$, with all of $\Omega$ on one side of the hyperplane through $w$, normal to $v$.
We use the notation $||\cdot||_\Omega^*$ because when the origin is included in the interior of $\Omega$, this is dual to the norm where $\Omega$ is the unit ball.

For a multi-set (i.e. repeats allowed) $P = \{(i_s,j_s)\}_{s=1}^q$ of ordered pairs in $\mathbb{Z}_{\geq 0}^2 \setminus {(0,0)}$, positive integer $k$, and $\ell \in \mathbb{Z}_{> 0} \cup \{\infty\}$, we say that the collection is \textbf{$(k,\ell)$-admissible}, with $P \in \mathcal{P}_{k,\ell}$, if it satisfies the following properties:
\begin{itemize}
    \item $\frac{1}{2}$\textit{index:} $\sum_{s=1}^q(i_s+j_s) + q - 1 = k$.
    \item \textit{weak permissibility:} if $q \geq 2$, then $(i_1,...,i_q) \neq (0,...,0)$ and $(j_1,...,j_q) \neq (0,...,0)$.
    \item \textit{positive ends:} $q \leq \ell$.
\end{itemize}
Note that when $\ell = \infty$, the positive ends criterion is always satisfied, and so these criteria will match those of the McDuff-Siegel capacity computations.
We can now state the main theorem, and we emphasize that the primary contribution of this paper is the proof and application of property \ref{thm:properties_subitem_computation}.
Note that a \textbf{generalized Liouville embedding} is a smooth embedding $i \colon (X,\lambda) \hookrightarrow (X',\lambda')$ between equidimensional Liouville domains such that the closed 1-form $(i^*(\lambda') - \lambda)|_{\partial X}$ is exact.

\begin{theorem}\label{thm:properties}
For positive integer $k$ and for $\ell \in \mathbb{Z}_{> 0} \cup \{\infty\}$, $\Tilde{\mathfrak g}_k^\ell$ satisfies the following properties:
\begin{itemize}
    \item \textbf{Liouville Domains:} Let $X$ and $X'$ be equidimensional Liouville domains.
    Then,
    \begin{enumerate}[label=(\alph*)]
        \item \textit{scaling:} $\Tilde{\mathfrak g}_k^\ell(X,\mu \lambda) = \mu \Tilde{\mathfrak g}_k^\ell(X,\lambda)$ for $\mu \in \mathbb{R}_{\geq 0}$.
        \item \textit{nondecreasing in $k$:} $\Tilde{\mathfrak g}_1^\ell(X) \leq \Tilde{\mathfrak g}_2^\ell(X) \leq \Tilde{\mathfrak g}_3^\ell(X) \leq \cdots$.
        \item \textit{nonincreasing in $\ell$:} $\Tilde{\mathfrak g}_k^1(X) \geq \Tilde{\mathfrak g}_k^2(X) \geq \Tilde{\mathfrak g}_k^3(X) \geq \cdots$.
        \item \textit{monotonicity:} If there is a generalized Liouville embedding from $X$ into $X'$, then $\Tilde{\mathfrak g}_k^\ell(X) \leq \Tilde{\mathfrak g}_k^\ell(X')$.
    \end{enumerate}
    \item \textbf{4-D Convex Toric Domains:} Let $X_\Omega$ be a four-dimensional convex toric domain.
    Then,
    \begin{enumerate}[resume*]
        \item \textit{stabilization:} $\Tilde{\mathfrak g}_k^\ell(X_\Omega \times B^2(c)) = \Tilde{\mathfrak g}_k^\ell(X_\Omega)$ for any $c \geq \Tilde{\mathfrak g}_k^\ell(X_\Omega)$.
        \item\label{thm:properties_subitem_computation} \textit{computation:} $$\Tilde{\mathfrak g}_k^\ell(X_\Omega) = \min_{\{(i_s,j_s)\}_{s=1}^q \in \mathcal{P}_{k,\ell}} \sum_{s=1}^q ||(i_s,j_s)||_\Omega^*.$$
    \end{enumerate}
    
\end{itemize}
\end{theorem}

\begin{remark}
    Note that $\Tilde{\mathfrak g}_k^\ell$ recovers the Gutt-Hutchings capacity \cite{Gutt-Hutchings} and the McDuff-Siegel capacity \cite{McDuff-Siegel_unperturbed_curves} when computed on four-dimensional convex toric domains:
    $$c_k^{GH}(X_\Omega) = \Tilde{\mathfrak g}_k^1(X_\Omega), \quad\quad \Tilde{\mathfrak g}_k(X_\Omega) = \Tilde{\mathfrak g}_k^\infty(X_\Omega).$$
\end{remark}

\begin{example}
    Before we state the general computation results, we'll start with a specific example where the capacities defined in this paper give sharp obstructions, previously unseen by the McDuff-Siegel or Gutt-Hutchings capacities.
    
    We ask if $P(a,1) \xhookrightarrow{s} X_\Omega$, where $a$ is any positive value, and $\Omega \subset \mathbb{R}^2$ is the convex hull of the points $(0,0)$, $(0,1)$, $(1,1)$, $(2,1/2)$, and $(2,0)$.
    See Figure \ref{fig:intro_sharp_obstruction} for a moment map visual.
    For $a \leq 1$, the embedding exists by inclusion, and for $a>7/4$, the embedding is obstructed by volume.
    For $a = 1.1$, we make the following capacity computations (See Example \ref{ex:sharp_obstruction} for full details).
    The bold red values in the table show that this embedding is obstructed by capacities $\Tilde{\mathfrak g}_k^2$ for $k \geq 5$, yet it is unobstructed by any $\Tilde{\mathfrak g}_k^1$ or $\Tilde{\mathfrak g}_k^\infty$.

    \vspace{5pt}
    \begin{center}
    \renewcommand{\arraystretch}{1.5} 
    \begin{tabular}{|c|c|c|c|c|c|c|c|c|c|}
        \hline
        \rowcolor{blue!20} $k$ & 1 & 2 & 3 & 4 & 5 & 6 & 7 & 8 & 9\\
        \hline
        \hline
        $\Tilde{\mathfrak g}_k^1(P(1.1,1))$ & 1 & 2 & 3 & 4 & 5 & 6 & 7 & 8 & 9 \\
        $\Tilde{\mathfrak g}_k^1(X_\Omega)$ & 1 & 2 & 3 & 4 & 5 & 6 & 7 & 8 & 9\\
        \hline
        $\Tilde{\mathfrak g}_k^2(P(1.1,1))$ & 1 & 2 & 2.1 & 3.1 & \cellcolor{red!20}\textbf{4.1} & \cellcolor{red!20}\textbf{5.1} & \cellcolor{red!20}\textbf{6.1} & \cellcolor{red!20}\textbf{7.1} & \cellcolor{red!20}\textbf{8.1}\\
        $\Tilde{\mathfrak g}_k^2(X_\Omega)$ & 1 & 2 & 3 & 3.5 & \cellcolor{red!20}\textbf{4} & \cellcolor{red!20}\textbf{5} & \cellcolor{red!20}\textbf{6} & \cellcolor{red!20}\textbf{7} & \cellcolor{red!20}\textbf{8}\\
        \hline
        $\Tilde{\mathfrak g}_k^\infty(P(1.1,1))$ & 1 & 2 & 2.1 & 3.1 & 3.1 & 4.1 & 4.1 & 5.1 & 5.1\\
        $\Tilde{\mathfrak g}_k^\infty(X_\Omega)$ & 1 & 2 & 3 & 3.5 & 4 & 4.5 & 5 & 5.5 & 6\\
        \hline
    \end{tabular}
    \end{center}
    \vspace{5pt}
\end{example}

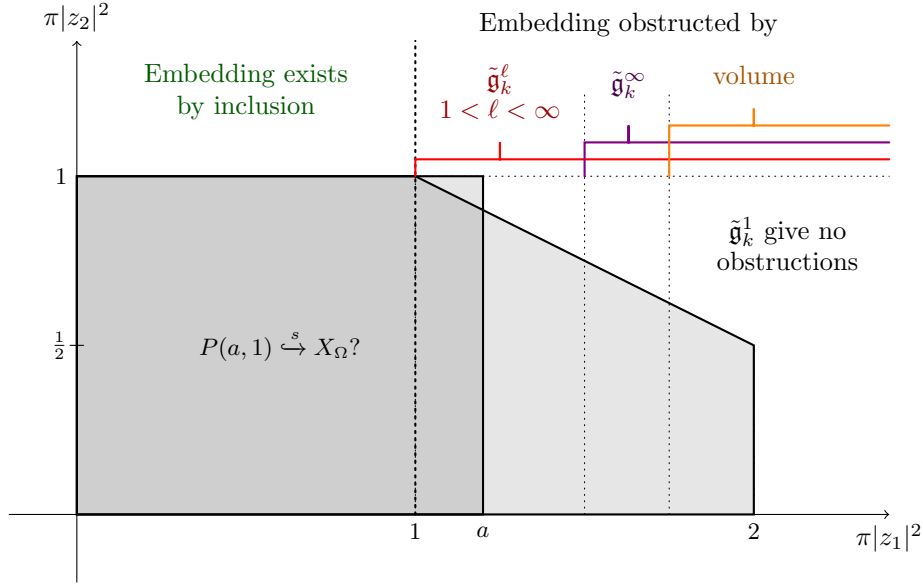
\begin{figure}
    \centering
    \begin{tikzpicture}[scale=4.5,cap=round]
    
        \draw[->] (-0.2,0) -- (2.4,0) node[below] {$\pi|z_1|^2$};
        \draw[->] (0,-0.2) -- (0,1.4) node[above] {$\pi|z_2|^2$};
    
        \draw[dotted] (1,0) -- (1,1);
        \node[below] at (1,0) {\small $1$};
    
        \draw[dotted] (0,1) -- (2.4,1);
        \node[left] at (0,1) {\small $1$};
    
        \draw[dotted] (2,0) -- (2,0.5);
        \node[below] at (2,0) {\small $2$};
        \node[left] at (0,0.5) {\small $\frac{1}{2}$};
        \draw (-.02,0.5) -- (.02,0.5);

        \draw[thick, dotted] (1,0) -- (1,1.4);
        \draw[dotted] (1.5,0) -- (1.5,1.25);
        \draw[dotted] (1.75,0) -- (1.75,1.25);

        \draw[thick] 
            (0,0) -- (0,1) -- (1,1) -- (2,0.5) -- (2,0) -- cycle;
        \fill[opacity=0.1] 
            (0,0) -- (0,1) -- (1,1) -- (2,0.5) -- (2,0) -- cycle;

        \draw[thick] 
            (0,0) -- (0,1) -- (1.2,1) -- (1.2,0) -- cycle;
        \node at (0.6,0.5) {\small $P(a,1) \xhookrightarrow{s} X_\Omega$?};
        \node[below] at (1.2,-0.01) {\small $a$};
        \fill[opacity=0.1] 
            (0,0) -- (0,1) -- (1.2,1) -- (1.2,0) -- cycle;

        \draw[thick, red]
            (1,1) -- (1,1.05) -- (1.25,1.05) -- (1.25,1.1) -- (1.25,1.05) -- (2.4,1.05);
        \draw[thick, violet]
            (1.5,1) -- (1.5,1.1) -- (1.63,1.1) -- (1.63,1.15) -- (1.63,1.1) -- (2.4,1.1);
        \draw[thick, orange]
            (1.75,1) -- (1.75,1.15) -- (2,1.15) -- (2,1.2) -- (2,1.15) -- (2.4,1.15);

        \node[text=green!35!black, align=center] at (0.5,1.25) {Embedding exists\\by inclusion};
        \node[text=red!60!black, align=center] at (1.25,1.25) {$\Tilde{\mathfrak g}_k^\ell$\\$1 < \ell < \infty$};
        \node[text=violet!55!black, align=center] at (1.63,1.27) {$\Tilde{\mathfrak g}_k^\infty$};
        \node[text=orange!65!black, align=center] at (2,1.3) {volume};
        \node[align=center] at (1.63,1.45) {Embedding obstructed by};
        \node[align=center] at (2.1,0.8) {$\Tilde{\mathfrak g}_k^1$ give no\\obstructions};

    \end{tikzpicture}
    \caption{$\Tilde{\mathfrak g}_k^\ell$ provide sharp obstructions when $1 < \ell < \infty$, showing that $P(a,1)$ does not embed into $X_\Omega$ for any $a > 1$. This is an improvement on the McDuff-Siegel capacities $\Tilde{\mathfrak g}_k^\infty$, which only give obstructions for $a > 3/2$, and on the Gutt-Hutchings capacities $\Tilde{\mathfrak g}_k^1$, which do not give any obstructions for this embedding.}
    \label{fig:intro_sharp_obstruction}
\end{figure}

Now we state the main computation results of Section \ref{sec:computations}, which provide a closed computation formula for a large class of domains, as well as novel sharp obstructions between certain manifolds in this class.

\begin{proposition}[The ball, Proposition \ref{prop:ball}]
    The capacities of the ball are given by
    $$\Tilde{\mathfrak g}_k^\ell(B^4(1)) = \begin{cases}
        \ell - 1 + \lceil\tfrac{k-3(\ell - 1)}{2}\rceil & k > 3(\ell - 1)\\
        1 + i & k = 3i+1 \leq 3(\ell - 1)\\
        1 + i & k = 3i+2 \leq 3(\ell - 1)\\
        2 + i & k = 3i+3 \leq 3(\ell - 1)\\
    \end{cases}$$
\end{proposition}

\begin{proposition}[Polydisks, Proposition \ref{prop:polydisk}]\label{prop:intro-polydisk}
    For positive integers $k,\ell$, we compute $\Tilde{\mathfrak g}_k^\ell$ for the Polydisk $P(a,1)$ with $a > 1$:
    $$\Tilde{\mathfrak g}_k^\ell(P(a,1)) = 
    \begin{cases}
        k, & \ell = 1 \text{ or } k = 1,2\\
        \min\{k,k+a-\ell\}, & \ell \geq 2, \; k \geq 3, \; \frac{k-1}{2} \geq \ell - 1\\
        \min\{k,\lceil\frac{k-1}{2}\rceil + a\}, & \ell \geq 2,\; k \geq 3,\; l-1 \geq \frac{k-1}{2}.
    \end{cases}$$
\end{proposition}

\begin{theorem}[Theorem \ref{thm:comp_for_XOmega_unit_square_in_Omega}]
    Let $X_\Omega$ be a 4-dimensional convex toric domain such that the boundary $\partial \Omega$ intersects the $y$-axis at $(0,1)$.
    Assume $\Omega$ contains the unit square $[0,1]\times[0,1]$.
    Then for any positive integers $k$ and $\ell$, and $J$ given by Lemma \ref{lem:define_J_transition_where_trading_away_0,1_viable}, the capacities $\Tilde{\mathfrak g}_k^\ell(X_\Omega)$ are given by
    $$\Tilde{\mathfrak g}_k^\ell =  \begin{cases}
        \min\{k, \ell - 1 + ||(1,k - 2\ell + 1)||_\Omega^*\} & k - 2\ell + 1 \geq J\\
        k & k - 2\ell + 1 < J,\; k - J < 1\\
        \min\{k, \tfrac{k-J-1}{2} + ||(1,J)||_\Omega^*\} & k - 2\ell + 1 < J,\;  k - J > 0, \text{ odd}\\
        \min\{k, \tfrac{k-J-2}{2} + ||(1,J+1)||_\Omega^*\} & k - 2\ell + 1 < J,\; k - J > 0, \text{ even}
    \end{cases}$$
\end{theorem}

\begin{theorem}[Theorem \ref{thm:sharp_obstructions}]
    Let $X_\Omega$ be a convex toric domain such that the polytope boundary $\partial \Omega$ contains the vertices $(0,1)$, $(a,1)$, $(b,\tfrac{1}{2}+\epsilon)$, and $(b,0)$, for $a > 1$ and $b > a + 1/2$.
    Then the embedding $P(a',1) \hookrightarrow X_\Omega$ is obstructed for $a' > a$ by $\Tilde{\mathfrak g}_k^\ell$ when $\lceil a\rceil < \ell < \infty$, but not when $\ell = \infty$ or $\ell < a$.
    This obstruction is sharp since the inclusion map works for $a' \leq a$, and this obstruction stabilizes to higher dimensions.
\end{theorem}

\subsection{Motivation and Remarks}

For further motivation of this specific project, we cite the inspiring question by Helmut Hofer, mentioned in Schlenk's extended survey \cite{Schlenk2018Survey-old-and-new-long}:

\vspace{5pt}
\noindent\textit{Is there a general theory of embedding invariants, built from closed characteristics and $J$-curves between them, that specializes to SH capacities and ECH capacities by selecting two particular subclasses of orbit sets and $J$-curves?}
\vspace{5pt}

\noindent This paper provides a partial answer to this question with a generalization of elementary McDuff-Siegel and Gutt-Hutchings capacities, which would fall under the Symplectic Homology (SH) capacities in question, but this paper doesn't consider ECH capacities.
One could hope to answer this question more fully in future work by defining $\Tilde{\mathfrak g}_k^{\ell,h}$, where $h$ is a further restriction on genus, specializing to all three families in the following way, at least on 4-dimensional convex toric domains:
$$\Tilde{\mathfrak g}_k^{1,0} = c_k^{GH},\quad\quad \Tilde{\mathfrak g}_k^{\infty,0} = c_k^{MS}, \quad\quad \Tilde{\mathfrak g}_k^{\infty,\infty} = c_k^{ECH}.$$
For stabilization and higher-dimensional considerations, ECH capacities do not extend beyond 4 dimensions, but \textit{elementary} ECH capacities can be defined in higher dimensions, albeit behaving differently than the 4-dimensional capacities.

\begin{remark}[Full SFT package and relation to Gutt-Hutchings capacities]
    In \cite{siegel2019higher}, Siegel uses Symplectic Field Theory to create capacities $\mathfrak g_k$ via the chain level filtered $\mathcal{L}_\infty$ structure on linearized contact homology and tangency constraints.
    But this definition requires virtual perturbations of the moduli spaces, and in \cite{McDuff-Siegel_unperturbed_curves}, McDuff and Siegel use an ersatz, elementary definition $\Tilde{\mathfrak g}_k$ involving honest psuedoholomorphic curves to avoid any virtual perturbations.
    Under mild assumptions ($\pi_1(X) = 0, 2c_1(TX) = 0$) the full SFT capacities $\mathfrak g_k$ agree with the Gutt-Hutchings capacities \cite[Theorem 4.36]{pereira2022lagrangian}:
    $$\mathfrak g_k^1(X) = c_k^{GH}(X).$$
\end{remark}

\subsection{Acknowledgments}
The author thanks his thesis advisor Kyler Siegel for his extensive help on the paper, as well as Richard Hind, Dan Cristofaro-Gardiner, Yuan Yao, Morgan Weiler, and Luya Wang for their mentorship and helpful conversations.

\subsection{Paper Outline}
In Section \ref{sec:definition_properties}, we define the capacity and prove most of its properties, following \cite{McDuff-Siegel_unperturbed_curves} closely.
In Section \ref{sec:lower_bound}, we prove a lower bound for the computation formula on four-dimensional convex toric domains, again following \cite{McDuff-Siegel_unperturbed_curves}.
We prove the upper bound over Sections \ref{sec:upper_bound}, \ref{sec:formal_perturbation_invariance}, and \ref{sec:existence_of_curves}, where Section \ref{sec:upper_bound} gives the general argument, Section \ref{sec:formal_perturbation_invariance} proves the formal perturbation invariance of desired curves, and Section \ref{sec:existence_of_curves} proves their existence.
In Section \ref{sec:computations}, we use the formula to compute the entire family of capacities for balls and a certain class of domains between the cube and the cylinder, giving sometimes sharp obstructions to embeddings in unstabilized and stabilized cases.
We conclude with suggestions for future work in Section \ref{sec:future_work}.

\section{Definition and Properties}\label{sec:definition_properties}
Let $X$ be a compact manifold with boundary, equipped with a one form $\lambda$, whose restriction to the boundary $\alpha := \lambda|_{\partial X}$ is a contact form, and whose exterior derivative $\omega := d\lambda$ is a symplectic form on $X$. 
A Liouville vector field $V$ is defined everywhere on $X$ by $\iota_V\omega = \lambda$, which is transverse at the boundaries.
The negative boundary $\partial^- X$ is where $V$ points inward, and the positive boundary $\partial^+ X$ is where $V$ points outward.
This pair $(X,\lambda)$ is a \textbf{Liouville cobordism} from $\partial^- X$ to $\partial^+ X$, and if $\partial^- X = \emptyset$, then $(X,\lambda)$ is called a \textbf{Liouville domain}.
The \textbf{Reeb vector field} $R_\alpha$ is defined along the boundary by $\iota_{R_\alpha}d\alpha = 0$ and $R(\alpha) = 1$, and the \textbf{action} of a Reeb orbit $\mathcal{A}_{\partial X}\gamma$ is the period, $\int_\gamma \alpha$, assuming that $\gamma$ is parametrized so that it's velocity is always equal to $R_\alpha$.
We say a Liouville cobordism has a \textbf{nondegenerate contact boundary} if all of the Reeb orbits on the boundary are nondegenerate.

If $(Y,\alpha)$ is a contact manifold, then the \textbf{symplectization} is $\mathbb{R} \times Y$ with symplectic form $d(e^r\alpha)$.
We denote by $\mathcal{J}(Y)$ the set of \textbf{admissible almost complex structures}, which are compatible with $\omega$ and $r$-translation invariant, which send $\partial_r$ to $R_\alpha$, and which restrict to an almost complex structure on every contact hyperplane.
For a Liouville cobordism $X$, the \textbf{symplectic completion} is
$$\widehat{X}:= (\mathbb{R}_{\geq 0} \times \partial^+X) \cup X \cup (\mathbb{R}_{\leq 0} \times \partial^-X),$$
overlapping at the $0$-slices, with symplectic form extending $\omega$ into the half-symplectizations and restricting to the corresponding symplectization forms on those ends.
We define the set of \textbf{admissible almost complex structures} $\mathcal{J}(X)$ similarly.
And if $D$ is a local symplectic divisor near a point $p$ in the interior of $X$, we define $\mathcal{J}(X;D)$ to be the subset of almost complex structures which are integrable near $p$ and preserve the germ of $D$.

Now, given a Liouville cobordism $X$, divisor $D$ and point $p$ as above, an almost complex structure $J \in \mathcal{J}(X;D)$, collections $\Gamma^+,\Gamma^-$ of Reeb orbits on the boundaries, and an integer $k \geq 1$, we define
$$\overline{\overline{\mathcal{M}}}_X^J(\Gamma^+;\Gamma^-)\ll \mathcal{T}^{(k)}p\gg$$
as in \cite{McDuff-Siegel_unperturbed_curves} to be the compactified moduli space of $J$-holomorphic curves in the symplectic completion of $X$, which are asymptotically cylindrical to $\Gamma^+$ and $\Gamma^-$ and are tangent to $D$ with contact order $k$.
Compactification with a tangency constraint involves cases where $p$ is on a ghost component of a resulting psuedoholomorphic building, and details can be found in \cite[\S 2.2]{McDuff-Siegel_unperturbed_curves}.
For Liouville domains, we use $\Gamma = \Gamma^+$ in the notation and supress $\Gamma^- = \emptyset$.
Figure \ref{fig:moduli} depicts a potential curve in this compactification.

\begin{figure}
    \centering
    \includegraphics[width=0.5\linewidth]{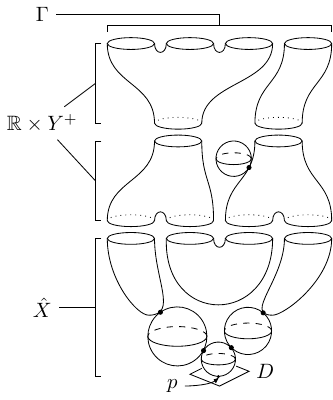}
    \caption{Example of a potential curve in $\overline{\overline{\mathcal{M}}}_X^J(\Gamma)\ll \mathcal{T}^{(k)}p\gg$}
    \label{fig:moduli}
\end{figure}

\begin{definition}
    Let $X$ be a Liouville domain with nondegenerate contact boundary, $D$ a smooth local symplectic divisor passing through a point $p$ in the interior of $X$, $k$ a positive integer, and $\ell \in \mathbb{Z}_{>0} \cup \{\infty\}$.
    We put
    $$\Tilde{\mathfrak{g}}_k^\ell(X) := \sup_{J \in \mathcal{J}(X;D)}\left( \inf_{\Gamma \in \mathcal{G}_{J,k,\ell}} \mathcal{A}_{\partial X}(\Gamma)\right),$$
    where $\mathcal{G}_{J,k,\ell}$ is the set of all tuples $\Gamma = (\gamma_1,...,\gamma_m)$ of Reeb orbits with $m \leq \ell$ such that $\overline{\overline{\mathcal{M}}}_X^J(\Gamma)\ll \mathcal{T}^{(k)}p \gg$ is nonempty.
    Here we put $\mathcal{A}_{\partial X}(\Gamma) := \sum_{i=1}^m \mathcal{A}_{\partial X}(\gamma_i)$.
\end{definition}

\begin{remark}
    Consistent with the notation, $\Tilde{\mathfrak g}_k^\ell$ is independent of the choice of $p \in \text{Int}\; X$ and the local divisor $D$.
    See \cite[Lemma 3.1.3]{McDuff-Siegel_unperturbed_curves}.
    This follows from Moser's trick, giving a symplectomorphism which is the identity near the boundary and creating a bijection between moduli spaces with different divisors.
    Since we fix the boundary, the positive asymptotic behavior of a curve in this bijection (including the number of positive ends and total action) does not change.
\end{remark}

Now we prove all the properties of $\Tilde{\mathfrak g}_k^\ell$ except the computation formula.
These make it a capacity for generalized Liouville embeddings between Liouville domains.
A \textbf{generalized Liouville embedding} is a smooth embedding $i \colon (X,\lambda) \hookrightarrow (X',\lambda')$ between equidimensional Liouville domains such that the closed 1-form $(i^*(\lambda') - \lambda)|_{\partial X}$ is exact.

\subsection{Proofs of Theorem \ref{thm:properties} Properties (a) - (e)}
The following proofs follow from \cite{McDuff-Siegel_unperturbed_curves}, but are included here for completion.

\begin{proof}[Proof of scaling]
    A Reeb orbit $\gamma$ of $X$ has action which scales with the contact form: $\mathcal{A}_{\mu\lambda}(\gamma) = \int_\gamma \mu\lambda = \mu \int_\gamma \lambda = \mu \mathcal{A}_{\lambda}(\gamma)$.
    By definition, $\Tilde{\mathfrak g}_k^\ell(X,\mu\lambda)$ scales in the same way.
\end{proof}

\begin{proof}[Proof of nondecreasing in $k$]
    Let $J$ and $J'$ attain the suprema of $\Tilde{\mathfrak g}_k^\ell(X)$ and $\Tilde{\mathfrak g}_{k+1}^\ell(X)$, respectively.
    Then
    $$\Tilde{\mathfrak g}_k^\ell(X) = \inf_{\Gamma \in \mathcal{G}_{J,k,\ell}} \mathcal{A}(\Gamma) \;\leq\; \inf_{\Gamma \in \mathcal{G}_{J,k+1,\ell}} \mathcal{A}(\Gamma) \;\leq\; \inf_{\Gamma \in \mathcal{G}_{J',k+1,\ell}} \mathcal{A}(\Gamma) = \Tilde{\mathfrak g}_{k+1}^\ell(X).$$
    The first inequality follows because all curves which satisfy $\ll T^{(k+1)}p\gg$ also satisfy $\ll T^{(k)}p\gg$, and the second inequality follows by the supremum in the definition of $\Tilde{\mathfrak g}_{k+1}^\ell(X)$.
\end{proof}

\begin{proof}[Proof of nonincreasing in $\ell$]
    Similarly, if $J$ and $J'$ attain the suprema of $\Tilde{\mathfrak g}_k^\ell(X)$ and $\Tilde{\mathfrak g}_k^{\ell+1}(X)$, respectively, then
    $$\Tilde{\mathfrak g}_k^\ell(X) = \inf_{\Gamma \in \mathcal{G}_{J,k,\ell}} \mathcal{A}(\Gamma) \;\geq\; \inf_{\Gamma \in \mathcal{G}_{J',k,\ell}} \mathcal{A}(\Gamma) \;\geq\; \inf_{\Gamma \in \mathcal{G}_{J',k,\ell+1}} \mathcal{A}(\Gamma) = \Tilde{\mathfrak g}_k^{\ell+1}(X).$$
    Here the first inequality comes from the supremum, and the second inequality follows because all curves with at most $\ell$ positve ends have at most $\ell + 1$ positive ends too.
\end{proof}

\begin{proof}[Proof of monotonicity]
    Let $\iota$ be a generalized Liouville embedding from $X$ into $X'$, both Liouville domains of the same dimesnion.
    Let $p$ be a point in the interior of $X$, with a divisor $D$ near $p$.
    Say $p' := \iota(p)$ and $D' := \iota(D)$.
    Choose an arbitrary admissible almost complex structure for $X$, $J \in \mathcal{J}(X;D)$, and choose any admissible complex structure for $X'$, $J'$, that restricts to $\iota_*J$ on $\iota(X)$.
    Now pick a family $\{J_t'\}_{t \in [0,1)}$ of admissible almost complex structures which realizes neck-stretching along $\partial\iota(X)$.
    Since $\widetilde{g}_k^\ell(X')$ is a supremum over admissible almost complex structures, we see that for all $J_t'$, we have
    $$\inf_{\Gamma \in \mathcal{G}_{J_t',k,\ell}} \mathcal{A}_{\partial X'}(\Gamma) \leq \Tilde{\mathfrak g}_k^\ell(X').$$
    Thus, for all $t \in [0,1)$, there exists a $\Gamma_t \in \mathcal{G}_{J_t',k,\ell}$ such that $\mathcal{A}_{\partial X'}(\Gamma) \leq \Tilde{\mathfrak g}_k^\ell(X')$.
    And since the boundary of $X'$ is nondegenerate, there are only finitely many Reeb orbits with action at most $\Tilde{\mathfrak g}_k^\ell(X')$.
    Hence, with only finitely many plausible orbit sets, there exists an increasing sequence $t_1,t_2,... \in [0,1)$ approaching $1$ such that $\Gamma_{t_i}$ is the same for all $t_i$.
    Then by SFT compactness, we get a psuedoholomorphic building in the compactified moduli space with the same asymptotics, and therefore with combined action at most $\Tilde{\mathfrak g}_k^\ell(X')$.
    Now, we look at this building and focus on the curve component in the $\iota(X)$ level which contains $p'$.
    This curve has asymptotics with action at most $\Tilde{\mathfrak g}_k^\ell(X')$.
    And since $\iota$ is a generalized Liouville embedding, caps cannot exist in the psuedoholomorphic building by Stokes theorem.
    Thus, no genus and no caps implies that the curve component in the bottom level cannot have more positive ends than the top level.
    This implies that the component still has at most $\ell$ positive ends.
    Hence,
    $$\inf_{\Gamma \in \mathcal{G}_{J,k,\ell}} \mathcal{A}_{\partial X} (\Gamma) \leq \Tilde{\mathfrak g}_k^\ell(X'),$$
    and since $J$ was arbitrary, then we conclude that $\Tilde{\mathfrak g}_k^\ell(X) \leq \Tilde{\mathfrak g}_k^\ell(X')$.
\end{proof}

\begin{proof}[Proof of stabilization.]
    This is proved in \cite{McDuff-Siegel_unperturbed_curves} sections 3.6 and 3.7.
    Here we give a summary of those arguments and why they are unaffected by introducing the parameter $\ell$ on positive ends.

    First we prove a lower bound: if $c \geq \Tilde{\mathfrak g}_k^\ell(X)$, then $\Tilde{\mathfrak g}_k^\ell(X \times B^2(c)) \geq \Tilde{\mathfrak g}_k^\ell(X)$.
    Indeed, we take a smoothed subdomain $X \overset{\frown}{\times} B^2(c) \subset X \times B^2(c)$, such that all Reeb orbits with action less than $c - \epsilon$ are entirely contained in $\partial X \times \{0\}$.
    Let $J_X$ be the almost complex structure on $X$ which achieves the supremum in $\Tilde{\mathfrak g}_k^\ell(X)$, and consider an admissible almost complex structure $J$ which restricts to $J_X$ on $\partial X \times \{0\}$.
    If a $J$-holomorphic curve $C$ in the completion of $X \overset{\frown}{\times} B^2(c)$ is asymptotic to any Reeb orbits outside of $\partial X \times \{0\}$, then the action of $C$ is larger than the capacity.
    On the other hand, if $C$ is asymptotic to orbits only in $\partial X \times \{0\}$, and is tangent to $D$ near $p \in \partial X \times \{0\}$, then we can show that the entire curve is contained in $\partial X \times \{0\}$.
    Therefore $C$ corresponds to a curve in $\overline{\overline{\mathcal{M}}}_X^{J_X}(\Gamma)\ll \mathcal{T}^{(k)}p\gg$ and thus has action greater than or equal to the capacity.
    The parameter on positive ends does not affect this argument, since the curves in the stabilization are exactly the curves in $X$, and in particular have the same number of positive ends.

    For the upper bound, we note that 4-dimensional convex toric domains satisfy the following criteria: any simple, index-0 formal curve component $C$ in $X$ satisfying the tangency constraint and having energy equal to the capacity are also formally perturbation invariant and have a regular moduli space with nonzero signed count for some almost complex structures.
    Formal perturbation invariance is defined in section \ref{sec:formal_perturbation_invariance}, and the above criteria is proven in Proposition \ref{prop:formal_perturbation_invariance}.
    Then we consider a formal curve in the smoothed product $X \overset{\frown}{\times} B^2(c)$, $\widetilde{C}$ which corresponds to $C$.
    It can be shown that $\widetilde{C}$ is also formally perturbation invariant and thus, for all relevant $J$'s, the moduli space of $\widetilde{C}$ is also nonempty.
    Hence, when we take the supremum over these $J$'s, the capacity of $X \overset{\frown}{\times} B^2(c)$ is still less than or equal the energy of $C$ which is the capacity of $X$.
    Again, because the correspondence between $C$ and $\widetilde{C}$ is exact, implying they share the same number of positive ends, the parameter $\ell$ does not change the argument.
\end{proof}

The remainder of the paper is devoted to proving the computation formula for 4-dimensional convex toric domains and applying it to specific examples, finding embedding obstructions that are not witnessed by the Gutt-Hutchings or McDuff-Siegel capacities.
First we will prove the lower bound for the capacity, which follows \cite{McDuff-Siegel_unperturbed_curves} without much modification.
Then we will prove upper bound, which requires more work: constructing the curves which realize the new combinatorial optimization in the formula.

\section{Computation Lower Bound}\label{sec:lower_bound}

In this section, we prove that, on four-dimensional convex toric domains, the defined capacity is bounded below by the formula given in Theorem \ref{thm:properties}\ref{thm:properties_subitem_computation}, following \cite{McDuff-Siegel_unperturbed_curves}.
Recall that a collection of ordered pairs $P = \{(i_s,j_s)\}_{s=1}^q$ is a \textbf{$(k,\ell)$-permissible word}, and say $P \in \mathcal{P}_{k,\ell}$ if it's $\frac{1}{2}$-index is equal to $k$, it is weakly permissible, and it has length at most $\ell$.
We make a similar definition where the $\frac{1}{2}$index is required to be at least $k$ rather than equal to $k$; we call this \textbf{$(\geq k,\ell)$-permissibility}, and denote the set of such words by $\mathcal{P}_{\geq k,\ell}$.

\begin{theorem}\label{thm:computation_lower_bound}[Theorem \ref{thm:properties}(f) lower bound]
    For a convex toric domain $X_\Omega$, using the notation from above, we have the following inequality on the capacity:
    $$\sup_{J \in \mathcal{J}(X_\Omega;D)} \left(\inf_{\Gamma \in \mathcal{G}_{J,k,\ell}} \mathcal{A}_{\partial X_\Omega}(\Gamma)\right) \geq \min_{\{(i_s,j_s)\}_{s=1}^q\in\mathcal{P}_{k,\ell}} \left( \sum_{s=1}^q ||(i_s,j_s)||_\Omega^* \right).$$
\end{theorem}
We will prove this in two parts, Propositions \ref{prop:bound_over_P_geq_k} and \ref{prop:bound_over_P_k} below.
First, we will prove the bound over $\mathcal{P}_{\geq k,\ell}$.
And it will suffice to do so for a small perturbation of $X_\Omega$ which has simpler Reeb dynamics, called the fully rounded and perturbed domain $\widetilde{X}_\Omega$.
The Reeb orbits up to a certain action in such a domain correspond to the ordered pairs $(i_s,j_s)$ on the right hand side, and their action is close to $||(i_s,j_s)||_\Omega^*$.
To show this initial bound, we will start with a generic $J$ and a $J$-holomorphic curve $C$ that has less energy than the capacity, and we'll prove that it's orbit set satisfies $(\geq k,\ell)$-permissibility.
Second, we will finish the proof by studying the action of orbit sets with respect to the index and showing that the minimizers of the right hand side indeed live in $\mathcal{P}_{k,\ell}$.
Since we start with a $J$-holomorphic curve with at most $\ell$ positive ends, it's orbit set clearly satisfies the combinatorial rule $q \leq \ell$, and the rest of the proof follows as in \cite{McDuff-Siegel_unperturbed_curves}.
We include the modified proofs for completion.

\begin{lemma}\label{lem:big_fully_rounding}
    Let $X_\Omega$ be a 4-dimensional convex toric domain.
    For any small $\epsilon > 0$ and large $K \gg 0$, there exists a \textbf{fully rounded and perturbed} convex toric domain $\widetilde{X}_\Omega$, which is an $\epsilon$-small $C_0$ deformation of $X_\Omega$, such that the Reeb dynamics on the boundary can be described as follows:
    \begin{enumerate}[label=(\alph*)]
        \item All of the orbits with action below $K$ are nondegenerate and either elliptic or positive hyperbolic.
        The elliptic orbits are enumerated by $e_{i,j}$ for $(i,j) \in \mathbb{Z}^2_{\geq 0} \setminus (0,0)$ with action $\widetilde{\mathcal{A}}(e_{i,j})$ close to $||(i,j)||_\Omega^* < K$.
        The hyperbolic orbits are enumerated by $h_{i,j}$ for $(i,j) \in \mathbb{Z}^2_{\geq 1}$ with action $\widetilde{\mathcal{A}}(h_{i,j})$ close to $||(i,j)||_\Omega^* < K$.
        We call these orbits \textbf{acceptable}, compared to the ones whose actions are too large.
        \item Further, let $C$ be a $J$-holomorphic curve which is asymptotically cylindrical to a set of acceptable orbits
        $$\Gamma = \{e_{i_s,j_s}\}_{s=1}^E \cup \{h_{i_t,j_t}\}_{t=1}^H$$
        for some generic $J \in \mathcal{J}(\widetilde{X}_\Omega)$, $E,H \geq 0$, and satisfying a $\ll \mathcal{T}^{(k)} p\gg$ tangency constraint.
        Then 
        $$\emph{ind}(C) = E + H - 2 + \sum_{s=1}^E (2i_s + 2j_s + 1) + \sum_{t=1}^H(2i_t + 2j_t) - 2k$$
        In particular, if the set is elliptic, i.e. $E = q$ and $H = 0$, then the half-index is given by $\frac{1}{2}\text{ind}(C) = \sum (i_s + j_s) + q - 1 - k$.
        \item If $C$ is a somewhere-injective, then $\Gamma$ is strongly permissible (i.e. weakly permissible and excluding words of the form $e_{0,k}$ and $e_{k,0}$ for $k \geq 2$).
        \item The actions of different acceptable orbits are distinct and linearly independent over $\mathbb{Q}$, and the action difference between $e_{i,j}$ and $h_{i,j}$ is very small:
        $$0 < \widetilde{\mathcal{A}}(e_{i,j}) - \widetilde{\mathcal{A}}(h_{i,j}) < |\widetilde{\mathcal{A}}(\Gamma_1) - \widetilde{\mathcal{A}}(\Gamma_2)|,$$
        for all $\Gamma_1, \Gamma_2$ distinct words of acceptable elliptic orbits.
    \end{enumerate}
\end{lemma}

\begin{remark}\label{rem:C0_continuity}
    By a standard observation, the capacity $\Tilde{\mathfrak{g}}_k^\ell$ is continuous with respect to $C_0$ deformations of $X$ inside of $\widehat{X}$.
    Also, $\Tilde{\mathfrak g}_k^\ell$ is continuous on convex toric domains, with respect to the Hausdorff distance of their moment images, see \cite[Lemma 2.3]{CCGFHR_sympl_embs_4d_concave_toric_domains}.
    Thus it suffices to consider these fully rounded and perturbed domains.
\end{remark}

\begin{proposition}\label{prop:bound_over_P_geq_k}
    For any four-dimensional convex toric domain $X_\Omega$, we have
    $$\Tilde{\mathfrak g}_k^\ell(X_\Omega) \geq \min_{\{(i_s,j_s)\}_{s=1}^q \in \mathcal{P}_{\geq k,\ell}} \sum_{s=1}^q||(i_s,j_s)||_\Omega^*.$$
\end{proposition}

\begin{proof}
    This proof follows \cite[Lemma 4.2.3]{McDuff-Siegel_unperturbed_curves} and is modified only slightly for the filtration on $\ell$ positive ends: we start with a curve in the moduli space with an $\ell$-restriction, so it naturally satisfies the combinatorial condition that $q \leq \ell$.
    
    Now, we use Lemma \ref{lem:big_fully_rounding} and Remark \ref{rem:C0_continuity} to simplify the problem: by $C^0$ continuity of $\tilde{\mathfrak{g}}_k^\ell$, it suffices to prove the analogous claim for the fully rounded and perturbed $\widetilde{X}_\Omega$.
    Recall the definition
    $$\Tilde{\mathfrak{g}}_k^\ell(\widetilde{X}_\Omega) := \sup_{J \in \mathcal{J}(\widetilde{X}_\Omega;D)} \inf_{\Gamma \in \mathcal{G}_{J,k,\ell}} \widetilde{\mathcal{A}}_{\partial \widetilde{X}_\Omega}(\Gamma),$$
    where the infimum is over all tuples $\Gamma = (\gamma_1,...,\gamma_q)$ of Reeb orbits with $q \leq \ell$ such that $\overline{\overline{\mathcal{M}}}_{\widetilde{X}_\Omega}^J(\Gamma)\ll \mathcal{T}^{(k)}p \gg$ is nonempty.
    Let $J \in \mathcal{J}(\widetilde{X}_\Omega;D)$ be generic.
    By definition, 
    $$\tilde{\mathfrak{g}}_k^\ell (\widetilde{X}_\Omega) \geq \inf_{\Gamma \in \mathcal{G}_{J,k,\ell}} \widetilde{\mathcal{A}}_{\partial \widetilde{X}_\Omega}(\Gamma),$$
    so let $C$ be a curve in $\overline{\overline{\mathcal{M}}}_{\widetilde{X}_\Omega}^{J,\ell}(\Gamma)\ll\mathcal{T}^{(k)}p \gg$ for some $\Gamma$ such that $\tilde{\mathfrak{g}}_k^\ell(\widetilde{X}_\Omega) \geq E(C).$
    We want $C$ to correspond to a set $\{(i_s,j_s)\}_{s=1}^q \in \mathcal{P}_{\geq k,\ell}$.
    By Lemma \ref{lem:big_fully_rounding}, we get a corresponding set and immediately see that $q \leq \ell$ by definition of the moduli space.
    It's left to show the index inequality and weak permissibility.

    Now, it may be the case that $C$ is an $n$-fold multiple cover, so let $\overline{C}$ and $\overline{\Gamma}$ be the underlying simple curve and corresponding word of orbits.
    By genericity of $J$, $\text{ind}(\overline{C}) \geq 0$.
    Then by Lemma \ref{lem:tangency_constraint_index_inequality} below we get an index inequality for multiply covered curves: $\text{ind}(C) \geq n\cdot\text{ind}(\overline{C}) \geq 0$.
    Thus, $\sum_{s=1}^q(2i_s + 2j_s + 1) + q - 2 \geq \text{ind}(C) + 2k \geq 2k$.

    Also by the genericity of $J$, $\overline{C}$ is somewhere injective, implying strong permissibility of $\overline{\Gamma}$ by Lemma \ref{lem:big_fully_rounding}. Then $\Gamma$ is weakly permissible, since the only difference between strong and weak permissibility is the presence of multiple covers of $e_{1,0}$ and $e_{0,1}$.
    Thus, $\{(i_s,j_s)\}_{s=1}^q \in \mathcal{P}_{\geq k,\ell}$.
    And, $\widetilde{\mathcal{A}}(\Gamma)$ is $\epsilon$-close to $\sum ||(i_s,j_s)||_\Omega^*$ by Lemma \ref{lem:big_fully_rounding}, where $\epsilon$ depends on the fully rounding procedure.
    Thus, 
    $$\tilde{\mathfrak{g}}_k^\ell (\widetilde{X}_\Omega) \geq E(C) = \widetilde{\mathcal{A}}(\Gamma) \geq \min_{\mathcal{P}_{\geq k,\ell}} \sum ||(i_s,j_s)||_\Omega^* - \ell\epsilon.$$
    Hence, taking $\epsilon$ to $0$ concludes the proof.
\end{proof}

To finish the proof of \ref{thm:computation_lower_bound}, we need to show that minimizers on the combinatorial side of the inequality indeed have index equal to $2k$ rather than index greater than $2k$.
Thus our goal is to prove the following:

\begin{proposition}\label{prop:bound_over_P_k}
    For $k,\ell \geq 1$,
    $$\min_{\mathcal{P}_{\geq k, \ell}} \sum_s||(i_s,j_s)||_\Omega^* \geq \min_{\mathcal{P}_{k, \ell}} \sum_s||(i_s,j_s)||_\Omega^*$$
\end{proposition}

\begin{proof}
    Let $P = \{(i_s,j_s)\}_{s=1}^q \in \mathcal{P}_{\geq k,\ell}$ be a minimizer of the left hand side.
    If the index of $P$ equals $2k$, then we are done, so assume $P \in \mathcal{P}_{k',\ell}$ for $k' > k$.
    We claim there exists a word $P' \in \mathcal{P}_{k'-1,\ell}$ with action at most the action of $P$.
    If this is true, then we can iterate until we get a word in $\mathcal{P}_{k,\ell}$ with action at most the action of $P$, finishing the proof.

    First, assume $P$ has a point with either $i_s > 1$ or $j_s > 1$.
    If it's $i_s > 1$, take
    $$P' = \{(i_1,j_1),...,(i_s - 1,j_s),...,(i_q,j_q)\}.$$
    Now the $\frac{1}{2}$-index is $k' - 1$, the length is still $q \leq \ell$, and $P'$ is still weakly permissible since $i_s-1 > 0$.
    That is $P' \in \mathcal{P}_{k'-1,\ell}$.
    And the action of $P'$ is at most the action of $P$, since $||(i_s - 1,j_s)||_\Omega^* \leq ||(i_s,j_s)||_\Omega^*$.
    The argument is identical if $j_s > 1$.

    Now, if the first assumption is not true, then $P$ is of the form $$\{(0,1)\}^{\times m_1} \cup \{(1,0)\}^{\times m_2} \cup \{(1,1)\}^{\times m_3},$$
    with $m_1,m_2,m_3 \geq 0$.
    Denoting the corners of $\Omega$ by $(a,0)$ and $(0,b$), without loss of generality taking $a \geq b$, we have
    $$||(0,1)||_\Omega^* = b, \quad ||(1,0)||_\Omega^* = a, \quad a \leq ||(1,1)||_\Omega^* \leq a+b.$$
    If $m_3 > 0$, replace a $(1,1)$ with either $(0,1)$ or $(1,0)$, whichever does not spoil weak permissibility, and we are done.
    Otherwise, if $m_2 > 1$, replace two $(1,0)$'s with a single $(1,1)$. 
    The action trade-off is
    $$||(1,1)||_\Omega^* \leq a+b \leq 2a = 2||(1,0)||_\Omega^*,$$
    and $\frac{1}{2}$-index is $k'-1$, so we are done.
    Otherwise, if $m_2 = 1$, and $m_1 = 0$, then $k' = 1$ and we are already done.
    Otherwise, if $m_2 = 1$, and $m_1 > 0$, then take $P' = \{(0,1)\}^{\times m_1 - 1} \cup \{(1,1)\}$.
    The action trade-off is
    $$||(1,1)||_\Omega^* \leq a + b = ||(1,0)||_\Omega^* + ||(0,1)||_\Omega^*,$$
    and $\frac{1}{2}$-index is $k'-1$, so we are done.
    Lastly, if $m_2 = 0$, then weak permissibility implies that $P = \{(0,1)\}$, with $k' = 1$, so we are done.
\end{proof}

We spend the rest of this section proving Lemma \ref{lem:big_fully_rounding}.

\begin{lemma}\label{lem:fully rounding}
    Given small $\epsilon,\nu > 0$, if $X_\Omega$ is a four-dimensional convex toric domain, then it's moment image $\Omega$ is $\epsilon$-close in the Hausdorff distance to some `fully rounded' domain $\Omega^{FR}$ which satisfies the following properties:
    \begin{enumerate}[label=(\alph*)]
        \item The boundary of $\Omega^{FR}$ is given by the axes and a smooth function $h \colon [0,a] \rightarrow [0,b]$, for some $a,b \in \mathbb{R}_{>0}$, with $h(0) = b$ and $h(a) = 0$.
        \item $h$ is strictly decreasing and strictly concave down.
        \item $0 > h'(0) > -\nu$, and $-1/\nu > h'(a) > -\infty$, so the boundary forms nearly right angles at the axes.
        \item $h'(0)$ and $h'(a)$ are irrational.
    \end{enumerate}
\end{lemma}

\begin{proof}[Proof \cite{McDuff-Siegel_unperturbed_curves} Section 4.1.]
    Roughly, we first make a piecewise linear approximation inside an $\epsilon/2$ neighborhood of the boundary, such that it satisfies (c) and (d). Then we take a smoothing of this, small enough so that it it fits inside an $\epsilon$ neighborhood of the boundary and satisfies (a) and (b).
    Also, see \cite[\S 2.2]{Gutt-Hutchings}.
\end{proof}

Note that, with the above properties, $\Omega^{FR}$ has a compact and convex symmetrization, and is therefore the moment image of a convex toric domain $X_{\Omega^{FR}}$.
Now we'll take a look at the Reeb dynamics on the boundary.
Let $X_{\Omega^{\text{FR}}}$ be a fully rounded convex toric domain, with boundary data $h,a,b$ and $\nu$ given by Lemma \ref{lem:fully rounding}.
For every $(i,j) \in \mathbb{Z}_{\geq 1}^2$ such that $-\nu > -j/i > -1/\nu$, there is a point $(c_1,c_2)$ on the boundary of $\Omega^{\text{FR}}$ such that $h'(c_1) = -j/i$.
The preimage of the moment map at this point, $\mu^{-1}(c_1,c_2)$, is an $S^1$ family of $\text{gcd}(i,j)$-fold cover Reeb orbits with action $c_1j + c_2i$ in the boundary of $X_{\Omega^{\text{FR}}}.$
The preimages $\mu^{-1}(0,b)$ and $\mu^{-1}(a,0)$ are  nondegenerate elliptic Reeb orbits of period $b$ and $a$, respectively.
Call these orbits $e_{0,1}$ and $e_{1,0}$ respectively.
Moreover, the multiple covers of these orbits are also nondegenerate and elliptic.

Returning to the Reeb orbits away from these corners, we have the following lemma which turns a degenerate family of orbits into a pair of nondegenerate orbits up to a certain action.
See \cite[\S 5.3]{Hutchings_2016_beyond_ech} and \cite[\S 2.2]{Bourgeois_a_Morse-Bott_approach} for proof.
The idea is to perturb the contact form based on a Morse function on the circle of Reeb orbits.
Then the two nondegenerate orbits correspond to the two critical points of this Morse function, and any other Reeb orbits that may arise during the process must have high action.

\begin{lemma}\label{lem:perturbation}
    For any large $K \gg 0$ and small $\delta > 0$, there exists a perturbation of the contact form $\alpha^* = f\alpha$ with $||f-1||_{C^0} < \delta$ such that
    \begin{itemize}
        \item $\alpha^* = \alpha$ near $e_{0,1}$ and $e_{1,0}$.
        \item Any torus $\pi^{-1}(c_1,c_2)$ with $h'(c_1) = -j/i$, whose Reeb orbits have action less than $K$ is replaced by an elliptic orbit $e_{i,j}$ and a positive hyperbolic orbit $h_{i,j}$, both of which are nondegenerate.
        \item The actions of $e_{i,j}$ and $h_{i,j}$ are less than $K$ and are within $\delta$ of the action from the torus orbits.
        \item There are no other Reeb orbits of action less than $K$ after perturbation.
    \end{itemize}
\end{lemma}

\begin{proof}[Proof of Lemma \ref{lem:big_fully_rounding}]
    The above lemmas give us the orbits we need and show that their actions are close to $||(i,j)||_\Omega^*$, proving part (a).
    Part (c) is true because somewhere injective curves which are not strongly permissible fail to satisfy the relative adjunction formula, as shown in \cite[Lemma 4.2.2]{McDuff-Siegel_unperturbed_curves}.
    Part (d) is given by \cite[Lemma 4.1.1]{McDuff-Siegel_unperturbed_curves}

    To see part (b), we use a trivialization $\tau_{\text{ex}}$ (see \cite[\S 3.2]{McDuff_Siegel_counting_curves}) for which the Conley Zehnder indices are given by
    $$CZ_{\tau_{\text{ex}}}(e_{i,j}) = 2i + 2j + 1, \quad\quad CZ_{\tau_{\text{ex}}}(h_{i,j}) = 2i + 2j,$$
    and for which the first Chern class vanishes.
    Then for positive asymptotics $\Gamma = \{\gamma_s\}_1^q$ and a $k$-tangency constraint, the index formula gives the desired result:
    \begin{equation*}
    \begin{aligned}
        \text{ind}(C) &= -\chi(C) + 2c_{\tau_{\text{ex}}}(C) + \sum_{s=1}^q CZ_{\tau_{\text{ex}}}(\gamma_s) - 2k\\
        &= q - 2 + \sum_{s=1}^q CZ_{\tau_{\text{ex}}}(\gamma_s) - 2k.
    \end{aligned}
    \end{equation*}
\end{proof}

\section{Computation Upper Bound}\label{sec:upper_bound}

In the following sections, we prove the upper bound for the computation of $\Tilde{\mathfrak g}_k^\ell$ on 4-dimensional convex toric domains.

\begin{theorem}\label{thm:upper_bound}[Theorem \ref{thm:properties}(f) upper bound]
    For any 4-dimensional convex toric domain $X_\Omega$, we have the following bound on the capacity $\Tilde{\mathfrak{g}}_k^\ell$:
    \begin{equation}\label{eq:upper_bound}
        \sup_{J \in \mathcal{J}(X_\Omega;D)} \left(\inf_{\Gamma \in \mathcal{G}_{J,k,\ell}} \mathcal{A}_{\partial X_\Omega}(\Gamma)\right) \leq \min_{\{(i_s,j_s)\}_{s=1}^q \in \mathcal{P}_{k,\ell}} \left( \sum_{s=1}^q ||(i_s,j_s)||_\Omega^*\right).
    \end{equation}
\end{theorem}

Again it suffices to prove this bound up to a small $\delta$ for the fully rounded and perturbed domain $\widetilde{X}_\Omega$, where $\delta$ approaches $0$ as the perturbation shrinks.
In the proof, we will start with a combinatorial minimizer from the right hand side, $\{(i_s,j_s)\}_{s=1}^q$, and show that for all $J \in \mathcal{J}(\widetilde{X}_\Omega;D)$ there exists a curve in $\overline{\overline{\mathcal{M}}}_X^J(\Gamma)\ll \mathcal{T}^{(k)}p \gg$ for $\Gamma := e_{i_1,j_1} \times \cdots \times e_{i_q,j_q}$, the corresponding word of elliptic Reeb orbits.
Thus for any $J$, the infimum on the left hand side is less than or equal to the right hand side, up to a small $\delta$ coming from the perturbation which turns degenerate tori into two nondegenerate orbits.
Note that proving the upper bound is the main theoretical contribution of this paper, because new combinatorial minimzers on the right hand side arise from the $\ell$ constraint, thus requiring new arguments for existence in the moduli space.

To make the above argument work for all $J$ we will need the notion of formal curves and formal perturbation invariance to extend our results from generic $J$.
In section \ref{sec:formal_perturbation_invariance}, we define these notions and prove that the orbit sets corresponding to our combinatorial minimizers yield formally perturbation invariant curves.
Then the count invariance in the following proposition shows that existence of our desired curves for generic $J$ implies existence for all $J$.
In Section \ref{sec:existence_of_curves} we prove the existence of curves in the moduli space corresponding to combinatorial minimizers.
In this section, we assume formal perturbation invariance and existence, and prove Theorem \ref{thm:upper_bound} by analyzing the minimization problem.

\begin{proposition}[{\cite[Proposition 2.4.2]{McDuff-Siegel_unperturbed_curves}}]\label{prop:perturbation_invariance_implies_count_independence}
    Let $X$ be a Liouville domain with nondegenerate contact boundary $Y$, and let $C$ be a simple index zero formal curve component $X$ which carries a local tangency constraint $\ll \mathcal{T}^{(k)} p\gg$ and is asymptotic to $\Gamma$.
    Assume that $C$ is formally perturbation invariant with respect to some generic $J_Y \in \mathcal{J}(Y)$.
    Then the associated moduli space $\mathcal{M}_X^{J_X}(\Gamma)\ll \mathcal{T}^{(k)} p\gg$ is regular and finite for generic $J_X \in \mathcal{J}^{J_Y}(X;D)$, and moreover the signed count $\#\mathcal{M}_X^{J_X}(\Gamma)\ll \mathcal{T}^{(k)} p\gg$ is independent of $J_X$ provided that $\mathcal{M}_X^{J_X}(\Gamma)\ll \mathcal{T}^{(k)} p\gg$ is regular.
\end{proposition}

\begin{remark}
    Note that since the moduli space has a fixed $\Gamma$, there is no dependence on the number of positive ends, so the proof of this proposition follows \cite{McDuff-Siegel_unperturbed_curves} exactly as written.
\end{remark}

Now, let's analyze the optimization problem on the right hand side.

\begin{lemma} \label{lem:index-preserving-move}
    Let $X_\Omega$ be a 4-dimensional convex toric domain, with $\partial \Omega$ given by a piece-wise smooth function $h \colon [0,x] \rightarrow [0,1]$, with a finite number of non-smooth vertices (fully rounded domains satisfy this condition).
    Assume the intersection of $h$ with the line $y=x$ is not a vertex of $h$, and denote the slope of $h$ at this point by $m$.
    For $(i,j) \in \mathbb{Z}_{\geq 1}^2$, the following index-preserving moves which weakly decrease action are possible:
    \begin{itemize}
        \item If $-(i-1)/(j+1) \leq m$, then $||(i-1,j+1)||_\Omega^* \leq ||(i,j)||_\Omega^*$.
        \item If $-(i+1)/(j-1) \geq m$, then $||(i+1,j-1)||_\Omega^* \leq ||(i,j)||_\Omega^*$.
    \end{itemize}
    In particular, the minimum norm for a given index $k$,
    $$\min\{||(i,j)||_\Omega^* \;\vert\; i + j = k\},$$
    is minimized at $(i,j)$ if $-i/j = m$, or if there is no such $i,j$, then it is minimized at one of $(i,j), (i-1,j+1)$ where $-i/j < m < -(i-1)/(j+1)$.
\end{lemma}

\begin{proof}
    First, we will assume $-(i-1)/(j+1) \leq m$, implying that $-i/j < m$.
    Choose $(a,b)$ and $(a',b')$ to be maximizers for the inner products of $(i,j)$ and $(i-1,j+1)$, i.e.
    $$||(i,j)||_\Omega^* = \langle (i,j), (a,b)\rangle = ai + bj,$$
    $$||(i-1,j+1)||_\Omega^* = \langle (i-1,j+1), (a',b')\rangle = a'(i-1) + b'(j+1).$$

    Now, assume $a' \neq a$, and $b' \neq b$.
    By the slope assumptions, and the fact that $\Omega$ is convex, we must have that $(a,b)$ and $(a',b')$ are below the $y = x$ line, implying $a > b$, $a' \geq b'$, $a' < a$, and $b' > b$.
    
    By the inner-product-maximizing assumption on $(a',b')$, we have
    \begin{equation*}
    \begin{aligned}
        a(i-1) + b(j+1) &\leq a'(i-1) + b'(j+1)\\
        (b-b')(j+1) & \leq (a'-a)(i-1)\\
        (b'-b)(j+1) & \geq (a'-a)(-(i-1))\\
        \frac{b'-b}{a'-a} &\leq \frac{-(i-1)}{j+1} \leq m\\
    \end{aligned}
    \end{equation*}
    Putting everything together, we see that
    \begin{equation*}
    \begin{aligned}
        ||(i-1,j+1)||_\Omega^* - ||(i,j)||_\Omega^* & = a'(i-1) + b'(j+1) -ai - bj\\
        &= (a'-a)i + (b'-b)j + (b'-a')\\
        &\leq (a'-a)i + (b'-b)j\\
        &= (a'-a)i + (a'-a)\frac{b'-b}{a'-a}j\\
        &\leq (a'-a)i + (a'-a)mj\\
        &= (a'-a)(i+mj)\\
        &< 0.
    \end{aligned}
    \end{equation*}
    
    Now, if $a' = a$, then $(i,j) = (i,0)$, and $(a',b')$ must be at a vertex on a vertical segment, so we can choose $(a,b)$ to be at that same vertex.
    Similar reasoning holds when $b' = b$.
    So in this case where $(a,b) = (a',b')$, we have
    $$||(i-1,j+1)||_\Omega^* - ||(i,j)||_\Omega^* = a(i-1) + b(j+1) -ai -bj = b-a \leq 0.$$
    The argument is symmetric when $-(i+1)/(j-1) \geq m$, completing the proof.
\end{proof}

\begin{lemma}\label{lem:upper_bound_rhs_minimizers}
    Given $k$ and $\ell$, the minimum
    $$\min_{\{(i_s,j_s)\}_{s=1}^q \in \mathcal{P}_{k,\ell}} \left( \sum_{s=1}^q ||(i_s,j_s)||_\Omega^*\right)$$
    is achieved by a set $\{(i_s,j_s)\}_{s=1}^q$ such that all but possibly one of the tuples $(i_s,j_s)$ are individual minimizers of
    $$\min_{i+j = k_s}||(i,j)||_\Omega^*,$$
    for $k_s = i_s + j_s$ and the possible extra tuple is of the form $(1,j)$ for some $j \geq 0$, which may be needed to preserve weak permissibility.
\end{lemma}

\begin{proof}
    Let $\{(i_s,j_s)\}_{s=1}^q$ be any $(k,\ell)$-permissible set, which we recall satisfies
    \begin{itemize}
        \item $\left(\sum_{s=1}^1i_s+j_s\right) + q - 1 = k$
        \item $q \geq 2 \Rightarrow (i_1,...,i_q) \neq (0,...,0)$ and $(j_1,...,j_q) \neq (0,...,0)$.
        \item $q \leq \ell$.
    \end{itemize} 
    Note that if $q = 1$, we are done, so assume $q > 1$.
    We proceed by making two substitutions.
    After both steps we will have a set which is of the desired form and of weakly less action than the original.
    First, if needed, replace any $(i_s,j_s)$ with the minimizer of 
    $$\min_{i+j = k_s} ||(i,j)||_\Omega^*.$$
    This looks like moving all the points diagonally to some ``minimizing region'' in $\mathbb{Z}_{\geq 0}^2$ made of ordered pairs which straddle the slope $m$ from Lemma \ref{lem:index-preserving-move}.
    If the first step was utilized and we moved all points $(i_s,j_s)$ onto the, say vertical, axis and spoiled weak permissibility, then the second step is to substitute some $(0,i_t + j_t)$ for $(1,i_t + j_t - 1)$, such that the original $i_t \neq 0$
    (One of the original $i_s$ must have originally been nonzero, otherwise the original set would not have been weakly permissible).
    After this substitution, we again have a $(k,\ell)$-permissible set.
    A symmetric argument holds for the horizontal axis.

    Now, we wish to show that after these two steps, we have not increased action.
    If only the first step was utilized, then we did not increase the action, because none of the substitutions in the first step did so.
    Thus our only concern is when the step 2 is utilized.
    Let $\{(i_s,j_s)\}_1^q$ be the original set and $\{(0,i_s+j_s)\}_1^q$ the set after step 1.
    Then the resulting set after step 2 looks like
    $$\{(0,i_1+j_1),...,(1,i_t+j_t-1),...,(0,i_q+j_q)\}$$
    for some $t \in \{1,...,q\}$.
    We have
    $$\sum||(i_s,j_s)||_\Omega^* \geq ||(0,i_1+j_1)||_\Omega^* + \cdots ||(i_t,j_t)||_\Omega^* + \cdots + ||(0,i_q+j_q)||_\Omega^*,$$
    so it suffices to show that $||(i_t,j_t)||_\Omega^* \geq ||(1,i_t+j_t-1)||_\Omega^*$.
    Above in step 2, we took $i_t \neq 0$.
    Then, this follows by Lemma \ref{lem:index-preserving-move}, which says that action weakly increases as you move away from the minimizing line.
    Indeed, $(1,i_t+j_t-1)$ is as close we can get to the minimizing line without being on it, so it has weakly less action then the original $(i_t,j_t)$.
\end{proof}

We can now prove the Theorem.

\begin{proof}[Proof of Theorem \ref{thm:upper_bound}]

    Let $P = \{(i_s,j_s)\}_{s=1}^q \in \mathcal{P}_{k,\ell}$ be a minimizer of the right hand side of (\ref{eq:upper_bound}), whose form is given above by Lemma \ref{lem:upper_bound_rhs_minimizers}.
    Let $\Gamma = \{e_{i_s,j_s}\}_{s=1}^q$ be the set of elliptic orbits on the boundary of $\widetilde{X}_\Omega$. 
    By definition of $\Tilde{\mathfrak g}_k^\ell$, for all $J \in \mathcal{J}(\widetilde{X}_\Omega ; D)$, we must find a $J$-holomorphic curve $C$ asymptotic to the word of orbits represented by $P$, satisfying the tangency constraint $\ll \mathcal{T}^{(k)}p\gg$.
    That is, we must prove that $\mathcal{M}_{\widetilde{X}_\Omega}^J(\Gamma)\ll\mathcal{T}^{(k)}p\gg$ is nonempty for all $J \in \mathcal{J}^{J_{\partial \widetilde{X}_\Omega}}(\widetilde{X}_\Omega)$.
    To do this, we prove in Section \ref{sec:formal_perturbation_invariance} that the \emph{formal curve component} $C$ in $\widetilde{X}_\Omega$ with tangency constraint $\ll \mathcal{T}^{(k)}p\gg$ and positive asymptotics $\Gamma$ of index $2k$ is \emph{formally perturbation invariant} with respect to some generic $J_{\partial \widetilde{X}_\Omega} \in J(\partial \widetilde{X}_\Omega)$.
    Then the moduli space of interest is regular and has invariant signed count by Proposition \ref{prop:perturbation_invariance_implies_count_independence}.
    We will show here that the moduli space for the generic $J_{\partial \widetilde{X}_\Omega}$ is nonempty by explicitly constructing curves.
    Automatic transversality shows that these curves only count positively, so the signed count will be nonzero, and the moduli space must be nonempty for all $J$, as desired.
    
    By Section 2.2 of \cite{McDuff-Siegel_unperturbed_curves}, we may replace this tangency constraint with a skinny ellipsoid constraint, considering the moduli space of curves with positive ends asymptotic to $\Gamma$, and a negative end asymptotic to $\eta_k$, the $k$-fold cover of the short orbit of a sufficiently skinny, sufficiently small ellipsoid, $E_{\text{sk}}$, embedded in $\Tilde{X}_{\Omega}$.
    By \cite[Lemma 5.4.3]{McDuff-Siegel_unperturbed_curves}, for $i = 1$ or $j=1$ or both, there exist $J$-holomorphic cylinders positively asymptotic to $e_{i,j}$ and negatively asymptotic to $\eta_{i+j}$ in $\Tilde{X}_\Omega \setminus E_{\text{sk}}$.
    In the proof of \cite[Proposition 5.6.1]{McDuff-Siegel_unperturbed_curves}, there is an argument for the existence of $J$-holomorphic cylinders from $e_{i_s,j_s}$ to $\eta_{i_s+j_s}$ if $(i_s,j_s)$ is a minimizer of $||(i,j)||_\Omega^*$, subject to $i+j=k_s$, for $k_s := i_s + j_s$ as in the above lemma.
    Now, we wish to iteratively use obstruction bundle gluing, as in \cite[Lemma 5.4.2]{McDuff-Siegel_unperturbed_curves}, to create the desired curve $C$, positively asymptotic to $\{e_{i_s,j_s}\}_{s=1}^q$, and negatively asymptotic to $\eta_k$, where $k = \frac{1}{2}\text{ind}(\Gamma) = \left(\sum_{s=1}^q i_s+j_s\right) + q-1$.
    
    The obstruction bundle gluing lemma assumes that we are given two simple immersed index-0 $J$-holomorphic curves $C_1$ and $C_2$ with \emph{distinct} images in $\widetilde{X}_\Omega \setminus E_{\text{sk}}$, such that each $C_i$ is positively asymptotic to a word $\Gamma_i$ of Reeb orbits, and negatively asymptotic to $\eta_{k_i}$ for some $k_i$.
    Then there is a generic $J'$ agreeing with $J$ on the boundary such that there exists a simple immersed index-0 $J'$-holomorphic curve in $\widetilde{X}_\Omega \setminus E_{\text{sk}}$, positively asymptotic to $\Gamma_1 \cup \Gamma_2$, and negatively asymptotic to $\eta_{k_1+k_2+1}$.
    
    So if we want to construct $C$, asymptotic to $\{e_{i_s,j_s}\}_{s=1}^q$ and negatively asymptotic to $\eta_k$, we can do so using this lemma, as long as $C$ does not have identical positive ends, i.e. of the form $e_{i,j} \times \cdots \times e_{i,j}$ for some $i,j$ (note that $i$ and $j$ are both nonzero here by weak permissibility of $\Gamma$).
    For in this case, the images of any $C_1$ and $C_2$ we would like to start with are not distinct, so the Lemma does not apply.
    Otherwise, we can always use obstruction bundle gluing to construct a pair of pants first, and then add the rest of the positive ends iteratively, with the negative end updating as desired.
    Note here that we start with the assumption that $\{(i,j),...,(i,j)\}$ is a minimizer rather than $\{(i,j),(i,j)$, but it follows that the latter must also be a minimizer, by the contrapositive statement and action considerations.

    We deal with the issue of identical positive ends in section \ref{sec:existence_of_curves} and show that they do exist.
    Thus, for any combinatorial minimizer of (\ref{eq:upper_bound}), we've shown existence of the corresponding $J$-holomorphic curve for any $J \in \mathcal{J}(\widetilde{X}_\Omega;D)$, forcing the desired bound on the capacity.
\end{proof}

\section{Formal Perturbation Invariance}\label{sec:formal_perturbation_invariance}

In this section, we define formal curves and formal perturbation invariance, and we prove that a formal curve of $\widetilde{X}_\Omega$ which is asymptotic to the elliptic orbit set corresponding to the combinatorial minimizer in the computation formula is formally perturbation invariant.
This gives us the invariant signed count to make the arguments above work for all admissible almost complex structures, rather than just generic ones.
This section differs from the analogous \cite[\S 2.3, 2.4, 5.1]{McDuff-Siegel_unperturbed_curves} because we have more combinatorial minimizers to check.
In particular, we could have multiple covers that are not of the form $e_{0,k}$ or $e_{k,0}$.

First we restate the relevant definitions from \cite[\S 2.3]{McDuff-Siegel_unperturbed_curves}.
In brief, formal curves only record the data of asymptotics, homology class, and tangency via the index, and formal perturbation invariance of a curve means that when this curve breaks into a psuedoholomorphic building under certain assumptions, it can only break in a simple way with regular components and zero signed count.
Then we will do index calculations for formal covers and prove lemmas that will be needed to rule out more complicated breaking and show regularity and zero signed count for the buildings that do occur.

\subsection{Formal Curves and Formal Perturbation Invariance}

\begin{definition}
    A \textbf{formal curve component} $C$ in a compact symplectic cobordism $(X,\omega)$ is a triple $(\Gamma^+,\Gamma^-,A)$, where
    \begin{itemize}
        \item $\Gamma^+ = (\gamma_1^+,...,\gamma_a^+)$ and $\Gamma^- = (\gamma_1^-,...,\gamma_b^-)$ are tuples of Reeb orbits in $\partial^+ X$ and $\partial^- X$, respectively.
        \item $A \in H_2(X,\Gamma^+ \cup \Gamma^-)$ is a homology class such that the energy $E_X(C) := E_X(A) = \int_A \omega$ is nonnegative.
    \end{itemize}
    Similarly, a formal curve component $C$ in a strict contact manifold $(Y,\alpha)$ is a triple $(\Gamma^+,\Gamma^-,A)$, where $\Gamma^+,\Gamma^-$ are tuples of Reeb orbits in $Y$ and $A \in H_2(Y, \Gamma^+ \cup \Gamma^-)$ is a homology class, and we requeire the energy $E_Y(C) := E_Y(A) = \int_Ad\alpha$ to be nonnegative.
    As a shorthand, we will denote $\mathcal{M}_X^J(C) := \mathcal{M}_{X,A}^J(\Gamma^+;\Gamma^-)$, for $C$ a formal curve with the data of $\Gamma^+,\Gamma^-,A$, and we will use similar notation for all the other kinds of moduli spaces discussed.
\end{definition}

\begin{definition}
    A \textbf{connected formal nodal curve} $C$ in $X$ (resp. $Y$) consists of a tree $T$ and a formal curve component $C_v$ in $X$ (resp. $Y$) for each vertex $v$ of $T$.
    More generally, we drop the ``connected'' condition by allowing $T$ to be a forest.
    We will say that $C$ is \textbf{stable} if, for each nonconstant component $C_v$, the number of punctures plus the number of edges connected to $v$ is at least three.
\end{definition}

\begin{definition}
    A \textbf{formal building} in $X$ consists of:
    \begin{itemize}
        \item formal nodal curves $C_1,...,C_a$ in $\partial^+X$ for some $a \in \mathbb{Z}_{\geq 0}$,
        \item a formal nodal curve $C_0$ in $X$
        \item formal nodal curves $C_{-1},...,C_{-b}$ in $\partial^-X$ for some $b \in \mathbb{Z}_{\geq 0}$,
    \end{itemize}
    such that the tuple of positive Reeb orbits for $C_i$ coincides with the tuple of negative Reeb orbits for $C_{i+1}$ for $i = -b,..., a - 1$.
    We also assume that the graph given naturally by concatenating the forest of each level is acyclic.

    Similarly, a formal building in $Y$ consists of formal nodal curves $C_1, . . . , C_a$ in $Y$ for some $a \in \mathbb{Z}_{\geq 1}$, that the tuple of positive Reeb orbits for $C_i$ coincides with the tuple of negative Reeb orbits for $C_{i+1}$ for $i = 1, . . . , a - 1$, and such that the underlying graph is acyclic.
    We say that a formal building is \textbf{stable} if each constituent formal nodal curve is stable, and no level is a union of trivial cylinders.
\end{definition}

We denote the formal analogue of $\overline{\overline{\mathcal{M}}}_{X,A}(\Gamma^+;\Gamma^-)\ll\mathcal{T}^{(m)}p\gg$ by $$\overline{\overline{\mathcal{F}}}_{X,A}(\Gamma^+;\Gamma^-)\ll\mathcal{T}^{(m)}p\gg,$$ where the formal tangency constraint is just the additional data of the index, which can be calculated using $\Gamma^+,\Gamma^-,A$, and then decreased by $2n-4+2m$.

\begin{definition}
    Let $X$ be a symplectic filling, and let $\Gamma = (\gamma_1,...,\gamma_a)$ and $\overline{\Gamma} = (\overline{\gamma}_1,...,\overline{\gamma}_{\overline{a}})$ be tuples of Reeb orbits in $Y := \partial X$.
    Let $C = (\Gamma, \emptyset, A)$ and $\overline{C} = (\overline{\Gamma}, \emptyset, \overline{A})$ be formal curve components in $X$, satsifying constraints $\ll\mathcal{T}^{(m)}p\gg$ and $\ll\mathcal{T}^{(\overline{m})}p\gg$ respectively.
    We say that $C$ is a $\kappa$-fold \textbf{formal cover} of $\overline{C}$ if there exists
    \begin{itemize}
        \item a sphere $\Sigma$ with marked points $(z_0,...,z_a)$,
        \item a sphere $\overline\Sigma$ with marked points $(\overline{z}_0,...,\overline{z}_{\overline{a}})$, and
        \item a $\kappa$-fold branched cover $\pi \colon \Sigma \rightarrow \overline{\Sigma}$.
    \end{itemize}
    such that
    \begin{itemize}
        \item $\pi^{-1}(\{\overline{z}_1,...,\overline{z}_{\overline{a}}\}) = \{z_1,...,z_a\}$,
        \item $\pi(z_0) = \overline{z}_0$,
        \item for each $i = 1,...,a$, $\gamma_i$ is the $\kappa_i$-fold cover of $\overline{\gamma}_j$, where $j$ is such that $\pi(z_i) = \overline{z}_j$ and $\kappa_i$ is the ramification order of $\pi$ at $z_i$, and
        \item we have $\kappa\overline{m} \geq m$, where $\kappa$ is the ramification order of $\pi$ at $z_0$.
    \end{itemize}
    A formal curve component is \textbf{simple} if it cannot be written as a nontrivial formal cover of any other formal curve component.
\end{definition}

Now we may state the definition of formal perturbation invariance.

\begin{definition}[{\cite[Definition 2.4.1]{McDuff-Siegel_unperturbed_curves}} ]\label{def:formal_perturbation_invariance}
    Let $X$ be a Liouville domain with nondegenerate contact boundary $Y$, and let $C$ be an index zero simple formal curve component in $X$ with positive asymptotics $\Gamma = (\gamma_1,...,\gamma_a)$, homology class $A \in H_2(X,\Gamma)$, and carrying a constraint $\ll T^{(m)}p\gg$ for some $m \in \mathbb{Z}_{\geq 1}$.
    We say that $C$ is \textbf{formally perturbation invariant} if there exists a generic $J_Y \in \mathcal{J}(Y)$ such that the following holds.
    Suppose that $C' \in \overline{\overline{\mathcal{F}}}_{X,A}(\Gamma)\ll T^{(m)}p\gg$ is any stable formal building satisfying
    \begin{enumerate}
        \item[(A1)] Each nonconstant component of $C'$ in $X$ is a formal cover of some formal curve component $\overline{C}'$ with $\text{ind}(\overline{C}') \geq -1$.
        \item[(A2)] Each nonconstant component of $C'$ in $Y$ is a formal cover of some formal curve component $\overline{C}'$ which is either trivial or else satisfies $\text{ind}(\overline{C}') \geq 1$.
    \end{enumerate}
    Then either:
    \begin{enumerate}
        \item[(B1)] $C'$ consists of a single component, i.e $C' = C$, or else
        \item[(B2)] $C'$ is a two-level building, with bottom level in $X$ consisting of a single component $C_X$ which is simple with index $-1$, and with top level in $Y$ represented by a union of some trivial cylinders with a simple index $1$ component $C_Y$ in $\mathbb{R}\times Y$; moreover, we require that $\mathcal{M}_Y^{J_Y}(C_Y)$ is regular and satisfies $\# \mathcal{M}_Y^{J_Y}(C_Y) / \mathbb{R} = 0.$
    \end{enumerate}
    More generally, if $C$ is any formal curve component in $X$, we say that it is formally perturbation invariant if it is a formal cover of an index zero simple formal curve component $\overline{C}$ which is formally perturbation invariant as above.
\end{definition}

\subsection{Index Computations for Formal Covers}
\begin{lemma}\label{lem:formal_cover_index}
    Let $C$ be a genus $0$ formal curve component in a completion or a symplectization, which is a $\kappa$-fold formal cover of a genus $0$ formal curve $\overline{C}$.
    Then
    $$\emph{ind}(C) = \kappa\emph{ind}(\overline{C}) + (2\kappa - 2) - 2(\kappa \overline{E}_+ - E_+) - (\kappa\overline{H}_+ - H_+) - (\kappa\overline{H}_- - H_-),$$
    where $\overline{E}_+,E_+$ are the numbers of positive ends asymptotic to elliptic Reeb orbits in $\overline{C},C$, respectively; and similarly, $\overline{H}_+,H_+,\overline{H}_-,H_-$ are the numbers of positive and negative hyperbolic orbits.
\end{lemma}

\begin{proof}
    For $\overline{C}$, denote the orbit of an asymptotic end by $\overline{\gamma}_{\pm,e/h,i}$ depending on whether it's a positive or negative end, and whether it's elliptic or hyperbolic.
    Let's write down the index of $\overline{C}$, where we leave out $c_\tau$ since it is zero in our case, where $\tau := \tau_{\text{ex}}$ as before.
    \begin{equation*}
    \begin{aligned}
        \text{ind}(\overline{C}) &= -\chi(\overline{C}) + \sum_{i =1}^{\overline{E}_+}CZ_\tau(\overline{\gamma}_{+,e,i}) + \sum_{i =1}^{\overline{H}_+}CZ_\tau(\overline{\gamma}_{+,h,i}) \\
        &\quad \quad - \sum_{i =1}^{\overline{E}_-}CZ_\tau(\overline{\gamma}_{-,e,i}) - \sum_{i =1}^{\overline{H}_-}CZ_\tau(\overline{\gamma}_{-,h,i})
    \end{aligned}
    \end{equation*}
    Recall that a formal cover consists of a branched cover of spheres with marked points $\pi \colon \Sigma \mapsto \overline{\Sigma}$.
    Now, for $C$, denote the orbit of an asymptotic end by $\gamma_{\pm,e/h,i,j}$, to indicate that its corresponding marked point gets sent to the marked point of $\overline{\gamma}_{\pm,e/h,i}$ by $\pi$.
    Let $\kappa_{\pm,e/h,i,j}$ be the order of this cover of orbits, which is the ramification order of $\pi$ at the marked point in $\Sigma$.
    Let $K_{\pm,e/h,i}$ denote the number of marked points in $\Sigma$ being sent to the marked point corresponding to $\overline{\gamma}_{\pm,e/h,i}$ in $\overline{\Sigma}$.
    Then, since the preimage of the marked points in $\overline{\Sigma}$ is exactly the set of marked points in $\Sigma$, we have
    $$\sum_{j = 1}^{K_{\pm,e/h,i}} \kappa_{\pm,e/h,i,j} = \kappa,$$
    for every possible tuple $\pm,e/h,i$.
    Now, we write down the index of $C$, where we denote the $\kappa$ fold cover of an orbit $\overline{\gamma}$ by $\kappa\overline{\gamma}$.
    \begin{equation*}
    \begin{aligned}
        \text{ind}(C) &= -\chi(C) \\
        & \quad + \sum_{i =1}^{\overline{E}_+} \sum_{j = 1}^{K_{+,e,i}} CZ_\tau(\kappa_{+,e,i,j}\overline{\gamma}_{+,e,i}) + \sum_{i =1}^{\overline{H}_+} \sum_{j = 1}^{K_{+,h,i}} CZ_\tau(\kappa_{+,h,i,j}\overline{\gamma}_{+,h,i})\\
        & \quad - \sum_{i =1}^{\overline{E}_-} \sum_{j = 1}^{K_{-,e,i}} CZ_\tau(\kappa_{-,e,i,j}\overline{\gamma}_{-,e,i}) - \sum_{i =1}^{\overline{H}_-} \sum_{j = 1}^{K_{-,h,i}} CZ_\tau(\kappa_{-,h,i,j}\overline{\gamma}_{-,h,i})
    \end{aligned}
    \end{equation*}
    We can figure out the Conley-Zehnder indices of covers of elliptic and hyperbolic orbits in terms of the indices of their underlying orbits.
    \begin{equation*}
    \begin{aligned}
        CZ_\tau(\kappa e_{x,y}) &= 2\kappa x + 2\kappa y + 1 = \kappa (CZ_\tau(e_{x,y}) - 1) + 1\\
        CZ_\tau(\kappa h_{x,y}) &= 2\kappa x + 2\kappa y = \kappa CZ_\tau(h_{x,y})\\
    \end{aligned}
    \end{equation*}
    Thus, we can simplify the index of $C$.
    \begin{equation*}
    \begin{aligned}
        \text{ind}(C) &= -\chi(C)\\
        & \quad + \sum_{i =1}^{\overline{E}_+} \left[\kappa(CZ_\tau(\overline{\gamma}_{+,e,i}) - 1) + K_{+,e,i}\right] + \sum_{i =1}^{\overline{H}_+} \kappa CZ_\tau(\overline{\gamma}_{+,h,i})\\
        & \quad - \sum_{i =1}^{\overline{E}_-} \left[\kappa(CZ_\tau(\overline{\gamma}_{-,e,i}) - 1) + K_{-,e,i}\right] - \sum_{i =1}^{\overline{H}_-} \kappa CZ_\tau(\overline{\gamma}_{-,h,i})\\
    \end{aligned}
    \end{equation*}
    To further simplify, let's examine the Euler characteristic, and denote $E_+$, $E_-$, $H_+$, $H_-$ to be the numbers of elliptic/hyperbolic, positive/negative ends of $C$.
    For punctured spheres, the Euler characteristic is just $2$ minus the number of punctures, so we can check that
    \begin{equation*}
    \begin{aligned}
        \chi(C) &= \kappa\chi(\overline{C}) - (2\kappa - 2) + (\kappa \overline{E}_+ - E_+) + (\kappa \overline{H}_+ - H_+)\\
        &\quad + (\kappa \overline{E}_- - E_-) + (\kappa \overline{H}_- - H_-).
    \end{aligned}
    \end{equation*}
    Then plugging in for $\chi$ and simplifying the sum terms, we get
    \begin{equation*}
    \begin{aligned}
        \text{ind}(C) &= -[\kappa\chi(\overline{C}) - (2\kappa - 2) + (\kappa \overline{E}_+ - E_+) + (\kappa \overline{H}_+ - H_+)\\
        &\quad\quad\quad + (\kappa \overline{E}_- - E_-) + (\kappa \overline{H}_- - H_-)]\\
        &\quad + \kappa \Bigg[\sum_{i =1}^{\overline{E}_+} CZ_\tau(\overline{\gamma}_{+,e,i}) + \sum_{i =1}^{\overline{H}_+} CZ_\tau(\overline{\gamma}_{+,h,i})\\
        &\quad\quad\quad - \sum_{i =1}^{\overline{E}_-} CZ_\tau(\overline{\gamma}_{-,e,i}) - \sum_{i =1}^{\overline{H}_-} CZ_\tau(\overline{\gamma}_{-,h,i})\Bigg]\\
        &\quad - (\kappa \overline{E}_+ - E_+) + (\kappa\overline{E}_- - E_-)\\
        &= \kappa \text{ind}(\overline{C}) + (2\kappa - 2) - 2(\kappa \overline{E}_+ - E_+) - (\kappa\overline{H}_+ - H_+) - (\kappa\overline{H}_- - H_-)
    \end{aligned}
    \end{equation*}
\end{proof}

\begin{lemma}\label{lem:RH_bound_cover_index}
    Let $C$ be a genus $0$ formal curve component in a completion or a symplectization, which is a $\kappa$-fold formal cover of a genus $0$ formal curve $\overline{C}$.
    Then
    $$\emph{ind}(C) \geq \kappa\emph{ind}\overline{C} - (\kappa\overline{E}_+ - E_+) + (\kappa\overline{E}_- - E_-).$$ 
\end{lemma}

\begin{proof}
    Recall that this cover comes from a $\kappa$-fold cover of spheres $\Sigma \rightarrow \overline{\Sigma}$, with marked points corresponding to asymptotic ends of $C$.
    Denote by $r_p$ the ramification order of a point $p$ in the sphere $\Sigma$, and let $\mathcal{R}$ be the set of all ramification points, while $\mathcal{R}_{\text{int}}$ denotes the set of ramification points other than the marked points (i.e. those in the ``interior" of $C$).
    By the Riemann Hurwitz formula, we have
    $$\sum_{p\in \mathcal{R}}(r_p - 1) = 2\kappa - 2.$$
    We see that $\kappa\overline{E}_+ - E_+ = \sum_p (r_p - 1)$, where the sum is taken over positive elliptic ends.
    This holds similarly for $E_-,H_{\pm}$.
    Thus,
    $$(\kappa\overline{E}_+ - E_+) + (\kappa\overline{H}_+ - H_+) + (\kappa\overline{H}_- - H_-) = (2\kappa - 2) - (\kappa\overline{E}_- - E_-) - \sum_{p \in \mathcal{R}_{\text{int}}} (r_p - 1). $$
    Then we can plug in to our index formula from Lemma \ref{lem:formal_cover_index} to get
    \begin{equation*}
    \begin{aligned}
    \text{ind}(C) &= \kappa\text{ind}(\overline{C}) - (\kappa\overline{E}_+ - E_+) + (\kappa\overline{E}_- - E_-) + \sum_{p \in \mathcal{R}_{\text{int}}} (r_p - 1)\\
    &\geq \kappa\text{ind}(\overline{C}) - (\kappa\overline{E}_+ - E_+) + (\kappa\overline{E}_- - E_-).
    \end{aligned}
    \end{equation*}
\end{proof}

\begin{lemma}\label{lem:tangency_constraint_index_inequality}
    Now, let $C$ be a $\kappa$-fold formal cover of $\overline{C}$ in $\widetilde{X}_\Omega$, but there is a $\ll \mathcal{T}^{(m)}p\gg$ tangency constraint on $C$ and a $\ll \mathcal{T}^{(\overline{m})}p\gg$ tangency constraint on $\overline{C}$.
    Let $B_{\pm,e/h}$ denote the excess branching for positive/negative elliptic/hyperbolic orbits (e.g. $B_{+,e} = \kappa \overline{E}_+ - E_+$).
    Then
    $$\emph{ind}(C) \geq \kappa \emph{ind}(\overline{C}) + \emph{max}\{(2\kappa - 2) - 2B_{+,e} - B_{+,h}, \quad B_{+,h}\}.$$
    In particular, since $B_{+,h} \geq 0$, we have $\text{ind}(C) \geq \kappa \text{ind}(\overline{C})$.
\end{lemma}

\begin{proof}
    First, by the index formula in Lemma \ref{lem:formal_cover_index}, and since the index is decreased by twice the value of the tangency constraint, we get
    \begin{equation*}
    \begin{aligned}
        \text{ind}(C) &= \kappa\text{ind}(\overline{C}) + (2\kappa - 2) -2B_{+,e} - B_{+,h} + 2(\kappa \overline{m} - m)\\
        & \geq \kappa\text{ind}(\overline{C}) + (2\kappa - 2) -2B_{+,e} - B_{+,h},
    \end{aligned}
    \end{equation*}
    since by the definition of formal cover, $\kappa\overline{m} \geq m$.
    Note too, that we are talking about a curve in $\widetilde{X}_\Omega$ rather than a curve component in a building, so there are no negative ends.

    For the other part of the innequality, we will use Riemann Hurwitz, which says that $2\kappa - 2 \geq B_{+,e} + B_{+,h} + B$, where $B$ is the excess branching at the point $p$ in $C$ where the tangency constraint is satisfied.
    In other words, the ramification order of the cover at $p$ is $B + 1$.
    By the ramification of formal covers, we must have $(B+1)\overline{m} \geq m$.
    Also using $\overline{m}\geq 1$, we get
    \begin{equation*}
    \begin{aligned}
        \text{ind}(C) &= \kappa\text{ind}(\overline{C}) + (2\kappa - 2) -2B_{+,e} - B_{+,h} + 2(\kappa \overline{m} - m)\\
        & \geq \kappa\text{ind}(\overline{C}) + (2\kappa - 2) -2B_{+,e} - B_{+,h} + 2(\kappa - (B+1))\\
        & = \kappa\text{ind}(\overline{C}) + 2[(2\kappa - 2) -B_{+,e} - B] - B_{+,h}\\
        & \geq \kappa\text{ind}(\overline{C}) + B_{+,h}.
    \end{aligned}
    \end{equation*}
\end{proof}

\subsection{Combinatorial Minimizers are Formally Perturbation Invariant}

\begin{lemma}\label{lem:A1_implies_ind_C_geq_0}
    Let $C$ be a genus zero formal building in $\widetilde{X}_\Omega$ and symplectization levels $\mathbb{R} \times \partial\widetilde{X}_\Omega$, with one unmatched negative end.
    Assume that every component of $C$ in a symplectization level formally covers a simple component which is either a trivial cylinder or has index no less than $1$ (this is condition (A1) in the definition of formal perturbation invariance).
    Then, $\text{ind}(C) \geq 0$, with equality if and only if $C$ is made up of trivial cylinders.
\end{lemma}

\begin{proof}
    We will split the components of $C$ into those which have a negative end, and those which do not.
    Let $C_1,...,C_s$ be the components with at least one negative end, and let $b_1,...,b_s$ be the number of negative ends on each component, respectively.
    We show below in Lemma \ref{lem:num_components_no_neg_ends} that there are at least $\sum_1^s (b_i-1)$ components with no negative ends.
    The argument is by induction on the number of components and uses the absence of genus to decompose the building into disjoint sub-buildings in the inductive step.
    
    Now, because the index of a curve with no negative end is at least $2$, we see that
    $$\text{ind}(C) \geq \sum_{i=1}^s \left[\text{ind}(C_i) + 2(b_i - 1)\right].$$
    We will show for each $C_i$ that $\text{ind}(C_i) + 2(b_i - 1) \geq 0$, with equality only when $C_i$ is a trivial cylinder, which completes the proof.
    Indeed, let $C_i$ be a $\kappa$-fold cover of some underlying simple component $\overline{C}_i$.
    By assumption, either $\overline{C}_i$ is a trivial cylinder or has $\text{ind}(\overline{C}_i) \geq 1$.
    By Lemma \ref{lem:RH_bound_cover_index}, 
    \begin{equation*}
    \begin{aligned}
        \text{ind}(C_i) + 2(b-1) &\geq \kappa\text{ind}(\overline{C}_i) - (\kappa \overline{E}_+ - E_+) + (\kappa\overline{E}_- - E_-) + 2(b-1)\\
        & = \kappa\text{ind}(\overline{C}_i) - (\kappa \overline{E}_+ - E_+) + \kappa\overline{E}_- + E_- + 2H_- - 2
    \end{aligned}
    \end{equation*}
    where $E_{\pm},H_{\pm}$ are the numbers of positive/negative ends asymptotic to elliptic or hyperbolic Reeb orbits in $C_i$, and $\overline{E}_{\pm},\overline{H}_{\pm}$ are defined similarly for $\overline{C}_i$.
    First, if $\overline{C}_i$ is a trivial cylinder with elliptic ends, then $\text{ind}(\overline{C}_i) = 0, \overline{E}_+ = \overline{E}_- = 1,$ and $H_- = 0$, giving
    $$\text{ind}(C_i) + 2(b-1) \geq \kappa (0) - \kappa + E_+ + \kappa + E_- + 2(0) - 2 \geq 0,$$
    with equality only when $E_+ = E_- = 1$, i.e. when $C_i$ is a trivial cylinder.
    Similarly, if $\overline{C}_i$ is a trivial cylinder with hyperbolic ends, then
    $$\text{ind}(C_i) = \kappa (0) + (2\kappa - 2) - (\kappa - H_+) - (\kappa - H_-) = H_+ + H_- - 2,$$
    $$\Rightarrow\;\text{ind}(C_i) + 2(b_i - 1) = H_+ + H_- - 2 + 2(H_- - 1) = H_+ + 3H_- - 4 \geq 0,$$
    with equality only when $H_+ = H_- = 1$.
    Secondly, we consider the case when $\text{ind}(\overline{C}_i) \geq 1$.
    Here we will use the fact from Rieman Hurwitz that $\kappa\overline{E}_+ - E_+ \leq 2\kappa - 2.$
    This gives
    \begin{equation*}
    \begin{aligned}
        \text{ind}(C_i) + 2(b-1) &\geq \kappa\text{ind}(\overline{C}_i) - (\kappa \overline{E}_+ - E_+) + \kappa\overline{E}_- + E_- + 2H_- - 2\\
        &\geq \kappa - (2\kappa - 2) + \kappa\overline{E}_- + E_- + 2H_- - 2\\
        & = \kappa(\overline{E}_- - 1) + E_- + 2H_-
    \end{aligned}
    \end{equation*}
    Hence, this quantity is positive if $\overline{E}_- \geq 1$.
    Otherwise, there are no elliptic negative ends.
    Then we can assume there is at least one hyperbolic negative end, for if not, then our index is positive anyway.
    And we condition on the ramification of the positive elliptic ends as follows.
    If $\kappa\overline{E}_+ - E_+ \leq \kappa - 1$, then
    $$\text{ind}(C_i) + 2(b-1) \geq \kappa - (\kappa - 1) + 2H_- - 2 \geq 1.$$
    Now we check when $\kappa\overline{E}_+ - E_+ > \kappa - 1$.
    Again by Riemann Hurwitz we get the following useful innequality:
    $$H_- \geq (\kappa\overline{E}_+ - E_+) + (\kappa\overline{H}_+ - H_+) + \kappa\overline{H}_- - (2\kappa - 2)$$
    This implies that
    \begin{equation*}
    \begin{aligned}
        \text{ind}(C_i) + 2(b - 1) &\geq \kappa\text{ind}(\overline{C}_i) + (\kappa\overline{E}_+ - E_+) \\
        &\quad \quad + 2( (\kappa\overline{H}_+ - H_+) + \kappa\overline{H}_- - (2\kappa - 2)) - 2\\
        &\geq \kappa + (\kappa - 1) + 2(\kappa\overline{H}_+ - H_+) + 2\kappa - 4 \kappa  + 2\\
        & \geq 1
    \end{aligned}
    \end{equation*}
\end{proof}

\begin{lemma}\label{lem:num_components_no_neg_ends}
    Let $C$ be a building as in Lemma \ref{lem:A1_implies_ind_C_geq_0}, with $C_1,...,C_s$ the components with at least one negative end, and $b_1,...,b_s$ the number of negative ends on those components.
    Then there are at least $\sum_1^s (b_i - 1)$ components with no negative ends.
\end{lemma}

\begin{proof}
    We will prove this by induction on the number of components.
    If there is only one component, then it must have exactly one negative end, in which case $b_1 = 1$ and there are $b_1 - 1 = 0$ components with no negative ends.
    For the induction, consider a building as above with $s$ components.
    Without loss of generality, take $C_1$ to be a component with no matched positive ends.
    Then $b_1 - 1$ of these negative ends are matched with a building which has no unmatched negative ends, and the last end is either unmatched itself or matched with a building with one unmatched negative end.
    And the genus being zero ensures that these buildings are disjoint.
    Call these buildings $D_1,...,D_{b_1}$, and take $D_1$ to be the building with an unmatched negative end.
    Denote by $D_{i,j}$ the components in building $D_i$ with negative ends, and let $d_{i,j}$ be the number of negative ends of $D_{i,j}$.
    Note that the set of all $d_{i,j}$ is exactly $b_2,...,b_s$.

    Now, by induction the number of components with no negative ends in $D_1$ is at least $\sum_j (d_{1,j} - 1)$.
    Similarly, by induction the number of components with no negative ends in $D_i$ for $i = 2,...,b_1$ is at least $\sum_j (d_{i,j} - 1) + 1$.
    Indeed, this is because induction applies if you give one of $D_i$'s components with no negative end an unmatched negative end (which changes the count by 1).
    So, the total number of components with no negative end must be at least
    \begin{equation*}
    \begin{aligned}
        &\sum_j (d_{1,j} - 1) + \sum_{i=2}^{b_1} \left(\sum_j (d_{i,j} - 1) + 1\right)\\
        =& \sum_{i=2}^s (b_i - 1) + \sum_{i=2}^{b_1} 1\\
        =& \sum_{i=1}^s (b_i - 1).
    \end{aligned}
    \end{equation*}
    
\end{proof}

\begin{remark}
    In the following arguments, we denote the formal words of Reeb orbits in a fully rounded and perturbed convex toric domain $\widetilde{X}_\Omega$ by $w$, rather than $\Gamma$, to emphasize that in this setting $w$ is a formal word composed of letters of the form $e_{i,j}$ or $h_{i,j}$, corresponding to the specific elliptic and hyperbolic orbits found after perturbation.
\end{remark}

To clarify notation, for a collection of Reeb orbits $\Gamma = \{e_{i_s,j_s}\}_s \cup \{h_{i_t,j_t}\}_t$ in a fully rounded and perturbed convex toric domain $\widetilde{X}_\Omega$, we write $P(\Gamma)$ for the corresponding collection of ordered pairs $\{(i_s,j_s)\}_s \cup \{(i_t,j_t)\}_t$.
Recall that a word of ordered pairs $P$ is a member of $\mathcal{P}_{k,\ell}$ if $P$ is weakly permissible, has index $2k$, and has at most $\ell$ ordered pairs.
We will also say that $P$ is a $\kappa$-fold cover of $\overline{P}$ if there exists a mapping from $P$ to $\overline{P}$ such that each element $(\overline{i},\overline{j}) \in \overline{P}$ is mapped to by elements in $P$ whose sum is $(\kappa\overline{i},\kappa\overline{j})$.
For example $\{(2,0),(0,1),(0,1)\}$ is a $2$-fold cover of $\{(1,0),(0,1)\}$ by this definition.
This is compatible in the sense that if $C$ is a $\kappa$-fold cover of $\overline{C}$, with positive asymptotics $\Gamma$ and $\overline{\Gamma}$ respectively, then $P := P(\Gamma)$ is a $\kappa$-fold cover of $\overline{P} := P(\overline{\Gamma})$.
Note also that the action scales: $||(\kappa\overline{i},\kappa\overline{j})||_\Omega^* = \kappa||(\overline{i},\overline{j})||_\Omega^*$.

\begin{lemma}\label{lem:w_min_implies_simple_w_min}
    Fix $k,\ell$ and a fully rounded convex toric domain $\widetilde{X}_\Omega$.
    Let $C$ be a $\kappa$-fold formal cover of simple curve $\overline{C}$, where $C$ is asymptotic to a word of Reeb orbits
    $$\Gamma = e_{i_1,j_1} \times \cdots \times e_{i_q,j_q}$$
    of index $2k$, and $\overline{C}$ is asymptotic to a word of Reeb orbits
    $$\overline{\Gamma} = e_{\overline{i}_1,\overline{j}_1} \times \cdots \times e_{\overline{i}_{\overline{q}},\overline{j}_{\overline{q}}}.$$
    of index $2\overline{k}$.
    Also let $C$ satisfy a $\ll T^{(k)} p\gg$ tangency constraint, and let $\overline{C}$ satisfy a $\ll T^{(\overline{k})} p\gg$ tangency constraint.
    If $P := P(\Gamma)$ minimizes action over $\mathcal{P}_{k,\ell}$, then $\overline{P} := P(\overline{\Gamma})$ minimizes action over $\mathcal{P}_{\overline{k},\overline{q}}$.
\end{lemma}

\begin{proof}
    Assume there is some other word $\overline{P}'$ of index $2\overline{k}$ and length $\overline{q}' \leq \overline{q}$ with action less than the action of $\overline{P}$.
    We will create $P'$, a $\kappa$-fold cover of $\overline{P}'$, with as many positive ends as we can have, up to $q$.
    We will see that $P'$ has less action than $P$, but just as high an index, a contradiction.

    Indeed, lets examine the index calculation for covers.
    We have
    $$k = \sum_1^q (i_s + j_s) + q - 1 \quad \text{and} \quad \overline{k} = \sum_1^{\overline{q}}(\overline{i}_s + \overline{j}_s) + \overline{q} - 1.$$
    Let $\alpha$ and $\overline{\alpha}$ be the sums in these formulas; then $\alpha = \kappa\overline{\alpha}$.
    This is because at each branch point corresponding to an underlying orbit, there are some number of orbits in the cover, and their multiple-cover orders partition $\kappa$.
    I.e.
    $$\alpha = \sum_1^q (i_s + j_s) = \sum_1^q \kappa_s(\overline{i}_s + \overline{j}_s) = \kappa \sum_1^{\overline{q}} (\overline{i}_s + \overline{j}_s) = \kappa\overline{\alpha},$$
    where the $\kappa_s$ are the orbit multiple-cover orders.

    Now, the index of $\overline{P}'$ is
    $$\sum_1^{\overline{q}'}(\overline{i}_s' + \overline{j}_s') + \overline{q}' - 1 = \overline{k}.$$
    Letting $\overline{\alpha}'$ be the sum, we see that $\overline{\alpha}' = \overline{\alpha} + (\overline{q} - \overline{q}')$, by substituting for $\overline{k}$.
    Then we create a formal cover of $\overline{P}'$, called $P'$, and we want as many ends as we can have without going over $q$.
    Thus, $P'$ will either have length $q$ or be totally unramified, i.e. each underlying orbit will have $\kappa$ separate orbits mapping to it in the cover.

    Let's compute the index $k'$ of $P'$.
    In the first case, where the length of $P'$ is $q$, we have
    $$k' = \sum_1^q (i_s' + j_s') + q - 1 = \kappa \overline{\alpha}' + q - 1 \geq \kappa\overline{\alpha} + q - 1 = k.$$
    And in the second case, where $P'$ is totally unramified, we have
    $$k' = \kappa\overline{\alpha} + \kappa(\overline{q} - \overline{q}') + \kappa\overline{q}' - 1 = \kappa\overline{\alpha} + \kappa\overline{q} - 1 \geq \kappa\overline{\alpha} + q - 1 = k.$$
    In both cases, $k' \geq k$, but the action of $P'$ is $\kappa$ times the action of $\overline{P}'$, and so it is less than the action of $P$, contradicting the minimality of $P$ over $\mathcal{P}_{\geq k,\ell}$, which is implied by its minimality over $\mathcal{P}_{k,\ell}$ and Proposition \ref{prop:bound_over_P_k}.
    Hence the result holds.
    
\end{proof}

\begin{lemma}\label{lem:MB_two_cylinders}
    Let $\widetilde{X}_\Omega$ be a fully rounded and perturbed 4-dimensional convex toric domain, $k$ a positive integer, and $i,j$ relatively prime.
    Then there exists a generic $J$ in $\mathcal{J}(\widetilde{X}_\Omega)$ such that the only $J$-holomorphic curves asymptotic positively to $e_{ki,kj}$ and negatively to $h_{ki,kj}$ are the $k$-fold covers of the two simple cylinders along the gradient flow trajectories from the Morse-Bott perturbation, and such that these cylinders are regular and count with opposite sign.
\end{lemma}

\begin{proof}
    This is true by a Morse-Bott gluing argument which can be found in \cite{colin2023yao_appendix}: see Theorem 4.4.3 and the appendix, written jointly with Yuan Yao.
    The argument uses the fact that as we take our perturbation smaller and smaller, a $J$-holomorphic curve that is positively asymptotic to $e_{ki,kj}$ and negatively asymptotic to $h_{ki,kj}$ must approach one of the two gradient flow trajectories along the Morse-Bott Torus.
    Then by uniqueness of gluing, there can only be one sequence, up to translation.
    Thus the only possible sequences consist of multiple covers of the simple curves asymptotic to $e_{i,j}$ and $h_{i,j}$.
    The gluing argument holds in this multiply covered case because the linearized Cauchy-Riemann operator is still index-1 and surjective.
    For more details and arguments involving several-level cascades, see \cite{yao2022cascadesjholomorphiccurves}.

    A simpler argument based on intersection theory and the relative adjunction formula may rule out extraneous curves, but the initial writhe bound is not enough, since such curves could have trivial linking with the multiple covers even though they differ in their eigenstate decomposition via Siefring \cite{Siefring_intersection_theory_punctured}.
    The issue is that the space of eigenstates for a given winding number is 2.    
\end{proof}

\begin{proposition}\label{prop:formal_perturbation_invariance}
    Let $\widetilde{X}_\Omega$ be a fully rounded and perturbed convex toric domain.
    For $k,\ell \geq 1$, let $P_{\emph{min}}$ be an action minimizer over $\mathcal{P}_{k,\ell}$.
    Then, the formal curve component $C$ in $\widetilde{X}_\Omega$ positively asymptotic to the word $\Gamma_{\emph{min}}$ of elliptic Reeb orbits corresponding to $P_{\emph{min}}$ is formally perturbation invariant, with respect to generic $J_{\partial\widetilde{X}_\Omega}$ as in Lemma \ref{lem:MB_two_cylinders}.
\end{proposition}

\begin{proof}
    By the definition of formal perturbation invariance, we must show that $C$ is the formal cover of an index-0 simple formal curve $\overline{C}$ that is formally perturbation invariant with respect to $J_{\partial\widetilde{X}_\Omega}$.
    If $C$ is not already simple, then it is a formal cover of a simple $\overline{C}$, whose positive asymptotics we call $\overline{\Gamma}_{\text{min}}$ with index $2\overline{k}$ and length $\overline{q}$, and we impose a $\overline{k}$ tangency constraint to make $\overline{C}$ index-0.
    We see by Lemma \ref{lem:w_min_implies_simple_w_min}, that the corresponding $\overline{P}_{\text{min}} := P(\overline{\Gamma}_{\text{min}})$ minimizes action over $\mathcal{P}_{\overline{k},\overline{q}}$.
    To show perturbation invariance of $\overline{C}$, we suppose that $C' \in \overline{\overline{\mathcal{F}}}_{X,A}(\Gamma)\ll T^{(\overline{k})}\overline{p}\gg$ is a stable formal building satisfying:
    \begin{enumerate}
        \item[(A1)] each nontrivial curve component in the completion of $\widetilde{X}_\Omega$ is the formal cover of an underlying simple component with index no less than $-1$, and
        \item[(A2)] each nontrivial curve component in a symplectization level is the formal cover of an underlying simple component which is either trivial or has positive index.
    \end{enumerate}
    And we must prove either
    \begin{enumerate}
        \item[(B1)] $C'$ consists only of the curve component in the completion, i.e. $C' = \overline{C}$, or
        \item[(B2)] $C'$ has a simple main component in the completion, $C_0$, with index $-1$.
        There is one symplectization level containing a low energy cylinder $C_Y$ of index $1$, and the rest of the components are trivial cylinders.
        Moreover, $\mathcal{M}(C_Y)$ is regular and $\#\mathcal{M}(C_Y)/\mathbb{R} = 0$.
    \end{enumerate}
    We can decompose the building into the following sub-buildings: $C_0$ for the main component satisfying the tangency constraint in the bottom level completion, and $C_1,...,C_r$ for the sub-buildings that match up with each of $C_0$'s $r$ positive ends.
    Note that $C_i$ must only match up with $C_0$ at one negative end, for otherwise $C'$ would have genus.
    Thus, Lemma \ref{lem:A1_implies_ind_C_geq_0} and condition (A1) imply that for each $i$, $\text{ind}(C_i) \geq 0$ with equality only if $C_i$ is made of trivial cylinders.
    Note that if there is only one level, then (B1) would hold, and we'd be done.

    Now, if $\text{ind}(C_i) = 1$, then $C_i$ must consist of a low-energy cylinder (potentially a cover) and trivial cylinders in other levels.
    Indeed, since the index is odd and the top ends are elliptic, we must have that the negative end is hyperbolic, say $h_{i,j}$.
    Then, since the index is $1$, there is exactly one nontrivial component, which has index $1$.
    This is because each nontrivial component has positive index by Lemma \ref{lem:A1_implies_ind_C_geq_0}.
    Then $\overline{\Gamma}_{\text{min}}$ contains the positive ends of this component. 
    Assume for contradiction that this is not the set $\{e_{i,j}\}$.
    Let $\widetilde{\Gamma}_{\text{min}}$ be the word which replaces those positive ends with $e_{i,j}$.
    Note that this is a strict decrease in $\widetilde{\mathcal{A}}$ by Lemma \ref{lem:big_fully_rounding}(d), which describes low-energy cylinders.
    But this is a contradiction, since $\overline{\Gamma}_{\text{min}}$ minimizes $\widetilde{\mathcal{A}}$ with respect to words of no greater length, and $\text{len}(\widetilde{\Gamma}_{\text{min}}) \leq \text{len}(\overline{\Gamma}_{\text{min}})$.
    Thus, the nontrivial component in this subbuilding is a low-energy cylinder from $e_{i,j}$ to $h_{i,j}$.

    Now, let's take a look at the main component $C_0$, which is a $\kappa$-fold formal cover of a simple component $\overline{C}_0$, which has $\text{ind}(\overline{C}_0) \geq -1$.
    As before, let $H_+$ and $\overline{H}_+$ be the number of hyperbolic positive ends in $C_0$ and $\overline{C}_0$ respectively.
    If $\overline{H}_+ = 0$, then $\overline{C}_0$ must have elliptic ends and therefore $\text{ind}(\overline{C}_0) \geq 0$ for parity reasons.
    Then since the index of $C'$ is zero, this forces the index of all components to be 0, contradicting the stability of the symplectization levels.

    So assume $\overline{H}_+ \geq 1$.
    If there is no ramification at the positive ends, i.e. $B_{+,e} = B_{+,h} = 0$, then by Lemma \ref{lem:tangency_constraint_index_inequality}, we have $\text{ind}(C_0) \geq \kappa \text{ind}(\overline{C}_0) + 2\kappa - 2$.
    No ramification also implies that $H_+ = \kappa \overline{H}_+ \geq \kappa$.
    Hence,
    $$0 = \text{ind}(C_0) + \sum_i \text{ind}(C_i) \geq \kappa \text{ind}(\overline{C}_0) + 2\kappa - 2 + \kappa \geq 2\kappa - 2,$$
    implying that $\kappa$ must be $1$.
    Thus $\text{ind}(C_0) \geq -1$, and $\sum_i \text{ind}(C_i) \geq 1$, so the only way this sums to $0$ is if $C'$ satisfies (B2).
    Note that lemma \ref{lem:MB_two_cylinders} gives us the regularity and 0 signed count of the moduli space of the low-energy cylinder required by (B2).

    Lastly, consider when there is ramification at some of the positive ends.
    If there is ramifcation at a positive hyperbolic end, then the sub-building above it must be a column of cylinders, one of which is a low energy cylinder, with the rest trivial.
    Otherwise, its index would be greater than 1, in which case $\sum_i \text{ind}(C_i) \geq H_+ + 1$, implying
    $$0 = \text{ind}(C_0) + \sum_i \text{ind}(C_i) \geq \kappa \text{ind}(\overline{C}_0) + B_{+,h} + H_+ + 1 \geq -\kappa + \kappa \overline{H}_+ + 1 \geq 1,$$
    a contradiction.
    Here, the first inequality follows Lemma \ref{lem:tangency_constraint_index_inequality}, which utilizes Riemann-Hurwitz.
    The same calculation shows that if we have ramification on the elliptic ends, we must have only trivial cylinders above it, for otherwise $\sum_i \text{ind}(C_i) \geq H_+ + 2$.
    So the breaking picture looks like a main component in the completion of index $-\kappa$ and cylinders above this with $\kappa$ of them low-energy, and the rest trivial.
    
    We would like to avoid $\kappa > 1$ because that would imply the main component is not regular, however this will not be an issue.
    If $\kappa > 1$, then the main component $C_0$ is a nontrivial formal cover by assumption (A1), thus implying that the original curve we're considering, $\overline{C}$, is also a formal cover.
    This is because the word $\overline{\Gamma}_{\text{min}}$ is the word of positive orbits of $C_0$, except just replacing any $h_{i,j}$ with an $e_{i,j}$.
    But this contradicts our assumption that $\overline{C}$ is simple, so $\kappa$ must be $1$.
    Thus, in these cases we see breaking of the form (B2), and the low energy cylinder is regular by an argument below, and it has zero signed count by Lemma \ref{lem:MB_two_cylinders}, completing the proof.

    To see regularity, by Riemann-Hurwitz, a multiple cover of a low energy cylinder has no branching on the interior.
    Indeed, if we have a cylinder from $h_{\kappa_0i,\kappa_0j}$ to $e_{\kappa_0i,\kappa_0j}$, then total branching of the corresponding sphere mapping is given by $(\kappa_0 - 1) + (\kappa_0 - 1) = 2\kappa_0 - 2$.
    Thus any more branching would violate Riemann-Hurwitz.

    Then by Wendl \cite{Wendl_automatic_transversality}, these unbranched multiple covers are still regular.
    Indeed, $u$ is regular if
    $$\text{ind}(u;\textbf{c}) > 2g + \#\Gamma_0(\textbf{c}) + \#\pi_0(\partial\Sigma) - 2 + 2Z(du),$$
    where $\textbf{c}$ specifies the asymptotics of $u$, $g$ is the genus of $\Sigma$, $\#\Gamma_0$ is the number of ends with even Conley-Zehnder index, $\#\pi_0(\partial\Sigma)$ is the number of connected components in the boundary of $\Sigma$, which in our case is $0$, and $Z(du)$ depends on the orders of critical points and is therefore $0$ in unbranched multiple covers.
    Thus, since we have $\#\Gamma_0(\textbf{c}) = 1$ for a multiple cover of a low energy cylinder, it is regular.
    
\end{proof}

\section{Existence of Curves}\label{sec:existence_of_curves}
In this section we show the existence of curves needed in the proof of Theorem \ref{thm:upper_bound}.
Recall from the proof that we start with a combinatorial minimizer and show the existence of a curve with corresponding asymptotics.
With the refinement on number of positive ends, there are more combinatorial minimizers and therefore more curves which require proof of existence.
What's left to show is the existence of curves positively asymptotic to $e_{i,j}\times e_{i,j}$ in the case that $\{(i,j),(i,j)\}$ is a minimizer in $\mathcal{P}_{2i+2j+1,2}$.
Note that we must find this curve for all $J \in \mathcal{J}(\widetilde{X}_\Omega; D)$, but relying on the formal perturbation invariance of $e_{i,j} \times e_{i,j}$, it suffices to prove existence for a specific $J$.
The strategy is to embed $\widetilde{X}_\Omega$ into an ellipsoid, initialize a curve with well-chosen asymptotics in the ellipsoid, stretch the neck along the boundary $\partial \widetilde{X}_\Omega$, and show that the bottom level of the resulting psuedoholomorphic building contains the desired curve.
We start with a lemma about the actions of orbits on the ellipsoid, followed by the neck-stretching result.

\begin{lemma}\label{lem:ellipsoid_min_action_orbits}
    Let $i,j$ be positive integers with $j > i$.
    Consider the ellipsoid $E(j', i)$ where $j'$ is slightly greater than $j$ so that $j'$ and $i$ are rationally independent.
    Let $\eta_S$ be the short Reeb orbit of action $i$ and let $\eta_L$ be the long Reeb orbit of action $j'$.
    Then $\eta_L^i \times \eta_L^i$ has Fredholm index $2(1 + 2i + 2j)$ and has minimal action among orbit sets of the same index and at most two orbits.
\end{lemma}

\begin{proof}
    First, we compute the Fredholm index of $\eta_L^i \times \eta_L^i$.
    For this we need to order the orbits by action, labeling them $\mathfrak{o}_1,\mathfrak{o}_2,...$.
    We see the following actions $\mathcal{A}(\eta_L^i) = j'i$ and $\mathcal{A}(\eta_S^j) = ij$, and since $j'$ is only slightly larger than $j$, we can say that the orbits with less action are exactly
    $$\eta_S^1,...,\eta_S^{j-1},\eta_L^1,...,\eta_L^{i-1}.$$
    Hence
    $$\eta_S^j = \mathfrak{o}_{i+j-1}, \quad\quad \eta_L^i = \mathfrak{o}_{i+j}.$$
    Now, the Fredholm index corresponding to a given orbit set is given by
    \begin{equation*}
    \begin{aligned}
        \text{ind}(\mathfrak{o}_{i_1} \times \cdots \times \mathfrak{o}_{i_k}) &= -(2-k) + \sum_{s=1}^k CZ(\mathfrak{o}_{i_s})\\
        &= -2 + k + \sum_{s=1}^k (1 + 2i_s)\\
        &= 2 \left(k - 1 + \sum_{s=1}^k i_s\right)
    \end{aligned}
    \end{equation*}
    Thus,
    $$\text{ind}(\eta_L^i \times \eta_L^i) = \text{ind}(\mathfrak{o}_{i+j} \times \mathfrak{o}_{i+j}) = 2(1 + 2i + 2j),$$
    $$\mathcal{A}(\eta_L^i \times \eta_L^i) = 2ij'.$$
    So we need to check the actions of the following pairs:
    $$\mathfrak{o}_{i+j-1} \times \mathfrak{o}_{i+j+1}, \cdots, \mathfrak{o}_1 \times \mathfrak{o}_{2i+2j-1},$$
    since they have the same Fredholm index.
    To compare actions, we denote $\mathcal{A}_m := \mathcal{A}(\mathfrak{o}_m)$ and set $\mathcal{A}_0 = 0$, and we will show that the following is true within small $\epsilon$, for $m = 1,...,i+j-1$:
    $$\mathcal{A}(\mathfrak{o}_{i+j-m} \times \mathfrak{o}_{i+j+m}) = 2ij + \mathcal{A}_m - \mathcal{A}_{m-1}.$$
    Since these actions are all greater than $2ij'$, the result follows after we see that
    $$\mathcal{A}_{2i+2j+1} = 2ij + \mathcal{A}_1,$$
    since $\mathfrak{o}_{2i+2j+1} = \eta_S^{2j+1}$.

    To see the claim, let's return to the situation where $\mathcal{A}(\eta_L) = j$ without loss of generality, since $j'-j$ is so small that it doesn't change the order of actions we care about.
    In this case, the actions $\mathcal{A}_1,...,\mathcal{A}_{i+j-2}$ are given by the unordered set
    $$\{i,2i,...,(j-1)i, j,2j,...,(i-1)j\}.$$
    And we can apply a function on this set $(ij - \cdot)$, sending the above to
    $$\{(j-1)i,...,2i,i,(i-1)j,...,2j,j\}.$$
    This is an order-reversing involution, which means that for $m = 2,...,i+j-1$, we have
    $$\mathcal{A}_{i+j-m} = ij - \mathcal{A}_{m-1}.$$
    Simpler analysis on $\mathcal{A}_{i+j+1},\cdots\mathcal{A}_{2i+2j-2}$ shows that
    $$\mathcal{A}_{i+j+m} = ij + \mathcal{A}_m,$$
    finishing the claim.
\end{proof}

\subsection{Standard Complex Structure}
Let $E(\vec{a})$ be an $n$-dimensional ellipsoid with rationally independent parameters $\vec{a} = (a_1,...,a_n) \in \mathbb{R}^n_{> 0}$, and denote by $\widehat{E}(\vec a)$ its completion.
Let $\mathcal{J}(E(\vec a))$ denote the space of compatible almost complex structures on $E$, and let $\mathcal{J}(\partial E(\vec a))$ denote the space of SFT admissible almost complex structures on the symplectization $\mathbb{R}\times \partial E(\vec a)$.
Denote by $\mathcal{J}_{\text{tame}}(E(\vec a))$ the set of almost complex structures $J$ on $\widehat{E}(\vec a)$ such that
\begin{itemize}
    \item $J$ is tamed by the symplectic form on $E(\vec a)$
    \item $J$ agrees with the restriction of an element of $\mathcal{J}(\partial E(\vec a)$ on the cylindrical end.
\end{itemize}
Note that $\mathcal{J}_{\text{tame}}(E(\vec a))$ is a space of SFT admissible almost complex structures.
In $\mathbb{C}^n$ consider the standard complex structure $i$.

\begin{proposition}[\cite{mikhalkin_siegel} Proposition 4.3.3. (1)]\label{prop:ellip_diffeo_Cn_ac_structure}
    There is a diffeomorphism $Q^+_{\vec a} \colon \widehat{E}(\vec a) \rightarrow \mathbb{C}^n$ such that $(Q^+_{\vec a})^*i \in \mathcal{J}_{\text{tame}}(E(\vec a))$.
\end{proposition}

\begin{proposition}[\cite{mikhalkin_siegel} Proposition 4.4.4]\label{prop:std_cplx_str_asymp_to_which_Reeb_orbits}
    Let $\Sigma$ be a punctured Riemann sphere, and let $u \colon \Sigma \rightarrow \mathbb{C}^n$ be a proper $i$-holomorphic map.
    Let $w$ be a puncture of $\Sigma$, and let $\vec v = (v_1,...,v_n)$ denote the pole order of the components of $u$ at $w$.
    Take $i_M$ to be the index $1 \leq i \leq n$ for which $a_iv_i$ is maximal, and let $\eta_{i_M}$ be the Reeb orbit given by the intersection of $\partial E(\vec a)$ with the $i_M$-th axis.
    Then, at the puncture $w$, the curve $(Q^+_{\vec a})^{-1} \circ u \colon \Sigma \rightarrow \widehat E(\vec a)$ is positively asymptotic to the $v_{i_M}$-fold cover $\eta^{v_{i_M}}_{i_M}$.
\end{proposition}

\begin{lemma}\label{lem:existence_of_reg_curv_asymp_to_long_i_long_i}
    Consider $E(j^+,i)$ for positive integers $i,j$, with $j^+$ slightly larger than $j$ and irrational.
    Denote by $J_E$ the almost complex structure $(Q^+_{(j^+,i)})^*i \in \mathcal{J}_{\text{tame}}(E(j^+,i))$ from Proposition \ref{prop:ellip_diffeo_Cn_ac_structure}.
    Then there is a regular $J_E$-holomorphic curve positively asymptotic to $\eta_L^i \times \eta_L^i$ and satisfying $\ll T^{(2i+2j+1)}p\gg$, for $p$ in the interior of the ellipsoid.
\end{lemma}

\begin{proof}
    First we write the corresponding $i$-holomorphic map $f_\Sigma \colon \Sigma \rightarrow \mathbb{C}^2$ for $\Sigma$ the twice punctured Riemann sphere $\mathbb{CP}^1 \setminus\{w_1,w_2\}$, and marked points $w_2,...,w_{2i+2j+2}$:
    \begin{equation*}
    \begin{aligned}
        f_\Sigma^1(z) &= C_1(z-w_1)^{-i}(z-w_2)^{-i}(z-w_3)\cdots(z-w_{2i+2})\\
        f_\Sigma^2(z) &= C_2(z-w_1)^{-j}(z-w_2)^{-j}(z-w_{2i+3})\cdots(z-w_{2i+2j+2}),
    \end{aligned}
    \end{equation*}
    where $C_1,C_2$ are chosen so that $(Q^+_{(j^+,i)})^{-1} \circ f_\Sigma(w_0) = p$ for some $w_0 \in \Sigma$.
    Since this map is proper, the curve $u := (Q^+_{(j^+,i)})^{-1} \circ f_\Sigma \colon \Sigma \rightarrow \widehat E (j^+,i)$ is $J_E$-holomorphic and positively asymptotic to two Reeb orbits based on the pole orders of $f_\Sigma$ by Proposition \ref{prop:std_cplx_str_asymp_to_which_Reeb_orbits}.
    For both punctures, the pole orders are $(i,j)$, and since $ij^+ > ij$, we see that they must be asymptotic to $\eta_L^i$ in the ellipsoid, based on the above correspondence.
    Note that we had a lot of choice over $f_\Sigma$ here: we could have chosen pole orders $(i,1)$ for example, but including $j$ yields a generic curve (although this is unnecessary for our purposes).

    Now we show regularity of this curve without the tangency constraint by using automatic transversality.
    We use \cite[Theorem 1]{Wendl_automatic_transversality} for a regularity criterion and follow \cite[Lemma 5.2.5]{mcduff2024singularalgebraiccurvesinfinite} to show that the above curve satisfies the correct critical point bound.
    Then up to a further perturbation of the almost complex structure, this curve is immersed.
    
    Then we pick a divisor at $p$ and impose the tangency constraint $\ll T^{(2i+2j+2)}p\gg$, thereby cutting the index of $u$ down to $0$ on this immersed curve.
    We apply the same automatic transversality argument, now with no critical points and zero index to conclude that we still have regularity with the tangency constraint satisfied.
    Later this will allow us to say that this curve appears with expected dimension.
    Note that if the tangency order were any larger, regularity would not be guaranteed.
\end{proof}

\subsection{Existence via Neck Stretching}

First let's describe the neck stretching procedure following \cite{ganatra2024embeddingcomplexityliouvillemanifolds} and \cite{bourgeois2003compactness_BEH}.
Consider a symplectic cobordism $X$ with a codimension-1 contact type hypersurface $Y \subset X$, which separates $X$ into $X^-$ and $X^+$.
Say $J$ is an admissible almost complex structure on $X$ that is cylindrical in a neighborhood of $Y$: $(-\delta,\delta)\times Y$.
Denote by $J_Y,J_{\partial X^+},J_{\partial X^-}$ the induced almost complex structures on the full symplectizations $\mathbb{R}\times Y, \mathbb{R}\times \partial X^+, \mathbb{R}\times \partial X^-$.
Then consider a 1-parameter family of almost complex structures $\{J_t\}$ for $t \in [0,1)$ such that $J_t$ essentially stretches the neighborhood $U = (-\delta,\delta)\times Y$ to $(-R_t,R_t)\times Y$ with $R_t \rightarrow \infty$ as $t \rightarrow 1$.
More precisely, let $F_t \colon (-R_t,R_t) \rightarrow (-\delta,\delta)$ be a diffeomorphism with slope equal to 1 near $-R_t$ and $R_t$, and define $J_t$ by
$$J_t = \begin{cases}
    J|_{X^+} & X^+\\
    (F_t \times \mathds{1})_*(J_Y|_{(-R_t,R_t)\times Y}) & U\\
    J|_{X^-} & X^-.
\end{cases}$$
Let $\mathcal{M}_{X,A}^{\{J_t\}}(\Gamma^+;\Gamma^-)$ be the parametrized moduli space of $J_t$-holomorphic curves with homology class $A$ and asymptotics $\Gamma^+,\Gamma^-$.
The compactification $\overline{\mathcal{M}}_{X,A}^{\{J_t\}}(\Gamma^+;\Gamma^-)$ includes potential psuedoholomorphic buildings at $t \in [0,1)$ as well as buildings at $t = 1$.
For $t \in [0,1)$, the buildings have the following level structure:
\begin{itemize}
    \item $\geq 0$ symplectization levels $\mathbb{R}\times \partial X^+$, whose components are $J_{\partial X^+}$-holomorphic.
    \item one completion level $\widehat X$, whose components are $J$-holomorphic.
    \item $\geq 0$ symplectization levels $\mathbb{R}\times \partial X^-$, whose components are $J_{\partial X^-}$-holomorphic.
\end{itemize}
\vspace{5pt}
\noindent And for $t = 1$, we have buildings in the \textbf{broken symplectic cobordism} $X^+ \circledcirc X^-$ with the following level structure:
\begin{itemize}
    \item $\geq 0$ symplectization levels $\mathbb{R}\times \partial X^+$, whose components are $J_{\partial X^+}$-holomorphic.
    \item one completion level $\widehat X^+$, whose components are $J|_{X^+}$-holomorphic.
    \item $\geq 0$ symplectization levels $\mathbb{R}\times Y$, whose components are $J_Y$-holomorphic.
    \item one completion level $\widehat X^-$, whose components are $J|_{X^-}$-holomorphic.
    \item $\geq 0$ symplectization levels $\mathbb{R}\times \partial X^-$, whose components are $J_{\partial X^-}$-holomorphic.
\end{itemize}

These buildings are made of curve components asymptotic to Reeb orbits and satisfy the expected matching and stability criteria, i.e. the positive asymptotics of components in one level agree with the negative asymptotics of components in the next level up, and the top(bottom) level has positive(negative) asymptotics $\Gamma^+(\Gamma^-)$.
If $\partial X^+$ or $\partial X^-$ are empty, then there are no corresponding symplectization levels in any of the breakings.
The stability criteria requires that none of the symplectization levels consist entirely of trivial cylinders.

\begin{theorem}\label{thm:existence_of_curves}
    Let $\widetilde{X}_\Omega$ be a 4-dimensional, fully rounded and perturbed convex toric domain, such that the combinatorial minimizer (with respect to $k,\ell$) from Theorem \ref{thm:upper_bound} corresponds to a word of identical Reeb orbits $\Gamma = e_{i,j} \times e_{i,j}$.
    Let $J_{\partial \widetilde{X}_\Omega}$ be an almost complex structure associated with the formal perturbation invariance of a curve asymptotic to $\Gamma$.
    Then, $\mathcal{M}_{\widetilde{X}_\Omega}^J(\Gamma)\ll T^{(k)}p\gg$ is nonempty for some $J \in \mathcal{J}^{J_{\partial \widetilde{X}_\Omega}}(\widetilde{X}_\Omega)$.
\end{theorem}

\begin{proof}
    First, consider the ellipsoid $E := \lambda E(j',i)$ where $j' = j + \epsilon$ for small $\epsilon$, and $\lambda$ is a scalar chosen so that $E$ contains $\widetilde{X}_\Omega$ in its interior, but their boundaries are arbitrarily close.
    Let $(c_1,c_2)$ be the point on the moment image of $\partial \widetilde{X}_\Omega$ with normal vector $(i,j)$.
    Then $\lambda$ will be arbitrarily close to $c_2/i + c_1/j$ from above.
    The index of $\Gamma$ is $2(2i + 2j + 1)$, and by Lemma \ref{lem:ellipsoid_min_action_orbits}, the orbit set on $\partial E$ of minimal action with this index and at most two orbits is $\eta_L^i \times \eta_L^i$, where $\eta_L$ is the long Reeb orbit of action $\lambda j' \approx c_2j/i + c_1$.
    Thus the action of $\eta_L^i \times \eta_L^i$ is slightly larger than $2(c_1i + c_2j)$, which is the action of $\Gamma$.
    
    Now, by Lemma \ref{lem:existence_of_reg_curv_asymp_to_long_i_long_i}, there is an explicit SFT-admissible almost complex structure $J_E$ for which there is a regular pseudoholomorphic curve $C_E$ in $E$ asymptotic to $\eta_L^i \times \eta_L^i$ and satisfying $\ll T^{(2i+2j+1)}p\gg$, thus having index $0$.

    With this curve, we stretch the neck along the boundary of $\widetilde{X}_\Omega$ in two steps.
    First, we consider a generic $1$-parameter family of almost complex structures $J_t$ with $t \in [0,1]$, $J_0 = J_E$, and $J_1$ a perturbation of $J_E$ such that $J_1$ restricts to the generic, cylindrical $J_{\partial \widetilde{X}_{\Omega}}$ on a neighborhood $U$ of $\partial\widetilde{X}_\Omega$.
    We also require $J_1$ to restrict to generically chosen almost complex structures on $\widetilde{X}_\Omega \setminus U$ and $E \setminus (\widetilde{X}_\Omega \cup U)$.
    Then we consider the $1$-parameter family $J_t$ with $t \in [1,2)$ that realizes neck stretching along $\partial \widetilde{X}_\Omega$.

    Thus, we have a moduli space $\mathcal{M}_E^{\{J_t\}_{t\in [0,2)}}(\eta_L^i \times \eta_L^i)\ll T^{(k)}p\gg$, containing $C_E$ at $t = 0$.
    Since $C_E$ is regular, it appears with the expected dimension, and by index considerations, this is a $1$-dimensional parametrized family. 
    We want to see what happens to $C_E$ as $t$ approaches $2$, but first we need to rule out the possibility that $C_E$ breaks into a $J_t$-holomorphic building for any $t \in [0,2)$.
    If we can show this, then the compactification $\overline{\mathcal{M}}_E^{\{J_t\}_{t\in [0,2)}}(\eta_L^i \times \eta_L^i)\ll T^{(k)}p\gg$ projected to $[0,2]$ will contain a psuedoholomorphic building in the fiber over $2$ with bottom level the completion of $\widetilde{X}_\Omega$.
    After ruling out other possibilities, this level must contain the desired curve.

    \textbf{Step 1.}
    To see that there are no breakings for $t < 2$, a potential psuedoholomorphic building would have a bottom level $\widehat E$ and symplectization levels $\mathbb{R}\times \partial E$.
    As we saw before, simple curve components in $\widehat E$ must have index $\geq -1$ by genericity of the $1$-parameter family, and simple curve components in symplectization levels must have nonnegative index.
    This is because the index effects of $\mathbb{R}$-translation in the symplectization factor and genericity of the $1$-parameter family counteract each other.
    Further, the index of any component in $\widehat E$ will be even, so simple curves actually have nonnegative index, and Lemma \ref{lem:formal_cover_index} implies that any multiple-cover curves will therefore also have nonnegative index.
    Since multiple curve components in symplectizations have nonnegative index too \cite[Lemma 2.3.11]{mcduff2023ellipsoidalsuperpotentialssingularcurve}, it must be the case that our psuedoholomorphic building consists of curve components all with index $0$.

    Now consider the top level of the building, and formally glue all levels below this so we end up with a two-level building consisting of index-$0$ curve components.
    Let $C_0$ be the component in the bottom level which realizes the tangency constraint.
    There are no other fillings in the bottom level, since they would have positive index.
    And since $C$ has genus $0$, the only two possibilities for the top level are to either a pair of pants or two cylinders.
    Because $\eta_L^i \times \eta_L^i$ minimizes action among orbit sets with at most two ends and the same index, then we must have two cylinders, each negatively approaching $\eta_L^i$, for otherwise there would be negative energy components.
    However, now this is a level completely made of trivial cylinders, contradicting SFT compactness.
    Hence, there can be no breaking for $t \in [0,2)$.

    \textbf{Step 2.}
    Now, let's consider a building in the fiber over $2$, with the level structure given above.
    Since we chose $J_1$ to be a generic perturbation of $J_E$, the resulting almost complex structures on the building at $t=2$ are also generic, implying that all simple components in the building are regular and therefore have nonnegative index.
    Again, multiple covers have nonnegative index in the bottom level completion and symplectization levels.

    Consider the curve component in the bottom level which satisfies the tangency constraint.
    This curve must have at most $2$ positive ends, otherwise we'd have positive genus in $C$ or caps elsewhere in the building.
    Since the index of this component is nonnegative, then the index of the positive orbit set must be at least that of the tangency constraint.
    Thus, the energy of this curve component must be at least $2(c_1i + c_2j)$ since $e_{i,j} \times e_{i,j}$ minimizes action for these criteria.
    If the component is positively asymptotic to anything other than $e_{i,j} \times e_{i,j}$, then the energy would be greater than the whole curve $C$ by our construction of the ellipsoid, a contradiction.
    Hence, this bottom component is in the moduli space $\mathcal{M}_{\widetilde{X}_\Omega}^J(e_{i,j}\times e_{i,j})\ll T^{(k)}p\gg$, for $J \in \mathcal{J}^{J_{\partial \widetilde{X}_\Omega}}(\widetilde{X}_\Omega)$ as desired.
\end{proof}

\subsection{Discussion}
We pursue this particular construction argument, because many of the other naive approaches have their own difficulties.
First, one may want to construct these curves directly in the completion of $\widetilde{X}_\Omega$, as we did above in the ellipsoid.
But the same existence result does not work in this case, because although $\widetilde{X}_\Omega$ is diffeomorphic to $\mathbb{C}^2$, when we pull back the standard complex structure, we get an almost complex structure that is not SFT-admissible.
Another possible method is to start with a closed curve in a closed toric variety before stretching the neck.
If the variety corresponds to a Delzant polytope, then some of the exceptional divisors may have negative self intersection, allowing more negative-index curve components in psuedoholomorphic buildings when varying $J$.
Thus, one must know how to rule out all of these extra breakings to get the desired curve.
If the variety is not Delzant, then it will have orbifold points, which are difficult to deal with.

The next natural question would be why we can't use Theorem \ref{thm:existence_of_curves} to construct all of the curves for the upper bound.
This is because for an arbitrary word, $e_{i_1,j_1} \times \cdots \times e_{i_q,j_q}$, we no longer have a prescribed ellipsoid to start with.
In the above argument, we choose $E(j,i)$ specifically so the area considerations work out with the asymptotics $e_{i,j} \times e_{i,j}$.
Hence, our best option is to rely on the previous existence results and obstruction bundle gluing from \cite{McDuff-Siegel_unperturbed_curves} to obtain all the curves we need.

\section{Computations}\label{sec:computations}

Now that we have proven the computation formula for 4-dimensional convex toric domains, we compute $\Tilde{\mathfrak g}_k^\ell$ for a wide range of examples and find sharp obstructions to certain stabilized embeddings, previously unseen by other capacities.
First we restate previous lemmas and prove new ones which will help us solve the minimization problem.

\begin{lemma}[Lemma \ref{lem:index-preserving-move}]
    Let $X_\Omega$ be a 4-dimensional convex toric domain, with $\partial \Omega$ given by a piece-wise smooth function $h \colon [0,x] \rightarrow [0,1]$, with a finite number of non-smooth vertices (fully rounded domains satisfy this condition).
    Assume the intersection of $h$ with the line $y=x$ is not a vertex of $h$, and denote the slope of $h$ at this point by $m$.
    For $(i,j) \in \mathbb{Z}_{\geq 1}^2$, the following index-preserving moves which weakly decrease action are possible:
    \begin{itemize}
        \item If $-(i-1)/(j+1) \leq m$, then $||(i-1,j+1)||_\Omega^* \leq ||(i,j)||_\Omega^*$.
        \item If $-(i+1)/(j-1) \geq m$, then $||(i+1,j-1)||_\Omega^* \leq ||(i,j)||_\Omega^*$.
    \end{itemize}
    In particular, the minimum norm for a given index $k$,
    $$\min\{||(i,j)||_\Omega^* \;\vert\; i + j = k\},$$
    is minimized at $(i,j)$ if $-i/j = m$, or if there is no such $i,j$, then it is minimized at one of $(i,j), (i-1,j+1)$ where $-i/j < m < -(i-1)/(j+1)$.
\end{lemma}

\begin{lemma}[Lemma \ref{lem:upper_bound_rhs_minimizers}]
    Given $k$ and $\ell$, the minimum
    $$\min_{\{(i_s,j_s)\}_{s=1}^q \in \mathcal{P}_{k,\ell}} \left( \sum_{s=1}^q ||(i_s,j_s)||_\Omega^*\right)$$
    is achieved by a set $\{(i_s,j_s)\}_{s=1}^q$ such that all but possibly one of the tuples $(i_s,j_s)$ are individual minimizers of
    $$\min_{i+j = k_s}||(i,j)||_\Omega^*,$$
    for $k_s = i_s + j_s$ and the possible extra tuple is of the form $(1,j)$ for some $j \geq 0$, which may be needed to preserve weak permissibility.
\end{lemma}

\begin{lemma}\label{lem:||(i,j)||_nondecreasing_increase}
    For any convex toric domain $X_\Omega$, fix $i$ and consider the norms $||(i,j)||_\Omega^*$ for $j \in \mathbb{Z}_{\geq 0}$.
    As $j$ increases, the norm weakly increases at a nondecreasing rate.
    More precisely:
    $$||(i,j+1)||_\Omega^* - ||(i,j)||_\Omega^* \geq 0$$
    and
    $$||(i,j'+1)||_\Omega^* - ||(i,j')||_\Omega^* - \bigg(||(i,j+1)||_\Omega^* - ||(i,j)||_\Omega^*\bigg) \geq 0$$
    for all $j' > j \geq 0$.
\end{lemma}

\begin{proof}
    Indeed, the first inequality follows because for $(x,y) \in \partial \Omega$ the maximizer for $||(i,j)||_\Omega^*$, we have
    $$||(i,j+1)||_\Omega^* \geq \langle(i,j+1),(x,y)\rangle = ix + jy + y\geq ix + jy = ||(i,j)||_\Omega^*.$$
    The second inequality follows because for $(x,y) \in \partial \Omega$ the maximizer for $||(i,j+1)||_\Omega^*$, we have
    \begin{equation*}
    \begin{aligned}
        ||(i,j+2)||_\Omega^* + ||(i,j)||_\Omega^* & \geq ix + jy + 2y + ix + jy\\
        & = 2(ix + jy + y)\\
        & = 2||(i,j+1)||_\Omega^*
    \end{aligned}
    \end{equation*}
    Thus,
    $$||(i,j+2)||_\Omega^* - ||(i,j+1)||_\Omega^* \geq ||(i,j+1)||_\Omega^* - ||(i,j)||_\Omega^*,$$
    and the result follows by induction.
\end{proof}

\subsection{Domains Between the Cube and the Cylinder}

The minimization problem becomes tractable when we restrict ourselves to convex toric domains $X_\Omega$, where $\Omega$ intersects the $y$-axis at $(0,1)$ and contains the unit square.
With this restriction, $X_\Omega$ is contained in the unit cylinder, but contains the unit cube:
$$P(1,1) \subset X_\Omega \subset P(\infty,1).$$
We'll call these \textbf{long domains}.

\begin{lemma}\label{lem:contains_P11_forms}
    Let $X_\Omega$ be a 4-dimensional convex toric domain such that the boundary $\partial \Omega$ intersects the $y$-axis at $(0,1)$.
    Assume $\Omega$ contains the unit square $[0,1]\times[0,1]$.
    Then for any positive integers $k,\ell$, there exists a minimizing set $\{(i_s,j_s)\}_{s=1}^q$ for $\Tilde{\mathfrak g}_k^\ell$ which is in one of the following forms:
    \begin{itemize}
        \item $\{(0,k)\}$
        \item $\begin{aligned}[t]
            \{(0,1)^{\times r},(1,j)\}, \quad \quad & 0 \leq r \leq \ell - 1, \quad j \geq 0,\\
            & 2r + j + 1 = k.
        \end{aligned}$
    \end{itemize}
\end{lemma}

\begin{proof}
    To prove this, we apply Lemma \ref{lem:index-preserving-move} and the definition of $(k,\ell)$-admissible sets.
    First, assume there exists a minimizing set with one vector in it $\{(i,k-i)\}$ (which is in this form by the index criterion).
    Because the slope of the boundary is 0 at it's intersection with the $y = x$ line, Lemma \ref{lem:index-preserving-move} tells us that
    $$||(0,k)||_\Omega^* \leq ||(1,k-1)||_\Omega^* \leq \cdots \leq ||(i,k-i)||_\Omega^*.$$
    Thus, $\{(0,k)\}$ would also be a minimizing set.

    Now, assume there is a minimizing set with multiple vectors.
    By the same logic, we can move \textit{all but one} of these vectors up and to the left where $x = 0$ without increasing the norm-sum of the set. 
    If we moved all of them, we'd spoil weak permissibility.
    This extra vector, we can move to a vector $(1,j)$ again without increasing its norm.

    Lastly, if any of the vectors on the $y$-axis are not of the form $(0,1)$, we can make the following trade
    $$||(0,j')||_\Omega^* + ||(1,j)||_\Omega^* \geq ||(0,j'-1)||_\Omega^* + ||(1,j+1)||_\Omega^*.$$
    Indeed, if $(x,y)$ is the maximizer for $||(1,j+1)||_\Omega^*$, then
    $$||(1,j+1)||_\Omega^* - ||(1,j)||_\Omega^* \leq x + jy + y - (x + jy) = y \leq 1,$$
    and $||(0,j')||_\Omega^* - ||(0,j'-1)||_\Omega^* = j' - (j' - 1) = 1$.
    Thus, we move all of the $y$-axis vectors down to $(0,1)$, and move $(1,j)$ up to maintain the index.
\end{proof}

This reduces our search to the parameter $r$, and thus we need to analyze when the following index-preserving trade can be made:
$$||(0,1)||_\Omega^* + ||(1,j)||_\Omega^* \quad\text{vs.}\quad ||(1,j+2)||_\Omega^*.$$

\begin{lemma}\label{lem:define_J_transition_where_trading_away_0,1_viable}
    With meaning distinguished from almost complex structures by context, define
    $$J := \min\{j \geq 0 \;\colon\; ||(0,1)||_\Omega^* + ||(1,j)||_\Omega^* < ||(1,j+2)||_\Omega^*\}.$$
    Then $||(0,1)||_\Omega^* + ||(1,j)||_\Omega^* \geq ||(1,j+2)||_\Omega^*$ for all $j < J$ by definition, but also $||(0,1)||_\Omega^* + ||(1,j)||_\Omega^* < ||(1,j+2)||_\Omega^*$ for all $j > J$.
\end{lemma}

\begin{proof}
    Indeed, this is true by Lemma \ref{lem:||(i,j)||_nondecreasing_increase}, since if $j > J$, then
    \begin{equation*}
    \begin{aligned}
        ||(1,j+2)||_\Omega^* - ||(1,j)||_\Omega^* & = ||(1,j+2)||_\Omega^* - ||(1,j+1)||_\Omega^*\\
        &\quad\quad + \;||(1,j+1)||_\Omega^* - ||(1,j)||_\Omega^*\\
        & \geq 2 \big(||(1,j+1)||_\Omega^* - ||(1,j)||_\Omega^*\big)\\
        & \geq 2 \big(||(1,J+2)||_\Omega^* - ||(1,J+1)||_\Omega^*\big)\\
        & \geq ||(1,J+2)||_\Omega^* - ||(1,J+1)||_\Omega^*\\
        &\quad\quad +\; ||(1,J+1)||_\Omega^* - ||(1,J)||_\Omega^*\\
        & = ||(1,J+2)||_\Omega^* - ||(1,J)||_\Omega^*\\
        & > ||(0,1)||_\Omega^* = 1.
    \end{aligned}
    \end{equation*}
\end{proof}
In words, $J$ marks a transition - if you want to minimize action and you have $(1,j)$, then $j < J$ means you should throw away as many $(0,1)$ tuples as possible to get close to $J$, and $j > J$ means you should acquire as many $(0,1)$ tuples as allowed to get close to $J$.
The next lemma shows that there are at most 4 values we need to check to find $J$.

\begin{lemma}
    Let $X_\Omega$ be a 4-dimensional convex toric such that the boundary $\partial\Omega$ intersects the $y$-axis at $(0,1)$ and such that $\Omega$ contains the region $[0,1+\epsilon] \times [0,1]$ for small $\epsilon$.
    There exists convex toric domain $X_{\Omega'}$ close in the Hausdorff distance, such that the following is true.
    Let $(x_0,1/2)$ be the intersection of $\partial\Omega'$ and the line at height $1/2$, and let the scaled normal vector at $(x_0,1/2)$ be written $(1,\alpha)$ for $\alpha \in \mathbb{R}$.
    Then $J$ as defined in Lemma \ref{lem:define_J_transition_where_trading_away_0,1_viable}, must be found in the interval
    $$\lfloor\alpha\rfloor - 1 \leq J \leq \lceil\alpha\rceil + 1.$$
\end{lemma}

\begin{proof}
    Smooth the boundary of $\Omega$, so that $\alpha$ is not an integer, and so that the maximizers $(x_0,1/2)$ and $(x_1,y_1)$ for $||(1,\alpha)||_{\Omega'}$ and $||(1,\lceil\alpha\rceil)||_{\Omega'}$ are distinct, implying $y_1 > 1/2$.
    
    First, let $j = \lfloor\alpha\rfloor - 2$, and let $(x,y)$ be the maximizer of $||(1,\lfloor\alpha\rfloor)||_\Omega^*$.
    Then,
    \begin{equation*}
    \begin{aligned}
        ||(1,j+2)||_\Omega^* - ||(1,j)||_\Omega^* & \leq ||(1,\lfloor\alpha\rfloor
        )||_\Omega^* - ||(1,j)||_\Omega^* \\
        & \leq x + y\lfloor\alpha\rfloor - (x + y\lfloor\alpha\rfloor - 2y) \\
        & = 2y \leq 1.
    \end{aligned}
    \end{equation*}

    Second, let $j = \lceil\alpha\rceil + 2$
    Then,
    \begin{equation*}
    \begin{aligned}
        ||(1,j)||_\Omega^* - ||(1,j-2)||_\Omega^* & \geq ||(1,j
        )||_\Omega^* - ||(1,\lceil\alpha\rceil)||_\Omega^* \\
        & \geq x_1 + y_1\lceil\alpha\rceil + 2y_1 - (x_1 + y_1\lceil\alpha\rceil) \\
        & = 2y_1 > 1.
    \end{aligned}
    \end{equation*}
    The results for the rest of the values for $j$ follow from Lemma \ref{lem:||(i,j)||_nondecreasing_increase}.
\end{proof}

We now study how the parameters $k$ and $\ell$ affect the minimization given that we know $J$.

\begin{lemma}
    Assume $k - 2\ell + 1 \geq J$.
    Then $\Tilde{\mathfrak g}_k^\ell(X_\Omega) = \min\{k, \ell - 1 + ||(1,k - 2\ell + 1)||_\Omega^*\}$.
    Note that $\ell = \infty$ does not satisfy the assumption.
\end{lemma}

\begin{proof}
    Let $P = \{(0,1)^{\times r},(1,j)\}$ be an action minimizer with more than one orbit.
    Since the half index of $P$ is $k$, we have $j = k-2r -1$.

    \noindent\textbf{Case 1:} $r \leq \ell -2$.
    This implies $j = k - 2r - 1 \geq k - 2(\ell -2) - 1 = k - 2\ell + 3 \geq J + 2$.
    By definition of $J$, we could make a trade that contradicts minimimality of $P$:
    $$||(1,j)||_\Omega^* > ||(0,1)||_\Omega^* + ||(1,j-2)||_\Omega^*$$

    \noindent\textbf{Case 2:} $r = \ell - 1$.
    Then $P = \{(0,1)^{\ell-1}, (1,k - 2\ell + 1)\}$ by index considerations.
    Thus, the action minimizer is either $\{(0,k)\}$ or $\{(0,1)^{\ell - 1}, (1,k - 2\ell +1)\}$, and the result follows.
\end{proof}

\begin{lemma}
    Assume $k - 2\ell + 1 < J$, for with $\ell$ potentially infinite.
    Then
    $$\Tilde{\mathfrak g}_k^\ell(X_\Omega) = \begin{cases}
        k & k < J+1\\
        \min\{k, (k-J-1)/2 + ||(1,J)||_\Omega^* & k - J > 0, \text{ odd}\\
        \min\{k, (k-J-2)/2 + ||(1,J+1)||_\Omega^* & k - J > 0, \text{ even}
    \end{cases}$$
\end{lemma}

\begin{proof}
    Let $P = \{(0,1)^{\times r},(1,j)\}$ be an action minimizer with more than one orbit.
    Since the half index of $P$ is $k$, we have $j = k-2r -1$.

    \noindent\textbf{Case 1:} $j \geq J+2$, $r \leq \ell - 2$.
    Then we could make a trade contradicting minimality of $P$:
    $$||(0,1)||_\Omega^* + ||(1,j-2)||_\Omega^* < ||(1,j)||_\Omega^*.$$

    \noindent\textbf{Case 2:} $j \geq J+2$, $r = \ell - 1$.
    Then $P = \{(0,1)^{\times(\ell - 1)}, (1,k - 2\ell + 1)\}$, but this is a contradiction either because $k - 2\ell + 1 < 0$, or because of the assumption in the lemma statement:
    $$j = k - 2\ell + 1 < J \leq j - 2.$$
    
    \noindent\textbf{Case 3:} $j < J$.
    Then we could make a trade:
    $$||(0,1)||_\Omega^* + ||(1,j)||_\Omega^* \geq ||(1,j+2)||_\Omega^*$$
    We can continue making this trade until either $j \geq J$ or $r = 0$.
    If $r$ reaches $0$ before $j \geq J$, then by index considerations $k - 1 = j < J$, and $\{(0,k)\}$ will be the minimizer.
    
    Thus we are left with
    $P = \{(0,1)^{(k-J-1)/2}, (1,J)\}$ or $\{(0,1)^{(k-J-2)/2}, (1,J+1)\}$, depending on the parity of $k - J$.
    Hence, the result follows.
    Note that for $\ell = \infty$, Case 1 still yields a contradiction, Case 2 does not apply, and Case 3 follows unchanged.
\end{proof}

We summarize all the above work in a theorem.

\begin{theorem}\label{thm:comp_for_XOmega_unit_square_in_Omega}
    Let $X_\Omega$ be a 4-dimensional convex toric domain such that the boundary $\partial \Omega$ intersects the $y$-axis at $(0,1)$.
    Assume $\Omega$ contains the unit square $[0,1]\times[0,1]$.
    Then for any positive integers $k,\ell$, and $J$ given by Lemma \ref{lem:define_J_transition_where_trading_away_0,1_viable}, the capacity $\Tilde{\mathfrak g}_k^\ell(X_\Omega)$ is given by
    $$\Tilde{\mathfrak g}_k^\ell =  \begin{cases}
        \min\{k, \ell - 1 + ||(1,k - 2\ell + 1)||_\Omega^*\} & k - 2\ell + 1 \geq J\\
        k & k - 2\ell + 1 < J,\; k - J < 1\\
        \min\{k, \tfrac{k-J-1}{2} + ||(1,J)||_\Omega^*\} & k - 2\ell + 1 < J,\;  k - J > 0, \text{ odd}\\
        \min\{k, \tfrac{k-J-2}{2} + ||(1,J+1)||_\Omega^*\} & k - 2\ell + 1 < J,\; k - J > 0, \text{ even}
    \end{cases}$$
\end{theorem}

\begin{corollary}
    We highlight the Gutt-Hutchings and McDuff-Siegel capacities in this setting:
    $$\Tilde{\mathfrak g}_k^1(X_\Omega) = k, \quad \Tilde{\mathfrak g}_k^\infty(X_\Omega) = \begin{cases}
        k & k - J < 1\\
        \min\{k, \tfrac{k-J-1}{2} + ||(1,J)||_\Omega^*\} & k - J > 0, \text{ odd}\\
        \min\{k, \tfrac{k-J-2}{2} + ||(1,J+1)||_\Omega^*\} & k - J > 0, \text{ even}
    \end{cases}$$
\end{corollary}

\begin{proof}
    Indeed, for the Gutt-Hutchings capacity, we see that
    $$\ell - 1 + ||(1,k-2\ell+1)||_\Omega^* = ||(1,k-1)||_\Omega^* \geq k,$$
    and if $k - 2\ell + 1 = k - 1 < J$, then $k - J < 1$, so the capacity is $k$.
    This agrees with computing the Gutt-Hutchings capacity directly in such a setting.
\end{proof}

\subsection{Polydisks}
Here we compute the capacity for the polydisk $P(a,1)$, first directly, then using Theorem \ref{thm:comp_for_XOmega_unit_square_in_Omega}.

\begin{proposition}[Polydisks]\label{prop:polydisk}
    For positive integers $k,\ell$, we compute $\Tilde{\mathfrak g}_k^\ell$ for the Polydisk $P(a,1)$ with $a > 1$:
    $$\Tilde{\mathfrak g}_k^\ell(P(a,1)) = 
    \begin{cases}
        k, & \ell = 1 \text{ or } k = 1,2\\
        \min\{k,k+a-\ell\}, & \ell \geq 2, \; k \geq 3, \; \frac{k-1}{2} \geq \ell - 1\\
        \min\{k,\lceil\frac{k-1}{2}\rceil + a\}, & \ell \geq 2,\; k \geq 3,\; l-1 \geq \frac{k-1}{2}.
    \end{cases}$$
\end{proposition}

\begin{proof}
    First, if $\ell = 1$ or if $k = 1,2$, then the second form of Lemma \ref{lem:contains_P11_forms} is impossible, so $\Tilde{\mathfrak g}_k^\ell = k$ for these values of $\ell$ and $k$.

    For other values of $\ell$ and $k$, the second form comes into play: $\{(0,1)^{\times r},(1,j)\}$.
    If we consider the maximum value possible for $r$, we have the $r \leq \ell - 1$ bound, but we also have a bound 
    $$r = \frac{k - j - 1}{2} \leq \frac{k-1}{2}.$$
    This is where the inequalities in the formula arise.

    Now, to reduce our candidate minimizers, we can decrease $j$ and increase $r$.
    Indeed, computing norms with respect to the polydisk $P(1,a)$ gives
    $$||(1,j)||_\Omega^* = a + j > a + j - 1 = ||(1,j-2)||_\Omega^* + ||(0,1)||_\Omega^*.$$
    Thus, the minimizer is of the form
    $$\{(0,1)^{\times r}, (1,j)\}\quad r = \min\left\{\ell - 1, \left\lfloor \frac{k-1}{2}\right\rfloor\right\}, \quad j \geq 0, \quad 2r + j + 1 = k$$
    whose norm sum is
    $$r||(0,1)||_\Omega^* + ||(1,j)||_\Omega^* = r + a + j.$$

    So, if $\frac{k-1}{2} \geq \ell - 1$, then $r = \ell - 1$, implying $j = k - 2\ell + 1$, which gives a norm sum of
    $$r + a + j = k + a - \ell.$$
    Now, assume $\ell - 1 \geq \frac{k-1}{2}$.
    If $k$ is odd, then $r  = \frac{k-1}{2}$, implying $j = 0$, and therefore the norm sum is $\frac{k-1}{2} + a$.
    If $k$ is even, then $r = \frac{k-2}{2}$, implying that $j = 1$, and therefore the norm sum is $\frac{k}{2} + a$.
    Thus, we can write the norm sum for all $k$ as 
    $$\left\lceil \frac{k-1}{2} \right\rceil + a = \begin{cases}
        \frac{k-1}{2} + a & k \text{ odd}\\
        \frac{k}{2} + a & k \text{ even}
    \end{cases},$$
    completing the formula.
\end{proof}

\begin{remark}
    We recover the McDuff-Siegel capacity when $\ell = \infty$ \cite[Theorem 1.3.4]{McDuff-Siegel_unperturbed_curves}:
    $$\Tilde{\mathfrak g}_k^\infty(P(a,1)) = \min\{k,a + \lceil\tfrac{k-1}{2}\rceil\}.$$
\end{remark}

\begin{proof}[Proof that Theorem \ref{thm:comp_for_XOmega_unit_square_in_Omega} implies Proposition \ref{prop:polydisk}]
    Indeed, for $P(a,1)$, we find $J = 0$ because
    $$||(0,1)||_\Omega^* + ||(1,0)||_\Omega^* = 1 + a < a + 2 = ||(1,2)||_\Omega^*.$$
    \textbf{Case 1:} $\ell = 1$.
    Then $k - 2\ell + 1 = k - 1 \geq 0$, so
    \begin{equation*}
    \begin{aligned}
        \Tilde{\mathfrak g}_k^\ell &= \min\{k,\ell - 1 + ||(1,k-2\ell+1)||_\Omega^*\}\\
        & = \min\{k,||(1,k-1)||_\Omega^*\}\\
        & = k,
    \end{aligned}
    \end{equation*}
    since $||(1,k-1)||_\Omega^* \geq \langle(1,k-1),(1,1)\rangle = k$.
    
    \noindent\textbf{Case 2:} $k = 1,2$.
    If $k - 2\ell + 1 \geq J$, then $2 \ell \leq k + 1 \leq 3$, implying $\ell = 1$, and thus Case 1 applies.
    Otherwise, if $k - 2\ell + 1 < J$, we have for $k = 1$
    $$\Tilde{\mathfrak g}_k^\ell = \min\{k,(k - J - 1)/2 + ||(1,0)||_\Omega^*\} = \min\{1,||(1,0)||_\Omega^*\} = 1.$$
    Similarly, for $k = 2$, we have
    $$\Tilde{\mathfrak g}_k^\ell = \min\{k,(k - J - 2)/2 + ||(1,1)||_\Omega^*\} = \min\{2,||(1,1)||_\Omega^*\} = 2.$$

    \noindent\textbf{Case 3:} $\ell \geq 2, k \geq 3, \tfrac{k-1}{2} \geq \ell -1$.
    Then $k - 2\ell + 1 \geq 0 = J$, so
    \begin{equation*}
    \begin{aligned}
        \Tilde{\mathfrak g}_k^\ell &= \min\{k,\ell - 1 + ||(1,k-2\ell+1)||_\Omega^*\}\\
        & = \min\{k,\ell - 1 + a + k - 2\ell + 1\}\\
        & = \min\{k, k + a - \ell\}.
    \end{aligned}
    \end{equation*}

    \noindent\textbf{Case 4:} $\ell \geq 2, k \geq 3, \tfrac{k-3}{2} \leq \ell -2$.
    Then $k - 2\ell + 1 \leq 0$, so first consider $k - 2\ell + 1 = 0$ (implying $k$ odd).
    Then
    $$\Tilde{\mathfrak g}_k^\ell = \min\{k, \ell - 2 + ||(1,0)||_\Omega^*\} = \min\{k, \tfrac{k+1}{2} - 1 + a\} = \min\{k, \tfrac{k-1}{2}+a\}.$$
    Now, let $k - 2\ell + 1 < 0$, and see that for $k$ odd, we have
    $$\Tilde{\mathfrak g}_k^\ell = \min\{k,\tfrac{k-1}{2} + a\}$$
    and for $k$ even, we have
    $$\Tilde{\mathfrak g}_k^\ell = \min\{k,\tfrac{k}{2} + a\}$$
    giving the desired result.
\end{proof}

\subsection{New Obstructions}

Here we describe specific examples where $\Tilde{\mathfrak g}_k^\ell$ sees an obstruction for $1 < \ell < \infty$, but the capacity fails to witness this for $\ell = 1,\infty$.
In the following sections, we will focus on improving upon the volume obstruction and generalizing to a larger class of domains with sharp obstructions.

\begin{example}\label{ex:new_obstruction}
    Consider the polytope $\Omega \subset \mathbb{R}^2$ given by the vertices $(0,0)$, $(0,1)$, $(\frac{42}{28},1)$, $(\frac{51}{28},\frac{25}{28})$, $(2,0)$.
    And consider the slightly larger polytope $\hat{\Omega} \subset \mathbb{R}^2$ given by the vertices $(0,0)$,  $(0,1)$, $(\frac{48}{28},1)$, $(\frac{51}{28},\frac{25}{28})$, $(2,0)$.
    Denote their corresponding convex toric domains by $X_\Omega$ and $X_{\hat{\Omega}}$.
    We wish to compute $\Tilde{\mathfrak g}_k^\ell$ for these domains to see when we can distinguish the two domains.

Both domains contain the polydisk $P(1,1)$, so Lemma \ref{lem:contains_P11_forms} applies.
For the Gutt-Hutchings capacity $\Tilde{\mathfrak g}_k^1$, the only possible minimizing collection is $\{(0,k)\}$, which has the same norm for both domains, giving
$$\Tilde{\mathfrak g}_k^1(X_\Omega) = \Tilde{\mathfrak g}_k^1(X_{\hat{\Omega}}) = k.$$
Now for the computation of the general capacity $\Tilde{\mathfrak g}_k^\ell$, we first compute norms of interest.
For the inner domain $\Omega$, we have
$$\begin{cases}
    ||(0,y)||^*_\Omega = \langle (0,y), (0,1)\rangle = y & \text{ for all } y\\
    ||(1,y')||^*_\Omega = \langle (1,y'), (\frac{42}{28},1)\rangle = \frac{42}{28} + y' & \text{ for } y' \geq 3\\
    ||(1,y')||^*_\Omega = \langle (1,y'), (\frac{51}{28},\frac{25}{28})\rangle = \frac{51}{28} + y'\frac{25}{28} & \text{ for } y' = 1,2\\
    ||(1,0)||^*_\Omega = \langle (1,0), (2,0)\rangle = 2
\end{cases}$$
and for the outer domain $\hat{\Omega}$, we have
$$\begin{cases}
    ||(0,y)||^*_{\hat{\Omega}} = \langle (0,y), (0,1)\rangle = y & \text{ for all } y\\
    ||(1,y')||^*_{\hat{\Omega}} = \langle (1,y'), (\frac{48}{28},1)\rangle = \frac{48}{28} + y' & \text{ for } y' \geq 1\\
    ||(1,0)||^*_{\hat{\Omega}} = \langle (1,0), (0,2)\rangle = 2
\end{cases}$$
By Lemma \ref{lem:contains_P11_forms}, these are the only ordered pairs used in the computation of $\Tilde{\mathfrak g}_k^\ell$, and we note that the only points at which the norms differ between $\Omega$ and $\hat{\Omega}$ are of the form $(1,y')$ for $y' \geq 2$.

But assume such a point is in the minimizing set for the McDuff-Siegel capacity of the inner domain $\Tilde{\mathfrak g}_k^\infty(X_\Omega)$.
Since we can have unlimited positive ends, we can  make trade offs of the following form (which preserve index and do not spoil weak permissibility).
For $y' \geq 3$, we have
$$||(1,y')||^*_\Omega = \frac{42}{28} + y'  > 1 + y' = ||(1,0)||^*_\Omega + ||(0,y'-1)||^*_\Omega,$$
and for $y' = 2$, we have
$$||(1,2)||^*_\Omega = \frac{51}{28} + 2\cdot\frac{25}{28} = \frac{101}{28} > 3 = ||(1,0)||^*_\Omega + ||(0,1)||^*_\Omega.$$
The same is true for the outer domain $\Tilde{\mathfrak g}_k^\infty(X_{\hat{\Omega}})$: for $y' \geq 2$, we have
$$||(1,y')||^*_{\hat{\Omega}} = \frac{48}{28} + y' > 1 + y' = ||(1,0)||^*_{\hat{\Omega}} + ||(0,y'-1)||^*_{\hat{\Omega}}$$
Hence, the only points we need to consider for $\Tilde{\mathfrak g}_k^\infty$ are of the form $(0,y)$, $(1,0)$, or $(1,1)$, and the norms of these points agree on $\Omega$ and $\hat{\Omega}$.
Thus, $$\Tilde{\mathfrak g}_k^\infty(X_\Omega) = \Tilde{\mathfrak g}_k^\infty(X_{\hat{\Omega}})$$
for all $k$.

Now, let us compute $\Tilde{\mathfrak g}_5^2$ of $X_\Omega$ and $X_{\hat{\Omega}}$ and see that they differ.
By Lemma \ref{lem:contains_P11_forms}, we only have to check the norm sums of the following collections:

\vspace{5pt}
\begin{center}
\begin{tabular}{|c|c|c|}
    \hline
    Orbit Set & $\sum||\cdot||_\Omega^*$ & $\sum||\cdot ||_{\widehat \Omega}^*$\\
    \hline
    $\{(0,5)\}$ & $5$ & $5$\\
    $\{(0,1),(1,2)\}$ & $129/28 \approx 4.6$ & $132/28 \approx 4.7$\\
    $\{(0,2),(1,1)\}$ & $132/28$ & $132/28$\\
    $\{(0,3),(1,0)\}$ & $5$ & $5$\\
    \hline
\end{tabular}
\end{center}
\vspace{5pt}

\noindent Hence we see the difference
$$\Tilde{\mathfrak g}_5^2(X_\Omega) = 129/28 < 132/28 = \Tilde{\mathfrak g}_5^2(X_{\hat{\Omega}}).$$
Note that $\Tilde{\mathfrak g}_5^\infty$ does not see this difference because it allows point sets with 3 positive ends and therefore trades away $\{(1,2),(0,1)\}$ for $\{(0,1)^{\times 2}, (1,0)\}$.

Then by the monotonicity axiom, there does not exist a generalized Liouville embedding from $X_{\hat{\Omega}} \hookrightarrow X_\Omega$.
We already knew that by the volume obstruction, but by the stabaliztion axiom we see that
$$\Tilde{\mathfrak g}_5^2(X_\Omega \times B^2(c)) = 129/28 < 132/28 = \Tilde{\mathfrak g}_5^2(X_{\hat{\Omega}} \times B^2(c))$$
for any $c \geq 132/28$.
Thus, there does not exist a stabilized embedding
$$X_{\hat{\Omega}} \times \mathbb{C}^n \hookrightarrow X_\Omega \times \mathbb{C}^n$$
either, which is an obstruction not previously exhibited by any other capacity.

\end{example}

\begin{example}\label{ex:new_obstruction_long}
    Now consider the elongated version of the above example, where $\Omega$ is given by vertices $(0,0), (0,1), (\frac{70}{28},1), (\frac{79}{28},\frac{25}{28}),(3,0)$, and $\widehat \Omega$ is given by $(0,0), (0,1), (\frac{76}{28},1), (\frac{79}{28},\frac{25}{28}),(3,0)$.

    Similar to the above analysis, $\Tilde{\mathfrak g}_k^1$ and $\Tilde{\mathfrak g}_k^\infty$ do not distinguish $X_\Omega$ and $X_{\widehat\Omega}$ for any $k$.
    But now, $\Tilde{\mathfrak g}_k^2$ also does not distinguish these two domains because the orbit set $\{(0,k)\}$ has less action than all of the two-orbit sets.
    However, when we take $\ell = 3$, we get that
    $$\Tilde{\mathfrak g}_7^3(X_\Omega) = 185/28 < 188/28 = \Tilde{\mathfrak g}_7^3(X_{\widehat\Omega}).$$
    In general, if we want to see obstructions in wider domains like this, then we need to take $\ell$ large enough, but not to $\infty$.
    See Figure \ref{fig:bumped_out_obstruction} for a general version of this example, scaled by a length factor $a$.
\end{example}

\begin{figure}
    \centering
    \includegraphics[width=1\linewidth]{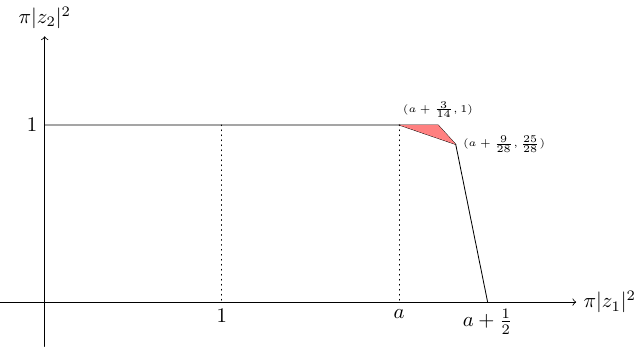}
    \caption{See Examples \ref{ex:new_obstruction} and \ref{ex:new_obstruction_long}: $\Tilde{\mathfrak g}_k^\ell$ for $1 < \ell < \infty$ sees the obstruction from embedding the larger domain into the smaller one, as would ECH or volume, but this obstruction stabilizes unlike ECH and volume.}
    \label{fig:bumped_out_obstruction}
\end{figure}

\begin{example}\label{ex:new_obstruction_polydisk_into_5-gon}
    Here we give another example where $\Tilde{\mathfrak g}_k^\ell$ provides a better obstruction for $1 < \ell < \infty$, but also better than the volume obstruction.
    Consider $\Omega$ given by vertices $(0,0)$, $(0,1)$, $(1,1)$, $(2,1/2)$, $(2,0)$.
    Define
    $$L := \max\{\lambda \mid P(\lambda2,\lambda) \hookrightarrow X_\Omega\}$$
    We ask what upper bounds we can achieve on $L$ through obstructions.
    \begin{itemize}
        \item $\Tilde{\mathfrak g}_k^1 \Rightarrow L \leq 1$. Volume $\Rightarrow L \leq 7/8$. $\Tilde{\mathfrak g}_k^\infty \Rightarrow L \leq 6/7$.
        \item But $\Tilde{\mathfrak g}_5^2(\tfrac{4}{5}P(2,1)) = 4 = \Tilde{\mathfrak g}_5^2(X_\Omega)$, implying that $L \leq 4/5$.
    \end{itemize}

    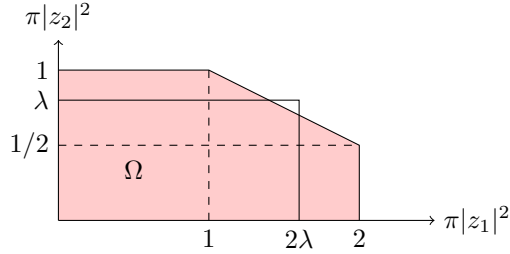
\begin{figure}
        \centering
        \begin{tikzpicture}[]
        
            \begin{scope}[scale=2]
            \fill [color=red!20!] (0,0) -- (0,1) -- (1,1) -- (2,1/2) -- (2,0);
            \draw [->] (0,0) -- (2.5,0) node [right] {$\pi|z_1|^2$};
            \draw [->] (0,0) -- (0,1.2) node [above] {$\pi|z_2|^2$};
            \draw (0,1) node[left] {$1$} -- (1,1) -- (2,1/2) -- (2,0) node[below] {$2$};
            \draw [dashed] (0,1/2) node[left] {$1/2$} -- (2,1/2);
            \draw [dashed] (1,1) -- (1,0) node[below] {$1$};
            \draw (1/2,1/3) node {$\Omega$};
        
            \draw (0,4/5) node[left] {$\lambda$} -- (8/5,4/5) -- (8/5,0) node[below] {$2\lambda$};
            \end{scope}
        
        
        \end{tikzpicture}
        \caption{$\Tilde{\mathfrak g}_k^2$ gives better obstruction than $\Tilde{\mathfrak g}_k^1$, $\Tilde{\mathfrak g}_k^\infty$, and volume in Example \ref{ex:new_obstruction_polydisk_into_5-gon}}
        \label{fig:new_obstruction_polydisk_into_5-gon}
    \end{figure}
\end{example}

\subsection{Sharp Obstructions}
The refined capacities give sharp obstructions between long domains, which were previously unseen by the Gutt-Hutchings or McDuff-Siegel capacity sequences.

\begin{example}\label{ex:sharp_obstruction}
    Consider $\Omega$ the moment polytope with vertices $(0,0)$, $(0,1)$, $(1,1)$, $(2,1/2)$, and $(2,0)$; see Figure \ref{fig:sharp_obstruction}.
    We try to find the longest polydisk $P(a,1)$ that embeds into $X_\Omega$.
    We know the Gutt-Hutchings capacities for both:
    $$\Tilde{\mathfrak g}_k^1(X_\Omega) = \Tilde{\mathfrak g}_k^1(P(a,1)) = k,$$
    for all $k$.
    For $\ell=\infty$, we apply Theorem \ref{thm:comp_for_XOmega_unit_square_in_Omega} by computing $J = 1$ and getting
    \begin{equation*}
    \begin{aligned}
        \Tilde{\mathfrak g}_k^\infty(X_\Omega) &= \begin{cases}
        1 & k=1\\
        \min\{k, \tfrac{k-2}{2} + ||(1,1)||_\Omega^*\} & k = 2,4,6,...\\
        \min\{k, \tfrac{k-3}{2} + ||(1,2)||_\Omega^*\} & k = 3,5,7,...
    \end{cases}\\
    & = \min\{k, \tfrac{k+3}{2}\}
    \end{aligned}
    \end{equation*}
    Recall that for the polydisk, we have
    $$\Tilde{\mathfrak g}_k^\infty(P(a,1)) = \min\{k,a + \lceil\tfrac{k-1}{2}\rceil\}.$$
    Thus, if $P(a,1) \hookrightarrow X_\Omega$, then
    $$\min\{k,a + \lceil\tfrac{k-1}{2}\rceil\} \leq \min\{k, \tfrac{k+3}{2}\}$$
    for all $k$.
    Thus it must be the case that $a + \lceil\tfrac{k-1}{2}\rceil \leq \tfrac{k+3}{2}$ for all $k \geq 4$.
    This tells us that $a \leq 3/2$.

    Now let's see how $\Tilde{\mathfrak g}_k^\ell$ does for $1 < \ell < \infty$.
    Fix $\ell$ and choose $k$ so that $k - 2\ell + 1 \geq 1$.
    Then by Theorem \ref{thm:comp_for_XOmega_unit_square_in_Omega},
    $$\Tilde{\mathfrak g}_k^\ell(X_\Omega) = k - \ell + 1, \quad\text{and} \quad\Tilde{\mathfrak g}_k^\ell(P(1+\epsilon,1)) = \min\{k,k+1 + \epsilon - \ell\}$$
    Hence,
    $$\Tilde{\mathfrak g}_k^\ell(P(1+\epsilon,1)) > \Tilde{\mathfrak g}_k^\ell(X_\Omega),$$
    giving an obstruction for any $a > 1$.
    Thus, the inclusion is the best embedding we can make.
\end{example}

\begin{figure}
    \centering
    \begin{tikzpicture}[scale=3,cap=round]
    
        \draw[->] (-0.2,0) -- (2.3,0) node[right] {$\pi|z_1|^2$};
        \draw[->] (0,-0.2) -- (0,1.3) node[above] {$\pi|z_2|^2$};
    
        \draw[thick] 
            (0,0) -- (0,1) -- (1,1) -- (1,0) -- cycle;
        \node at (0.5,0.5) {\small $P(1,1)$};
    
        \draw[thick] 
            (0,0) -- (0,1) -- (1,1) -- (2,0.5) -- (2,0) -- cycle;
        \node at (1.4,0.4) {\small $\Omega$};
    
        \draw[dotted] (1,0) -- (1,1);
        \node[below] at (1,0) {\small $1$};
    
        \draw[dotted] (0,1) -- (2.05,1);
        \node[left] at (0,1) {\small $1$};
    
        \draw[dotted] (2,0) -- (2,0.5);
        \node[below] at (2,0) {\small $2$};
        \node[right] at (2,0.5) {\small $\frac{1}{2}$};
    
    \end{tikzpicture}
    \caption{$\Tilde{\mathfrak g}_k^\ell$ provides sharp obstructions when $1 < \ell < \infty$, showing that $P(a,1)$ does not embed into $X_\Omega$ for any $a > 1$.}
    \label{fig:sharp_obstruction}
\end{figure}
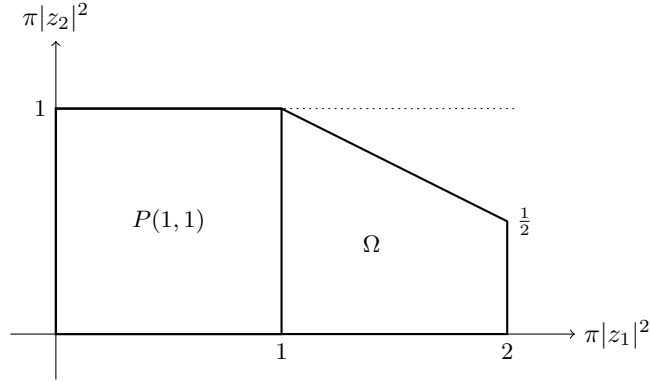

\begin{remark}
    Note that the vertical segment from $(2,0)$ to $(2,1/2)$ is necessary here to get a new obstruction, for if we just had a quadrilateral with an edge from $(1,1)$ to $(2,0)$, then $\Tilde{\mathfrak g}_k^\infty$ would give a sharp obstruction as well.
\end{remark}

\begin{remark}
    The Gutt-Hutchings capacity is known for giving sharp obstructions for the cube capacity when the target is a convex toric domain \cite{Gutt-Hutchings}, but this does not apply here, because we are considering sharpness with respect to only the $a$ factor, rather than scaling up both factors in the cube.
    Indeed, $\Tilde{\mathfrak g}_k^1$ obstructs the embedding of $P(1+\epsilon,1+\epsilon)$ into $X_\Omega$.
\end{remark}

\begin{theorem}\label{thm:sharp_obstructions}
    Let $X_\Omega$ be a convex toric domain such that the polytope boundary $\partial \Omega$ contains the vertices $(0,1)$, $(a,1)$, $(b,\tfrac{1}{2}+\epsilon)$, and $(b,0)$, for $a > 1$ and $b > a + 1/2$.
    Assume further that $a$ is the maximum value for which $(x,1)$ is contained in $\Omega$.
    Then the embedding $P(a',1) \hookrightarrow X_\Omega$ is obstructed for $a' > a$ by $\Tilde{\mathfrak g}_k^\ell$ when $\lceil a\rceil < \ell < \infty$, but not when $\ell = \infty$ or $\ell < a$.
    This obstruction is sharp since the inclusion map works for $a' \leq a$, and this obstruction stabilizes to higher dimensions.
\end{theorem}

\begin{proof}
    No obstruction can come from $\Tilde{\mathfrak g}_k^1$, since $\Tilde{\mathfrak g}_k^1(X_\Omega) = \Tilde{\mathfrak g}_k^1(P(a',1)) = k$ for all $a',k$.
    For the McDuff-Siegel capacity, we consider Theorem \ref{thm:comp_for_XOmega_unit_square_in_Omega} and see that $J = 0$.
    Then, similar to Example \ref{ex:sharp_obstruction}, the obstruction is at best
    $a + \lceil\tfrac{k-1}{2}\rceil \leq \tfrac{k-1}{2} + b$, resulting in
    $$a \leq b - 1/2,$$

    Consider $\ell < a$.
    If $k - 2\ell + 1 < 0$, then $k < 2\ell + 1 < 2a + 1.$
    By Theorem \ref{thm:comp_for_XOmega_unit_square_in_Omega}, the capacity is given by $\min\{k,\tfrac{k-1}{2} + b\}$, for $k$ odd, and $$\tfrac{k-1}{2} + b > \tfrac{k}{2} + a > \tfrac{k}{2} + \tfrac{k-1}{2} = k - 1/2.$$
    Thus, the capacity must be $k$.
    On the other hand, if $k - 2\ell + 1 \geq 0$, then the capacity is given by $\min\{k,\ell-1 + ||(1,k-2\ell+1)||_\Omega^*\}$, and $||(1,k-2\ell+1)||_\Omega \geq \langle(a,1),(1,k-2\ell+1)\rangle = a + k - 2\ell + 1$.
    Then since $k < \ell-1+a+k-2\ell+1 = k+a-\ell$, the capacity must be $k$ in this range too.
    Thus, $\Tilde{\mathfrak g}_k^\ell$ also do not see any obstruction for $\ell < a$.

    Now consider $\ell = \lceil a\rceil + 1$ and $k \gg \ell$.
    Then the capacity for $X_\Omega$ is $\min\{k, \lceil a\rceil + 1 - 1 + ||(1,k-2\ell+1)||_\Omega^*\}$.
    Fully rounding the domain for $x$ values larger than $a$ if necessary, $||(1,k-2\ell+1)||_\Omega^*$ will approach $\langle(a,1),(1,k-2\ell+1)\rangle = a + k - 2\ell + 1$ from above.
    Thus, $\lceil a\rceil + 1 - 1 + a + k - 2\ell + 1 = k + a - \lceil a \rceil - 1.$
    And
    $$\Tilde{\mathfrak g}_k^\ell(P(a+\epsilon,1)) = \min\{k, k+ a+\epsilon - \lceil a \rceil - 1\} = k + a - \lceil a\rceil - 1 + \epsilon,$$
    giving
    $$\Tilde{\mathfrak g}_k^\ell(P(a+\epsilon,1)) > \Tilde{\mathfrak g}_k^\ell(X_\Omega).$$
\end{proof}

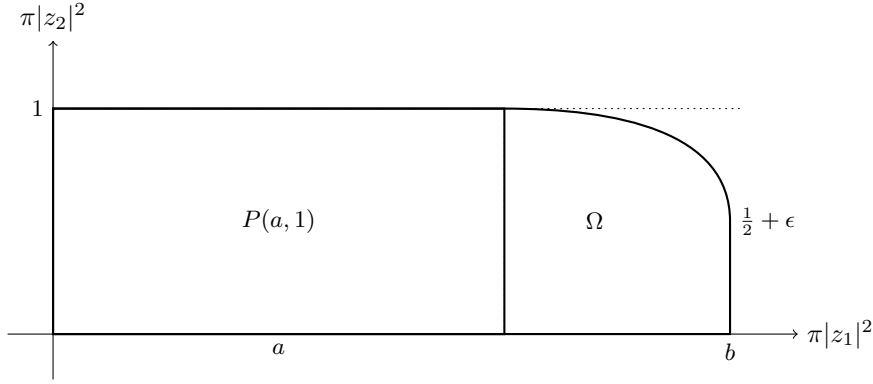
\begin{figure}
    \centering
    \begin{tikzpicture}[scale=3,cap=round]
    
        \draw[->] (-0.2,0) -- (3.3,0) node[right] {$\pi|z_1|^2$};
        \draw[->] (0,-0.2) -- (0,1.3) node[above] {$\pi|z_2|^2$};
    
        \draw[thick] 
            (0,0) -- (0,1) -- (2,1) -- (2,0) -- cycle;
        \node at (1,0.5) {\small $P(a,1)$};
    
        \draw[thick] 
            (0,0) -- (0,1) -- (2,1)
            .. controls (2.5,1) and (3,.9) ..
            (3,0.5) -- (3,0) -- cycle;
        \node at (2.4,0.5) {\small $\Omega$};
    
        \draw[dotted] (2,0) -- (2,1);
        \node[below] at (1,0) {\small $a$};
    
        \draw[dotted] (0,1) -- (3.05,1);
        \node[left] at (0,1) {\small $1$};
    
        \draw[dotted] (3,0) -- (3,0.5);
        \node[below] at (3,0) {\small $b$};
        \node[right] at (3,0.5) {\small $\frac{1}{2}+\epsilon$};
    
    \end{tikzpicture}
    \caption{The general picture for Theorem \ref{thm:sharp_obstructions}.}
    \label{fig:sharp_obstruction_long_domain}
\end{figure}

\begin{remark}
    Note that other capacities like ECH capacities may or may not see the sharp obstruction here in 4 dimensions, but ECH capacities do not stabilize since they allow genus in their moduli spaces.
    Thus, these sharp obstructions appear novel in the stabilized version of these embedding problems.
\end{remark}

\subsection{Ellipsoids}

Here we compute $\Tilde{\mathfrak g}_k^\ell$ in general for the ball, for the ellipsoid $E(2,1)$ for specific ranges of parameters $k,\ell$.
The capacity computation for ellipsoids is given by
$$\Tilde{\mathfrak g}_k^\ell(E(a,1)) = \min_{\{(i_s,j_s)\}_{s=1}^q \in \mathcal{P}_{k,\ell}} \sum_{s=1}^q ||(i_s,j_s)||_\Omega^*,$$
where $\Omega$ is the triangle intersecting the $y$-axis at $1$ and the $x$-axis at $a$.
We can simplify the norm to get
$$\Tilde{\mathfrak g}_k^\ell(E(a,1)) = \min_{\{(i_s,j_s)\}_{s=1}^q \in \mathcal{P}_{k,\ell}} \sum_{s=1}^q \max\{ai_s, j_s\}.$$
And the Gutt-Hutchings capacity gives us the following for $\ell = 1$ \cite[Lemma 2.1]{Gutt-Hutchings}:
$$\Tilde{\mathfrak g}_k^1(E(a,1)) = M_k(a,1),$$
where $M_k(a_1,...,a_n)$ denotes the sequence of positive integer multiples of $a_1,...,a_n$ arranged in nondecreasing order with repetitions.
Note that these values are the norm minimizers for their given index in our terminology.

\begin{example}
    We wish to compute $\Tilde{\mathfrak g}_k^\ell(E(1,1))$ for all $k, \ell$.
    The Gutt-Hutchings capacities are given by
    $$\Tilde{\mathfrak g}_k^1(E(1,1)) = 1,1,2,2,3,3,... = \lceil\tfrac{k}{2}\rceil,$$
    with minimizers
    $$(0,1),(1,1),(1,2),(2,2),(2,3),(3,3),...$$
    We consider what tradeoffs can be made that maintain index and don't increase action.
    By Lemma \ref{lem:upper_bound_rhs_minimizers}, minimizing sets for $\ell > 1$ will only contain these minimizers, and weak permissibility will not be an issue, since we can trade across the axes by symmetry.
    Accordingly, we write tradeoffs in terms of $\Tilde{\mathfrak g}_k^1$.
    \begin{enumerate}[label=(\alph*)]
        \item $\Tilde{\mathfrak g}_k^1 \rightsquigarrow \Tilde{\mathfrak g}_1^1 + \Tilde{\mathfrak g}_{k-2}^1$ for $k > 2$.
        \item $\Tilde{\mathfrak g}_{k_1}^1 + \Tilde{\mathfrak g}_{k_2}^1 \rightsquigarrow \Tilde{\mathfrak g}_{k_1 - 2}^1 + \Tilde{\mathfrak g}_{k_2+2}^1$ for $k_1 > 2$.
        \item $\Tilde{\mathfrak g}_{k_1}^1 + \Tilde{\mathfrak g}_{k_2}^1 \rightsquigarrow \Tilde{\mathfrak g}_{k_1 - 1}^1 + \Tilde{\mathfrak g}_{k_2+1}^1$ for $k_1 > 1, k_2$ odd.
    \end{enumerate}
    Consider the case where $k > 3(\ell - 1)$, and start with some minimizing set.
    Use (a) to trade for more orbits until you reach $\ell$.
    Then use (b) to pick one orbit and send it to as high an index as possible, trading other ends down to $\Tilde{\mathfrak g}_1^1$ or $\Tilde{\mathfrak g}_2^1$.
    Lastly, use (c) to send $\Tilde{\mathfrak g}_1^1$'s to $\Tilde{\mathfrak g}_2^1$.
    We're left with a minimizer of the form
    $$\{(\Tilde{\mathfrak g}_2^1)^{\times \ell - 1}, \Tilde{\mathfrak g}_{k - 3(\ell-1)}^1\},$$
    which has action
    $$\ell - 1 + \left\lceil\frac{k-3(\ell-1)}{2}\right\rceil.$$
    For $k \leq 3(\ell - 1)$, we claim that $\Tilde{\mathfrak g}_k^\ell = \Tilde{\mathfrak g}_k^\infty$.
    Indeed,
    $$\Tilde{\mathfrak g}_k^\infty(E(1,1)) = \begin{cases}
        1+i & k = 3i + 1\\
        1+i & k = 3i + 2\\
        2+i & k = 3i + 3,\\
    \end{cases}$$
    and if $k = 3i + 1$, then take the minimizer to be $\{(1,1)^i,(0,1)\}$, which has $i + 1 = \tfrac{k-1}{3} + 1 \leq \ell - 1 -\tfrac{1}{3} + 1 < \ell$ orbits.
    If $k = 3i + 2$, then take $\{(1,1)^{i+1}\}$, which has $i + 1 = \tfrac{k-2}{3} + 1 \leq \ell - 1 - \tfrac{2}{3} + 1 < \ell$ orbits.
    If $k = 3i + 3$, take $\{(1,1)^i,(0,1)^2\}$ which has $i + 2 = \tfrac{k-3}{3} + 2 \leq \ell - 1 - 1 + 2 = \ell$ orbits.
    The actions of these minimizing sets equal those of $\Tilde{\mathfrak g}_k^\infty$, and the claim follows by the capacity being nonincreasing in $\ell$ (Theorem \ref{thm:properties}(c)).
\end{example}

We summarize this in a Proposition
\begin{proposition}\label{prop:ball}
    The capacity of the ball is given by
    $$\Tilde{\mathfrak g}_k^\ell(B^4(1)) = \begin{cases}
        \ell - 1 + \lceil\tfrac{k-3(\ell - 1)}{2}\rceil & k > 3(\ell - 1)\\
        1 + i & k = 3i+1 \leq 3(\ell - 1)\\
        1 + i & k = 3i+2 \leq 3(\ell - 1)\\
        2 + i & k = 3i+3 \leq 3(\ell - 1)\\
    \end{cases}$$
\end{proposition}

\begin{remark}
    Note that if the index of an orbit set $\{(i_1,j_1),...,(i_q,j_q)\}$ is $k$, then $k = \sum_s (i_s + j_s) + q - 1 \geq 2q - 1$, implying that the number of orbits $q \leq \tfrac{k+1}{2}.$
    Thus we make the following conclusion:
    $$\tfrac{k+1}{2} \leq \ell \quad\Rightarrow\quad \Tilde{\mathfrak g}_k^\ell = \Tilde{\mathfrak g}_k^\infty.$$
\end{remark}

\begin{example}
    One can use similar methods to those above to find $\Tilde{\mathfrak g}_k^\ell(E(2,1))$.
    Consider the following tradeoffs that can be made which maintain index and don't decrease action:
    \begin{enumerate}[label=(\alph*)]
        \item $\Tilde{\mathfrak g}_k^1 \rightsquigarrow \Tilde{\mathfrak g}_2^1 + \Tilde{\mathfrak g}_{k-3}^1$, for $k > 3$.
        \item $\Tilde{\mathfrak g}_{k_1}^1 + \Tilde{\mathfrak g}_{k_2}^1 \rightsquigarrow \Tilde{\mathfrak g}_{k_1-3}^1 + \Tilde{\mathfrak g}_{k_2+3}^1$, for $k_1 > 3$.
        \item $\Tilde{\mathfrak g}_{2}^1 + \Tilde{\mathfrak g}_{k_2}^1 \rightsquigarrow \Tilde{\mathfrak g}_{1}^1 + \Tilde{\mathfrak g}_{k_2+1}^1$.
        \item $\Tilde{\mathfrak g}_{3}^1 \rightsquigarrow (\Tilde{\mathfrak g}_{1}^1)^{\times 2}$ if the trade doesn't violate the $\ell$ bound on number of orbits.
    \end{enumerate}

    Assume that $k > 2\ell - 1$, and consider a minimizing set.
    Use (a) to add as many ends as possible, within the $\ell$ bound.
    Use (b) to move all but potentially one orbit to $\Tilde{\mathfrak g}_k^1$ for $k = 1,2,3$.
    Use (c) to move any $\Tilde{\mathfrak g}_2^1$ orbit to $\Tilde{\mathfrak g}_1^1$, pushing one orbit to higher index.
    Use (d) to move any $\Tilde{\mathfrak g}_3^1$ orbit to $\Tilde{\mathfrak g}_1^1$ if there is still no violation with the $\ell$ bound.
    Then we have a minimizer of the form
    $$\{(\Tilde{\mathfrak g}_1^1)^{\times r_1}, (\Tilde{\mathfrak g}_3^1)^{\times r_2}, \Tilde{\mathfrak g}_{k_3}^1\}$$
    with $r_1 + r_2 = \ell - 1$, half index $r_1 + 3r_2 + k_3 + \ell - 1 = k$, and action $r_1 + 2r_2 + \Tilde{\mathfrak g}_{k_3}^1$
    Then make the following trade as long as $r_1 \geq 3$ and $k_3 > 6$:
    $$(\Tilde{\mathfrak g}_{1}^1)^{\times 3} + \Tilde{\mathfrak g}_{k_3}^1 \rightsquigarrow (\Tilde{\mathfrak g}_{3}^1)^{\times 3} + \Tilde{\mathfrak g}_{k_3-6}^1.$$
    This further reduces our possible minimizing sets.
    We know that for $k > 4(\ell - 1)$, we have enough index to write a minimizer of the form $\{(\Tilde{\mathfrak g}_3^1)^{\times (\ell - 1)},\Tilde{\mathfrak g}_{k-4(\ell - 1)}^1\}$, which has action $2(\ell - 1) + \Tilde{\mathfrak g}_{k-4(\ell - 1)}^1$.
    But it will take more work to see what $\Tilde{\mathfrak g}_k^\ell$ is for $2(\ell - 1) < k \leq 4(\ell - 1)$.
    We summarize this work here:

    $$\Tilde{\mathfrak g}_k^\ell(E(2,1)) = \begin{cases}
        2(\ell - 1) + M_{k - 4(\ell - 1)}(2,1) & k > 4(\ell -1)\\
        ? & 2(\ell - 1) < k \leq 4(\ell - 1)\\
        k & k = 1,2\\
        2 + i & k = 2 + 2i \leq 2(\ell - 1), i\geq 0\\
        2 + i & k = 2 + 2i + 1 \leq 2(\ell - 1), i \geq 0.
    \end{cases}$$
\end{example}

For general $E(a,1)$ with $a \in \mathbb{R}_{\geq 1}$, we can attempt to analyze the optimization problem in a similar way as before, with tradeoffs.
We see that the following tradeoff is beneficial for $i,j \in \mathbb{Z}_{\geq 2}$
$$||(i,j)||_\Omega^* \geq ||(0,1)||_\Omega^* + ||(i-1,j-1)||_\Omega^*.$$
To see this, let $(x,y) \in \partial\Omega$ be the maximizer of $||(i-1,j-1)||_\Omega^*$, which gives
$$||(i,j)||_\Omega^* \geq ix + jy \geq ix + jy + (1-x-y) = ||(0,1)||_\Omega^* + ||(i-1,j-1)||_\Omega^*,$$
since $(x,y) \in \partial\Omega$ implies that $x + y \geq 1$.
Note that $(i-1,j-1)$ may not be a norm minimizer in its index, but we can subsequently make that trade with no issue.
For the case $\ell = \infty$, we can always make this tradeoff for $(0,1)$, thus reducing the possibilities of optimal solutions.
The computation formula for the McDuff-Siegel capacity is closed as a result of this tradeoff \cite[Theorem 1.3.2]{McDuff-Siegel_unperturbed_curves}:
\begin{enumerate}[label=(\alph*)]
    \item For $1 \leq a \leq 3/2$, we have
    $$\Tilde{\mathfrak g}_k^\infty(E(a,1)) = \begin{cases}
        1 + ia & \text{for } k = 1 + 3i \text{ with } i \geq 0\\
        a + ia & \text{for } k = 2 + 3i \text{ with } i \geq 0\\
        2 + ia & \text{for } k = 3 + 3i \text{ with } i \geq 0.
    \end{cases}$$
    \item For $a > 3/2$, we have
    $$\Tilde{\mathfrak g}_k^\infty(E(a,1)) = \begin{cases}
        k & \text{for } 1 \leq k \leq \lfloor a\rfloor\\
        a + i & \text{for } k = \lceil a\rceil + 2i \text{ with } i \geq 0\\
        \lceil a\rceil + i & \text{for } k = \lceil a\rceil + 2i + 1 \text{ with } i \geq 0.
    \end{cases}$$
\end{enumerate}
However, when $\ell < \infty$, we must stop making this tradeoff at some point, and we are still left with the optimization problem.
One would like to say something here about a tradeoff of the following form, where $s < t$ and $(i_r,j_r)$ denotes the norm minimizer for index $r$:
$$\{(i_s,j_s), (i_t,j_t)\} \quad\text{vs.}\quad \{(i_{s-1},j_{s-1}),(i_{t+1},j_{t+1})\}.$$
The issue is that $M_k(a,1)$ as a function of $k$ has irregular step sizes, so sometimes this trade should go left to right, and other times it should go right to left.
The irregularity of optimal solutions is best seen by an example.

\begin{example}
    Consider $E(\varphi,1)$ for the golden ratio $\varphi \approx 1.618$, chosen since it's an irrational number in between $1$ and $2$.
    We can look at the capacity when $\ell = 3$ and see that patterns are tough to discern after a certain point:

\begin{center}
\begin{tabular}{ccc|cccc|ccc}
\toprule
$k$ & $\Tilde{\mathfrak g}_k^3$ &&& $k$ & $\Tilde{\mathfrak g}_k^3$ &&& $k$ & $\Tilde{\mathfrak g}_k^3$ \\
\midrule
1 & $1$ &&& 11 & $6$ &&& 21 & $12$ \\
2 & $\varphi$ &&& 12 & $3\varphi+2$ &&& 22 & $8\varphi$ \\
3 & $2$ &&& 13 & $2\varphi+4$ &&& 23 & $2\varphi+10$ \\
4 & $\varphi+1$ &&& 14 & $8$ &&& 24 & $14$ \\
5 & $3$ &&& 15 & $4\varphi+2$ &&& 25 & $9\varphi$ \\
6 & $\varphi+3$ &&& 16 & $9$ &&& 26 & $15$ \\
7 & $4$ &&& 17 & $6\varphi$ &&& 27 & $16$ \\
8 & $\varphi + 4$ &&& 18 & $2\varphi+7$ &&& 28 & $7\varphi+5$ \\
9 & $5$ &&& 19 & $11$ &&& 29 & $17$ \\
10 & $\varphi + 5$ &&& 20 & $4\varphi+5$ &&& 30 & $11\varphi$ \\
\bottomrule
\end{tabular}
\end{center}
    
\end{example}

\section{Future Work}\label{sec:future_work}
There are still many computations to be done with $\Tilde{\mathfrak g}_k^\ell$: ellipsoids $E(a,1)$ for $a > 1$, and \textbf{short domains} that do not contain the cube.
One would also like to compute ECH capacities for long domains to see if any of the obstructions here are novel in four dimensions.
One could study the same questions for concave, weakly convex, and monotone toric domains (see \cite{gutt-hutchings-ramos2022examples} for the introduction of monotone toric domains).
One could ask what these capacities look like in general in higher dimensions rather than just in the stabilized setting.
Studying the asymptotics of these capacities is a natural follow up too and has applications to dynamics.
The \textit{reverse engineering} question is also interesting: given an array of values, does there exist a domain whose capacities agree with those values (this question is inspired by Theorem 1.3 in \cite{blind_spots} and the recognition question from \cite{cieliebak2007quantitative}).

We could extend the family of capacities further to $\Tilde{\mathfrak g}_k^{\ell, h}$ where $h$ is a restriction on the genus of curves.
A naive computation formula for this capacity is
$$\Tilde{\mathfrak g}_k^{\ell,h}(X_\Omega) = \min_{\{(i_s,j_s)\}_{s=1}^q \in \mathcal{P}_{k,\ell,h}} \sum_{s=1}^q ||(i_s,j_s)||_\Omega^*,$$
where $\{(i_s,j_s)\}_{s=1}^q \in \mathcal{P}_{k,\ell,h}$ if there exists an integer $g \leq h$ such that:
\begin{itemize}
    \item $\frac{1}{2}$\textit{index:} $\sum_{s=1}^q(i_s+j_s) + q + g - 1 = k$.
    \item \textit{weak permissibility:} if $q \geq 2$, then $(i_1,...,i_q) \neq (0,...,0)$ and $(j_1,...,j_q) \neq (0,...,0)$.
    \item \textit{positive ends:} $q \leq \ell$.
\end{itemize}
The only novelty here is the presence of $g$ in the index criteria.
It's uncertain if this is correct: one must check that the restraint coming from the relative adjunction formula is seen by these criteria.
After that, one must check that indeed all curves relevant to these capacities exist as long as they satisfy the index and relative adjunction formulas.
This is conjectural, as sometimes curves that would be permitted by index and adjunction in fact do not exist.

One could try to prove this formula, or assume it and do a purely computational project.
Such a project could be a good testing ground for reasoning language models and algorithmic discovery.
One could explore Google DeepMind's AlphaEvolve \cite{novikov2025alphaevolve} and try to search domains efficiently, creating sharper bounds in symplectic embedding questions; the agent has shown success in other geometric packing problems and could hopefully be adapted to this setting.

\printbibliography

\end{document}